\documentclass[11pt]{amsart}

\usepackage{upgreek}

\usepackage{ytableau}

\usepackage{amsmath, amsxtra, amsthm, amssymb,mathtools,bm,amsfonts}
\usepackage{mathrsfs}
\usepackage[normalem]{ulem}
\usepackage[mathscr]{euscript}
\usepackage{graphicx}
\usepackage{url}
\usepackage{color}
\usepackage{bbm}
\usepackage{tikz-cd}
\usepackage{tikz}

\usepackage[T1]{fontenc}
\graphicspath{/Figures/}

\usepackage{enumitem}
\usepackage{comment}
\usepackage[pdftex,hidelinks,backref=page]{hyperref}
\hypersetup{
    colorlinks,
    citecolor=magenta,
    filecolor=magenta,
    linkcolor=blue,
    urlcolor=black
}
\usepackage{cleveref}

\usepackage{xcolor}
\usepackage[colorinlistoftodos]{todonotes}
\renewcommand{\emptyset}{\varnothing}
\newcommand{\NN}{\mathbb N}
\newcommand{\QQ}{\mathbb Q}
\newcommand{\RR}{\mathbb R}
\newcommand{\CC}{\mathbb C}
\newcommand{\PP}{\mathbb P}
\newcommand{\ZZ}{\mathbb Z}
\newcommand{\GG}{\mathbb G}

\newcommand{\id}{\mathrm{id}}

\newcommand{\Face}{\mathrm{Face}}
\renewcommand{\Vert}{\mathrm{Vert}}

\theoremstyle{definition}
\newtheorem{thm}{Theorem}[section]
\newtheorem{cor}[thm]{Corollary}
\newtheorem{lem}[thm]{Lemma}

\newtheorem{prop}[thm]{Proposition}
\newtheorem{defn}[thm]{Definition}

\newtheorem{conj}[thm]{Conjecture}
\newtheorem{eg}[thm]{Example}
\newtheorem{rem}[thm]{Remark}
\newtheorem{obs}[thm]{Observation}

\newtheorem{maintheorem}{Theorem}	\newtheorem{fact}[thm]{Fact}

\numberwithin{equation}{section}

\newcommand{\ar}[1]
{{\xrightarrow{#1}}}

\newcommand{\indexedforests}{\operatorname{\mathsf{Forest}}}

\newcommand{\forestpoly}[1]{\mathfrak{P}_{#1}} 

\newcommand{\qsym}[2][]{
{\ifx&#1&%
  {\operatorname{QSym}_{#2}}
\else
  {{}^{#1}\!\operatorname{QSym}_{#2}}
\fi}
} 

\newcommand{\eqsym}[2][]{
{\ifx&#1&%
  {\operatorname{EQSym}_{#2}}
\else
  {{}^{#1}\!\operatorname{EQSym}_{#2}}
\fi}
} 

\newcommand{\qseq}[2][]{
{\ifx&#1&%
  {\operatorname{QSeq}_{#2}}
\else
  {{}^{#1}\!\operatorname{QSeq}_{#2}}
\fi}
}

\newcommand{\qsymide}[2][]{
{\ifx&#1&%
  {\operatorname{QSym}_{#2}^+}
\else
  {{}^{#1}\!\operatorname{QSym}_{#2}^+}
\fi}
} 

\newcommand{\eqsymide}[2][]{
{\ifx&#1&%
  {\operatorname{EQSym}_{#2}^+}
\else
  {{}^{#1}\!\operatorname{EQSym}_{#2}^+}
\fi}
} 

\newcommand{\sym}[1]{\operatorname{Sym}_{#1}} 
\newcommand{\esym}[1]{\operatorname{ESym}_{#1}} 
\newcommand{\symide}[1]{\sym{#1}^+} 
\newcommand{\esymide}[1]{\esym{#1}^+} 
\newcommand{\supp}{\operatorname{supp}} 
\newcommand{\compatible}[2][]{
{\ifx&#1&%
  {\mathcal{C}(#2)}
\else
  {\mathcal{C}^{m}(#2)}
\fi}
} 
\newcommand{\internal}[1]{\operatorname{IN}(#1)} 
\newcommand{\suchthat}{\;|\;}

\newcommand{\schub}[1]{\mathfrak{S}_{#1}} 

\newcommand{\ncperm}{\operatorname{NCPerm}}
\date{}
\newcommand{\cat}[1]{\operatorname{Cat}_{#1}} 

\newcommand{\idem}{\operatorname{id}} 
\newcommand{\slide}[2][]{
{\ifx&#1&%
  {\mathfrak{F}_{#2}}
\else
  {\mathfrak{F}_{#2}^{\underline{#1}}}
\fi}
} 
\newcommand{\tope}[2][]{
{\ifx&#1&%
  {\mathsf{T}_{#2}}
\else
  {\mathsf{T}_{#2}^{\underline{#1}}}
\fi}
} 
\newcommand{\rope}[1]{\mathsf{R}_{#1}} 

\newcommand{\zigzag}[2][]
{
{\ifx&#1&%
  {\mathsf{ZigZag}_{#2}}
\else
  {\mathsf{ZigZag}_{#2}^{#1}}
\fi}
}

\newcommand{\ltfor}[2][] 
{
{\ifx&#1&%
  {\mathsf{LTFor}_{#2}}
\else
  {\mathsf{LTFor}_{#2}^{#1}}
\fi}
}

\newcommand{\rtfor}[2][] 
{
{\ifx&#1&%
  {\mathsf{RTFor}_{>#2}}
\else
  {\mathsf{RTFor}_{>#2}^{#1}}
\fi}
}

\newcommand{\suppfor}[2][] 
{
{\ifx&#1&%
  {\mathsf{Forest}_{#2}}
\else
  {\mathsf{For}_{#2}^{#1}}
\fi}
}

\newcommand{\binfor}[1][]{
{\ifx&#1&%
    \operatorname{\mathsf{BinFor}}
\else
    \operatorname{\mathsf{BinFor}}^{#1}
\fi}
} 

\newcommand{\hqsym}[2][]{
{\ifx&#1&%
  {\operatorname{HQSym}_{#2}}
\else
  {\operatorname{HQSym}_{#2}^{#1}}
\fi}
} 
\newcommand{\fl}[1]{\mathrm{Fl}_{#1}}
\newcommand{\GL}{\mathrm{GL}}
\newcommand{\Pl}{\mathrm{Pl}}
\newcommand{\coinv}[1]{\operatorname{Coinv}_{#1}} 

\newcommand{\qscoinv}[2][]{
{\ifx&#1&%
  {\operatorname{QSCoinv}_{#2}}
\else
  {{}^{#1}\!\operatorname{QSCoinv}_{#2}}
\fi}
}

\newcommand{\eqscoinv}[2][]{
{\ifx&#1&%
  {\operatorname{EQSCoinv}_{#2}}
\else
  {{}^{#1}\!\operatorname{EQSCoinv}_{#2}}
\fi}
}

\definecolor{ao}{rgb}{0.0, 0.5, 0.0}

\newcommand{\rt}{\Omega}

\newcommand{\rtseq}{\mathrm{RTSeq}}
\newcommand{\reseq}{\mathrm{RESeq}}

\newcommand{\qfl}{\mathrm{QFl}} 
\newcommand{\cqfl}[1]{\mathrm{Complex}(\qfl_{#1})} 

\newcommand{\nfor}{\operatorname{NestFor}} 
\newcommand{\bnfor}{\operatorname{BNestFor}} 
\newcommand{\bnfornf}{\bnfor^{\operatorname{nf}}} 

\newcommand{\wh}[1]{\widehat{#1}} 
\newcommand{\rletter}[1]{\mathsf{r}_{#1}}
\newcommand{\tletter}[1]{\mathsf{t}_{#1}}
\newcommand{\eletter}[1]{\mathsf{e}_{#1}}

\newcommand{\xletter}[1]{\mathsf{x}_{#1}}

\newcommand{\tl}{\textbf{t}}
\newcommand{\xl}{\textbf{x}}

\newcommand{\ev}{\operatorname{ev}}

\newcommand{\NC}{\operatorname{NC}}
\newcommand{\ForToNC}{\operatorname{ForToNC}}

\newcommand{\eope}[1]{{\mathsf{E}_{#1}}}

\newcommand{\mc}[1]{\mathcal{#1}}
\newcommand{\inv}[1]{\operatorname{Inv}(#1)}
\newcommand{\invnc}[1]{\operatorname{Inv}_{\NC}(#1)} 

\renewcommand\emph[1]{\textcolor{blue}{\textit{#1}}} 

\newcommand{\wnode}
{
\begin{tikzpicture}[baseline=-3,scale=1]
		\node at (0,0) {{$\wedge$}};
		\draw[fill=white] (0,.1) circle (.05);
\end{tikzpicture}
}
\newcommand{\bnode}
{
\begin{tikzpicture}[baseline=-3,scale=1]
		\node at (0,0) {{$\wedge$}};
		\draw[fill=black] (0,.1) circle (.05);
\end{tikzpicture}
}

\newcommand{\cox}{\bm{c}} 
\newcommand{\point}{\operatorname{pt}} 

\newcommand{\cay}[1]{\operatorname{Cayley}(#1)}

\newcommand{\tree}{\mathrm{Tree}}

\newcommand{\proj}{\mathsf{Proj}} 
\newcommand{\ctam}{\overset{c}{\sim}}
\newcommand{\degint}{\int}
\newcommand{\cone}[1]{\operatorname{Cone}_{#1}}

\newcommand{\LTFor}{\mathrm{LTForest}}
\newcommand{\LT}{\mathrm{LTer}}

\title[The quasisymmetric flag variety]{The quasisymmetric flag variety: a toric complex on noncrossing partitions}

\author{Nantel Bergeron}
\address{Dept. of Math. and Stat., York University, Toronto, Ontario M3J 1P3, Canada}
\email{\href{mailto:bergeron@yorku.ca}{bergeron@yorku.ca}}

\author{Lucas Gagnon}
\address{Dept. of Math. and Stat., York University, Toronto, Ontario M3J 1P3, Canada}
\email{\href{mailto:lgagnon@yorku.ca}{lgagnon@yorku.ca}}

\author{Philippe Nadeau}
\address{Universite Claude Bernard Lyon 1, CNRS, Ecole Centrale de Lyon, INSA Lyon, Université Jean Monnet, ICJ UMR5208, 69622 Villeurbanne, France}
\email{\href{mailto:nadeau@math.univ-lyon1.fr}{nadeau@math.univ-lyon1.fr}}

\author{Hunter Spink}
\address{Department of Mathematics,
University of Toronto, Toronto, ON M5S 2E4, Canada}
\email{\href{mailto:hunter.spink@utoronto.ca}{hunter.spink@utoronto.ca}}

\author{Vasu Tewari}
\address{Department of Mathematical and Computational Sciences, University of Toronto Mississauga, Mississauga, ON L5L 1C6, Canada}
\email{\href{mailto:vasu.tewari@utoronto.ca}{vasu.tewari@utoronto.ca}}

\thanks{
NB and LG were supported by the Natural Sciences and Engineering Research Council of Canada (NSERC) and York Research Chair in Applied Algebra.
PN was partially supported by French ANR grant ANR-19-CE48-0011 (COMBIN\'E). HS and VT acknowledge the support of the NSERC, respectively [RGPIN-2024-04181] and [RGPIN-2024-05433].}

\begin{document}

\begin{abstract}
We develop the geometric theory of equivariant quasisymmetry via a new ``quasisymmetric flag variety''. This is a toric complex in the flag variety whose fixed point set is the set of noncrossing partitions, and whose cohomology ring is the ring of quasisymmetric coinvariants. 
\end{abstract}

\maketitle

\section{Introduction}
In this paper we study a new geometric object we call the \emph{quasisymmetric flag variety} 
$\qfl_n$, which is
contained in the variety $\fl{n}$ of complete flags of subspaces $0=\mc{F}_0\subsetneq \mc{F}_1\subsetneq\cdots \subsetneq \mc{F}_n=\mathbb{C}^n$.
We show that $\qfl_n$ plays the same role for the quasisymmetric polynomials of Gessel \cite{Ges84} and Stanley \cite{StThesis} that $\fl{n}$ plays for symmetric polynomials. 
Our main results, summarized below, provide a complete geometric model for the quasisymmetric coinvariants.

Recall that a quasisymmetric polynomial in the variable set $\xl_n=\{x_1,\dots,x_n\}$ is one where the coefficient of $x_1^{a_1}\cdots x_k^{a_k}$ equals that of $x_{i_1}^{a_1}\cdots x_{i_k}^{a_k}$ for all increasing sequences $1\le i_1<\cdots<i_k\le n$ and all sequences $(a_1,\dots,a_k)$ of positive integers. 
We denote by $\qsym{n}$ the ring of all quasisymmetric polynomials, and note that since the defining condition for quasisymmetry is a weakening of symmetry, the ring $\sym{n}$ of symmetric polynomials  in $\xl_n$ is a subring of $\qsym{n}$.

In~\cite{ABB04}, Aval--Bergeron--Bergeron initiated the study of the \emph{quasisymmetric coinvariant ring}
\[
\qscoinv{n}\coloneqq \ZZ[\xl_n]/\qsymide{n}
\quad
\text{where}
\quad
\qsymide{n}\coloneqq \langle f(\xl_n)-f(0,\ldots,0)\suchthat f\in \qsym{n}\rangle.
\]
The graded ring $\coinv{n}\coloneqq \ZZ[\xl_n]/\symide{n}$ has a unimodal symmetric sequence of ranks, which reflects the fact that $H^\bullet(\fl{n})\simeq\coinv{n}$ and $\fl{n}$ is a smooth projective variety. 
In contrast, the  graded space  $\qscoinv{n}$  is not rank symmetric for any $n\geq 3$ \cite[Theorem 1.1]{ABB04}, and thus it cannot arise as the cohomology of any smooth projective variety.

We construct the quasisymmetric flag variety using a collection of $(n-1)$-dimensional smooth toric varieties $X(T)$ parameterized by the set $\tree_{n}$ of $n$-leaf planar binary trees, 
\begin{align*}
\qfl_n\coloneqq \bigcup_{T\in \operatorname{Tree}_n}X(T)\subset \fl{n}.
\end{align*}
Each variety $X(T)$ is a left-translated Richardson variety, and an iterated $\mathbb{P}^1$-bundle following a process determined by the combinatorial structure of the tree $T$. 
The moment polytopes of these varieties are both combinatorial cubes and \emph{polypositroids} in the sense of Lam--Postnikov \cite{LP20}.

\begin{maintheorem}\label{main:A}
We have $H^\bullet(\qfl_n)\cong\qscoinv{n}$.
\end{maintheorem}

We note that this does not contradict the previous observation on smooth projective varieties as $\qfl_n$ is reducible.
We give a similar presentation of the torus-equivariant cohomology ring via the coinvariant ring of the \emph{equivariantly quasisymmetric polynomials} defined in~\cite{BGNST1}; see~\Cref{sec:qfl_brl}.

The combinatorics of $\qfl_n$ is governed by \emph{noncrossing partitions}.  
Let $S_{n}$ denote the symmetric group on $n$ letters. 
The torus fixed points of $\qfl_{n}$ are exactly the permutations $\NC_{n} \subseteq S_{n}$ obtained by treating each block of a noncrossing set partition as a backwards cycle, see Biane~\cite{Bi97}.  
Noncrossing partitions also yield the following intrinsic characterization of $\qfl_{n}$, which we prove in \Cref{sec:Intrinsic}.

\begin{maintheorem}
\label{mainthm:plucker}
Denoting by $[\Pl_{\sigma}]_{\sigma\in S_n}$ the \emph{Pl\"ucker functions} on $\fl{n}$, we have that
\[
\qfl_n=\bigcap_{\sigma\in S_n\setminus \NC_n}\{\Pl_{\sigma}=0\}\subset \fl{n}.
\]
\end{maintheorem}

Our remaining results establish the following parallels between $\qfl_{n}$ and $\fl{n}$. 
\begin{enumerate}[label=(\arabic*)]
\item \label{it1:intro} The Bruhat decomposition gives a stratification of the flag variety into a union of affine Schubert cells $(X^w)^{\circ}$ with well-behaved closure relations known as an affine paving.  
In \Cref{sec:affinepaving}, we describe an affine paving 
\[
\qfl_n = \bigsqcup_{w \in \NC_{n}} (X^{w})^{\circ} \cap \qfl_n,
\]
which shares many combinatorial properties with the Bruhat decomposition.  In \Cref{sec:CoinvariantsGKM} we prove that the closures of our affine cells give a homology basis for $H_\bullet(\qfl_n)$. 

\item \label{it2:intro}
Schubert polynomials \cite{LS82} give a basis of $H^{\bullet}(\fl{n})$
in Borel's presentation, and this basis is Kronecker dual to the homology basis of Schubert cycles $[X^w]$.  
In \Cref{sec:equivbuildingopscohom} we prove that our homology basis is likewise Kronecker dual to the family of forest polynomials~\cite{NT_forest,NST_a} and double forest polynomials~\cite{BGNST1}.

\item \label{it3:intro}
The divided difference formalism for Schubert polynomials admits a geometric interpretation via Bott--Samelson resolutions of Schubert varieties.  
Forest polynomials satisfy similar formalisms, and we show in \Cref{sec:equivbuildingopscohom} that these operations can be similarly interpreted using iterated $\mathbb{P}^1$-bundles. 
We exploit this connection to compute the degree map of each toric variety $X(F)$ in $\qfl_{n}$ using a recursive combinatorial process first defined in~\cite{BGNST1}.

\item \label{it4:intro}
Torus-equivariant versions of~\ref{it1:intro}--\ref{it3:intro} also hold, highlighting parallels between the classical double Schubert polynomials and the double forest polynomials defined in~\cite{BGNST1}.   

\end{enumerate}

In fact, many of our ``non-equivariant'' results are derived by first proving their equivariant analogue. 
In doing so, we rely crucially on the combinatorial relationship between double forest polynomials and noncrossing partitions  developed in~\cite{BGNST1}. 
In particular, we show that the Goresky--Kottwitz--MacPherson graph \cite{GKM98} associated to $\qfl_n$ under the standard torus action is the Kreweras lattice on $\NC_{n}$, and specializations of double forest polynomials describe a free basis of the associated graph cohomology ring. 
This generalizes earlier work of the first and second author~\cite{BeGa23} on \emph{quasisymmetric orbit harmonics}, in which the main result shows that the ring
\[
\QQ[\xl_n]/\langle f(\xl_n)\suchthat f(w(1),\ldots,w(n))=0\text{ for all }w\in \NC_n\rangle
\]
has associated graded $\qscoinv{n}$.
\smallskip

We now give a brief outline of the paper. 
In Section~\ref{sec:preliminaries} we recall standard facts about $\fl{n}$, noncrossing partitions and binary trees.  
We then define a family of elementary ``building operations'' on $\fl{n}$ in Section~\ref{sec:BuildingOperators} and study the combinatorics of their compositions in Section~\ref{sec:RelationsonBuilding}. 
The building operations are the backbone of our work, and in Section~\ref{sec:BottManifold} we use them to construct a family of Bott manifolds $X(\wh{F})$ attached to bicolored nested forests $\wh{F}$ that includes the $X(T)$ as the top-dimensional case. 
We provide a precise description of the inclusion order of these varieties in Section~\ref{sec:torusfixedpoints}, highlighting the role of noncrossing partitions, and connect our work to $\operatorname{HHMP}_{n}$ and Richardson varieties in  Section~\ref{sec:richardsons_and_polypositroids}.
In Section~\ref{sec:qsymvar} we define $\qfl_{n}$ and completely describe the structure of this complex.  
The final four sections (\ref{sec:affinepaving}, \ref{sec:Intrinsic}, \ref{sec:CoinvariantsGKM}, and \ref{sec:equivbuildingopscohom}) contain our main results, as described above. 

\smallskip

In the remainder of this introduction, we explain the motivation behind the construction of the varieties $X(T)$.   
In \cite{NST_c} a partial solution to the geometric realization of $\qscoinv{n}$ was obtained via a different construction called the \emph{$\rt$-flag variety}.  
Let $S_{n-1}\subset S_n$ comprise permutations $w$ satisfying $w(n)=n$ and denote the backwards long cycle $\cox=(n\, n-1\,\cdots \,1) \in S_n$.
In~\cite{NST_c}, the third, fourth, and fifth authors consider the  complex of $(n-1)!$ smooth toric Richardson varieties
\begin{align*}\operatorname{HHMP}_n\coloneqq \bigcup_{w\in S_{n-1}}X^{w\cox}_w
\end{align*}
under the torus action of $T_n=(\mathbb{C}^*)^n$ acting on $\fl{n}$ via its action on $\mathbb{C}^n$. 
This complex was first studied in \cite{HHMP,lian2023hhmp} as arising from a toric degeneration of a general $T_n$-orbit closure in $\fl{n}$. 
We emphasize that this complex differs from $\qfl_n$, as for $n\ge 3$ it contains more Richardson varieties than there are $X(T)$ varieties.
Each toric Richardson in this complex is smooth and has a moment polytope that is a combinatorial cube. 

The toric complex $\operatorname{HHMP}_n$ is assembled in a combinatorially simple way, equivalent to the decomposition of $[0, 1] \times [0, 2] \times \cdots \times [0, n-1]$ into unit cubes; see \cite[Section 7.1]{NST_c}. 
One of the main results of \cite{NST_c} was that for the inclusion $\phi\colon \operatorname{HHMP}_n\to \fl{n}$, we have
\[
\phi^*H^\bullet(\fl{n})\cong \qscoinv{n}\subset H^\bullet(\operatorname{HHMP}_n).
\]
Here the inclusion is strict: multiple cycles $X^{w\cox}_w$ are equivalent under the left action of $S_n$ by permutation matrices. 

The geometry of $\qfl_n$ is constructed to avoid this duplication problem. We show (see also \cite[\S 6]{NST_c}) that the $X(T)$ varieties collapse the left-translation redundancy of the $X^{w\cox}_w$ via a surjective map
\[
\begin{array}{rcl}
\{X^{w\cox}_w\suchthat w\in S_{n-1}\}&\to& \{X(T)\suchthat T\in \tree_n\}\\
X^{w\cox}_w&\mapsto& w^{-1}X^{w\cox}_w
\end{array}
\]
taking the $(n-1)!$-many Richardson varieties to the $\cat{n-1}=\frac{1}{n}\binom{2n-2}{n-1}$-many $X(T)$ varieties. 

The trade-off in the construction of $\qfl_{n}$ is a more complicated toric complex. 
In particular, the moment polytopes of our translated Richardson varieties overlap in an irregular manner, and torus orbits associated to partially overlapping faces do not intersect. 
These phenomena first appears in the case $n = 3$, and \Cref{fig:hhmp_original_and_translated.pdf} contrasts the structure of the moment polytopes for $\operatorname{HHMP}_{n}$ with our ``translated'' moment polytopes, which reflects the structure of $\qfl_{n}$. Geometrically  $\operatorname{HHMP}_3$ is two Hirzebruch surfaces glued along a common torus-invariant $\mathbb{P}^1$, while $\qfl_3$ is two Hirzebruch surfaces glued along two different and intersecting torus-invariant $\mathbb{P}^1$'s. 
We revisit this figure in Section~\ref{sec:qsymvar} where the right panel is redrawn to resemble the Kreweras lattice $\NC_3$.

\begin{figure}[!htb]
    \centering
    \includegraphics[scale=1.1]{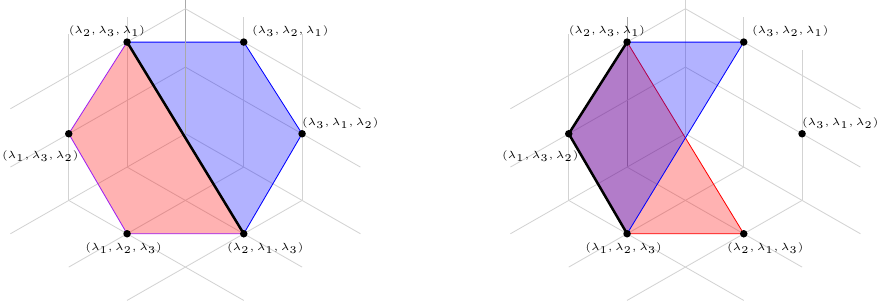}
    \caption{The two trapezoids comprising the HHMP subdivision of the $n=3$ permutahedron (left) and the intersecting $X(T)$ moment polytopes (right).}
    \label{fig:hhmp_original_and_translated.pdf}
\end{figure}

\subsection*{Acknowledgements}
We are very grateful to Allen Knutson and Alex Fink for numerous helpful conversations and their generous sharing of ideas.
VT owes a lot to enlightening conversations with Dave Anderson. 
Finally thanks also go out to Sara Billey, Alexander Woo,  and Alejandro Morales for pointers to relevant results in literature.

\section{Preliminaries}
\label{sec:preliminaries}

Let $s_1,\ldots,s_{n-1}$ be the simple transpositions $s_i=(i,i+1)$ generating the symmetric group $S_n$. We let
\[
\cox \coloneqq  (n\,n-1\,\cdots\,2\,1)=s_{n-1}s_{n-2}\cdots s_1 \in S_{n},
\] 
denote the backwards long cycle.  For all nonnegative integers $m$ we set $[m]\coloneqq \{1,\dots,m\}$.

\subsection{Recollections on $\fl{n}$}

We work over $\CC$ and denote by $T_{n},B_{n},B_{n}^-\subset \GL_{n}$ the subsets of diagonal, upper triangular, and lower triangular invertible $n\times n$ matrices.  When there is no ambiguity we write $T$, $B$, and $B^-$ for $T_{n}$, $B_{n}$, and $B^{-}_{n}$. We will denote $\chi_i$ for the $i$th standard character of $T_n$, corresponding to the $i$'th entry along the diagonal, so that for $a = \mathrm{diag}(a_{1}, \ldots, a_{n}) \in T$ we have $\chi_{i}(a) = a_{i}$.

We identify the complete flag variety 
\[
\fl{n} \coloneqq \text{flags of subspaces }\big(\{0\}=\mc{F}_0\subsetneq \mc{F}_1\subsetneq \mc{F}_2\subsetneq \cdots \subsetneq \mc{F}_{n-1}\subsetneq \mc{F}_{n}=\CC^{n}\big).
\]
with $\GL_{n}/B_{n}$ via the transitive action of $\GL_{n}$ with $B_{n}$ the stabilizer of the \emph{standard coordinate flag} with $\mc{F}_i=\langle e_1,\ldots,e_i\rangle$ for $1\leq i\leq n$.
Via this identification $ith$ subspace in the flag associated to $hB_{n}$ is the column span of the first $i$ columns of $h$ for $1\leq i\leq n$.

For $w \in S_{n}$, we will denote the Schubert cycles in $\fl{n}$ by $X^v=\overline{BvB}$, the opposite Schubert cycles by $X_u=\overline{B^-uB}$, and for $u\le v$ in the Bruhat order the Richardson varieties $X^v_u\coloneqq X^v\cap X_u=\overline{BvB}\cap \overline{B^-uB}$.
The Bruhat decomposition
\begin{align}
\label{eq:bruhat_decomposition}
\fl{n}=\bigsqcup_{w\in S_n}BwB
\qquad
\text{ with }
\qquad 
BwB\cong \mathbb{A}^{\ell(w)},
\end{align}
 gives an affine paving if the $BwB$ are ordered via any linear extension of the  Bruhat order on $S_n$; see Section~\ref{sec:affinepaving} for more details. 

For $w \in S_{n}$ and a dominant weight $\lambda=(\lambda_1\ge \cdots \ge \lambda_n)\in \mathbb{Z}^n$ we have a \emph{Pl\"{u}cker function} defined for $h\in \GL_n$ by \[
\Pl_{\lambda, w}(h) = (\det h_{w(1)})^{\lambda_1-\lambda_2}(\det h_{w(1),w(2)})^{\lambda_2-\lambda_3}\cdots (\det h_{w(1),\ldots,w(n)})^{\lambda_n},
\]
where $h_{i_1,\ldots,i_k}$ is the submatrix of $h$ with columns $1,\ldots,k$ and rows $i_1,\ldots,i_k$.
These functions together define a map $\Pl_{\lambda}:\fl{n}\to \mathbb{P}^{n!-1}$, which we write simply $\Pl$ if $\lambda$ is the fundamental dominant weight $(n,n-1,\ldots,1)$. 
Then $\Pl_{\lambda}$ is $T_n$-equivariant with respect to the $T_n$-action on $\mathbb{P}^{n!-1}$ where $\mathrm{diag}(a_1,\ldots,a_n)\in T_n$ acts in the $w$-coordinate by $a_{w(1)}^{\lambda_1}a_{w(2)}^{\lambda_2}\cdots a_{w(n)}^{\lambda_n}=a_1^{\lambda_{w^{-1}(1)}}\cdots a_n^{\lambda_{w^{-1}(n)}}$.  

 The moment polytope of $\fl{n}$ under $\Pl_{\lambda}$ is the generalized permutahedron 
\[
\operatorname{Perm}(\lambda) =\operatorname{conv}\{u\cdot {\lambda}\suchthat u\in S_n\} \subseteq \RR^{n},
\]
where $u\cdot {\lambda} = (\lambda_{u^{-1}(1)},\ldots,\lambda_{u^{-1}(n)})$ and $\operatorname{conv}$ denotes taking the convex hull. 
If $X\subset \fl{n}$ is a $T_n$-invariant subvariety, then its moment polytope under $\Pl_{\lambda}$ is  
\[
P(X;{\lambda})=\operatorname{conv}\{u\cdot {\lambda} \suchthat u\in X^T\} \subseteq \operatorname{Perm}(\lambda).
\]
The moment polytope of a $T_{n}$-orbit closure $X$ is always a \emph{flag matroid polytope} (see \cite{GGMS87}), meaning its vertices are contained in the vertices of $\operatorname{Perm}(\lambda)$ and all edges are parallel to the type $A_{n-1}$ roots $e_i-e_j$ for various $i,j$.

\subsection{Noncrossing partitions} 
\label{sec:ncp}

A \emph{combinatorial noncrossing partition} is a partition  $A_1\sqcup \cdots \sqcup A_k=[n]$ such that for $i\ne j$, distinct elements $a,b\in A_i$, and distinct elements $c,d\in A_j$, we never have $a<c<b<d$. We depict combinatorial noncrossing partitions as noncrossing arc diagram; for example we draw $\{1,2,3,6\}\sqcup \{4,5\}=[6]$ as
\[
\begin{tikzpicture}[scale = 0.75, baseline = 0.75*-0.2]
\foreach \x in {1, ..., 6}{\draw[fill] (\x - 1, 0) node[inner sep = 2pt] (\x) {$\scriptstyle \x$};}
\foreach \i\j in {1/2, 2/3, 3/6, 4/5}{\draw[thick] (\i) to[out = 35, in = 145] (\j);}
\end{tikzpicture}.
\]

A \emph{backwards cycle} is a cycle $(b_1\,b_2\,\cdots\, b_r)$ with $b_1>b_2>\cdots >b_r$. An \emph{algebraic noncrossing partition} is defined to be a permutation $w$ whose disjoint cycle decomposition $\operatorname{Cyc}(w)\coloneqq  C_{1}C_{2} \cdots C_{k}$ consists of backwards cycles whose underlying sets define a combinatorial noncrossing partition.  For example $w = (6321)(54)$ is the algebraic noncrossing partition associated to the arc diagram above. We denote
\[
\NC_{n} \coloneqq \{\text{algebraic noncrossing partitions of $[n]$}\}\subset S_n.
\]

The \emph{Kreweras order} on $\NC_n$ is defined by setting $u\le_K v$ if the combinatorial noncrossing partition associated to $u$ refines the combinatorial noncrossing partition associated to $v$. 
For example $(63)(21)(54)\le_K (6321)(54)$. 
The Kreweras order is a lattice, and the lattice structure is induced by the lattice of set partitions under refinement \cite{Kre72}.

Going forward, we will identify combinatorial and algebraic noncrossing partitions, using ``combinatorial'' and ``algebraic'' only when disambiguation is needed.

We say that permutations $u,w\in S_n$ are \emph{adjacent} if $w=(i\,j)u$ for some transposition $(i\,j)$. Let $\cay{S_n}$ be the \emph{Cayley graph} on $S_n$ in which adjacent permutations are connected by an edge.
Then $u,w\in \NC_n$ are adjacent in the Kreweras lattice (meaning one element covers the other) if and only if $u,w$ are adjacent in $\cay{S_n}$. 
In this way, we may identify the Hasse diagram of the Kreweras lattice with the induced subgraph of $\cay{S_n}$ on $\NC_n$ as was first observed by Biane~\cite{Bi97}.  
We denote this induced subgraph by $\cay{\NC_{n}}$.  

\begin{rem}
    The Kreweras order is different from the Bruhat order $\le$ restricted to $\NC_n$. For example, $\cox$ is the maximal element of $\le_K$, whereas $w_0$ is the maximal element of $\le$. 
    See \cite{BJV} for a more detailed study of the interaction between the two orders.
\end{rem}

\subsection{Binary trees and nested forests}
\label{sec:forest_prelims}

A \emph{planar binary tree} is a rooted tree $T$ in which each node $v$ is either an \emph{internal node} with exactly $2$ children $v_L$ and $v_R$ (the left and right child), or $v$ is a \emph{leaf} with zero children. 
Let $\internal{T}$ denote the set of internal nodes in $T$. We allow for the possibility that $|\internal{T}|=0$, in which case the unique node is both a root and a leaf. 

A \emph{nested forest supported on $[n]$} is a family $\wh{F} = (T_{C})_{C \in \operatorname{Cyc}(w)}$ of binary trees  $T_{C}$ indexed by the disjoint cycles of a noncrossing partition $w = C_{1} C_{2} \cdots C_{k} \in \NC_{n}$ such that each $T_{C}$ has $|C|$ leaves.  
We identify the leaves of each $T_{C}$ with the underlying set of $C$ in increasing order and depict $\wh{F}$ by drawing the $T_{C}$ in the upper half plane so that the set $[n]$ of all leaves appears in increasing order along the horizontal axis.  

The internal nodes of $\wh{F}$ are $\internal{\wh{F}} = \bigsqcup_{C \in \operatorname{Cyc}(w)} \internal{T_{C}}$. 
The \emph{canonical label} of each $v \in \internal{F}$ is the value of the rightmost leaf descendant of $v_L$.

We define a map $\ncperm: \nfor_{n} \to \NC_{n}$ that sends a nested forest $(T_{C})_{C \in \operatorname{Cyc}(w)}$ to its underlying noncrossing partition $w=\prod_{C\in\operatorname{Cyc}(w)}C$.  For example
\[
\ncperm\left( \begin{tikzpicture}[scale = 0.5, baseline = 0.5*1cm]
    \draw[thin] (0.75, 0) -- (7.25, 0);
    \foreach \x in {1, 2, 3, 4, 5, 6, 7}{
    	\draw[thick] (\x - 0.1, 0 - 0.1) -- (\x + 0.1, 0 + 0.1);
    	\draw[thick] (\x - 0.1, 0 + 0.1) -- (\x + 0.1, 0 - 0.1);
    	\draw (\x, 0) node[inner sep = -2pt]  (\x) {};
    	\draw (\x) node[below] {$\scriptstyle \x$};}
    \draw[fill] (2, 1)  circle (2pt) node[inner sep = 0pt] (l1) {};
    \draw[fill] (4, 3)  circle (2pt) node[inner sep = 0pt] (l3) {};
    \draw[fill] (4.5, 0.5)  circle (2pt) node[inner sep = 0pt] (l4) {};
    \draw[fill] (5, 1)  circle (2pt) node[inner sep = 0pt] (l5) {};
    \foreach \a/\b in {l3/l1, l3/7, l1/1, l1/3, l5/6, l5/l4, l4/4, l4/5} {\draw[thin] (\a) -- (\b);}
\end{tikzpicture}  \right) = (731)(654)(2).
\]

\begin{rem}
\label{rem:nstfor_to_NC}
Nested forests should be seen as an enriched version of noncrossing partitions.  In Section~\ref{sec:torusfixedpoints}, we show that each nested forest corresponds, up to certain trivial commutation relations, to a distinguished factorization of its associated noncrossing partition.
\end{rem}

\section{Building split $\mathbb{P}^1$-bundles}
\label{sec:BuildingOperators}

We now introduce the operations used to build $\qfl_n\subset \fl{n}$. 

\subsection{$\Psi_i^-$ and $\Psi_i^+$}
\label{section:psi_plus_minus}

We recall certain pattern maps that were studied by Bergeron--Sottile \cite{BS98} and summarize their essential properties. We refer the reader to \cite{BB03,BB05} for a more general perspective.

\begin{defn}
    Let $\Psi_{i,j}\colon \operatorname{Mat}_{m-1\times m-1}\to \operatorname{Mat}_{m\times m}$ be the operation $$(\Psi_{i,j}M)_{k,\ell}=\begin{cases}M_{k-\delta_{k<i},\ell-\delta_{\ell<j}}&k\ne i\text{ and }\ell\ne j\\
    1&(k,\ell)=(i,j)\\0&\text{otherwise.}\end{cases}$$    
\end{defn}

The map $\Psi_{1,i}$ was crucial to the construction in \cite[\S 5]{NST_c}.
In contrast, the following two pattern maps will be important to us:
\[
    \Psi_i^-\coloneqq \Psi_{i,i}\text{ and }\Psi_i^+\coloneqq \Psi_{i,i+1}.
\]
For example when $m=4$ we have
$$\Psi_{2}^-\begin{bmatrix}a&b&c\\d&e&f\\g&h&i\end{bmatrix}=\begin{bmatrix}a&0&b&c\\0&1&0&0\\d&0&e&f\\g&0&h&i\end{bmatrix}\text{ and }\Psi_2^+\begin{bmatrix}a&b&c\\d&e&f\\g&h&i\end{bmatrix}=\begin{bmatrix}a&b&0&c\\0&0&1&0\\d&e&0&f\\g&h&0&i\end{bmatrix}.$$

By restricting to invertible matrices, the pattern maps descend to closed embeddings 
\[
\Psi_{i,j}:\fl{m-1}\hookrightarrow \fl{m}.
\]
Define $\gamma_i:T_m\to T_{m-1}$ to be the map $\mathrm{diag}(a_1,\ldots,a_{m})\mapsto \mathrm{diag}(a_1,\ldots,a_{i-1},a_{i+1},\ldots,a_{m})$.
\begin{defn}
\label{defn:gamma_action}
    If $T$ is a torus and $\gamma:T\to T_m$ is a map of tori then we write $\fl{m}^\gamma$ for $\fl{m}$ equipped with the action of $T$ induced by $\gamma$.
\end{defn}

\begin{fact}
    The maps $\Psi_{i,j}:\fl{m-1}^{\gamma_i}\to \fl{m}$ are $T_m$-equivariant closed embeddings, and in particular this is true of $\Psi_i^{\pm}:\fl{m-1}^{\gamma_i}\hookrightarrow \fl{m}$.
\end{fact} 

Write $\epsilon_{i}:\mathbb{C}^{m-1}\hookrightarrow \mathbb{C}^{m}$ for the inclusion $(x_1,\ldots,x_{m-1})\mapsto (x_1,\ldots,x_{i-1},0,x_i,\ldots,x_{m-1})$. 
Then $\epsilon_i$ is a $T_m$-equivariant inclusion if we give $\mathbb{C}^{m-1}$ the action of $T_m$ induced by $\gamma_i$, and we have
\begin{align*}
\Psi_i^-\mc{F}&=\{0\}\subset \epsilon_{i}(\mc{F}_1)\subset \cdots \subset \epsilon_{i}(\mc{F}_i)\subset \epsilon_{i}(\mc{F}_{i})\oplus \langle e_i\rangle \subset \epsilon_{i}(\mc{F}_{i+1})\oplus \langle e_i\rangle\subset \cdots\\
\Psi_i^+\mc{F}&=\{0\}\subset \epsilon_{i}(\mc{F}_1)\subset \cdots \subset \epsilon_{i}(\mc{F}_i)\subset \epsilon_{i}(\mc{F}_{i+1})\subset \epsilon_{i}(\mc{F}_{i+1})\oplus \langle e_i\rangle\subset \cdots.
\end{align*}

\subsection{Building $\PP^1$-bundles with $\PP_i$}
\label{sec:pushpull}

We use the subbundle convention for relative projectivization, so that for $\mc{V}$ a vector bundle on a variety $X$ we have $\proj(\mc{V})_X:=(\mc{V}\setminus \{0\})/\mathbb{C}^*$.  
Consider the sequence of maps
\[
\fl{m-1}^{\gamma_i}\overset{\Psi_i^{\pm}}{\rightrightarrows}\fl{m}\xrightarrow{\pi_i}\GL_{m}/P_i,
\]
where $P_i = \langle B, s_{i} \rangle$ is the $i$th minimal parabolic subgroup of $\GL_{m}$, and $\pi_i$ is the projection map. 

The space $\GL_m/P_i$ is typically identified with the variety of partial flags $\{0\}\subset \mc{F}_1\subset \cdots \subset \mc{F}_{i-1}\subset \mc{F}_{i+1}\subset \cdots \subset \mc{F}_m=\mathbb{C}^m$ with $\dim \mc{F}_j=j$.  
Under this identification $\pi_i$ becomes the map which forgets the $i$th subspace of a complete flag.

\begin{fact}
    $\fl{m}\to \GL_m/P_i$ is $T_m$-equivariantly isomorphic to the $\mathbb{P}^1$-bundle $\proj(\mc{F}_{i+1}/\mc{F}_{i-1})_{\GL_{m}/P_i}$.
\end{fact}

We now study how the $\PP^1$-bundle $\pi_i$ interacts with the maps $\Psi_i^{\pm}$. 
Since $s_{i}P_{i} = P_{i}$, we have the equality
$\pi_i\Psi_i^-=\pi_i\Psi_i^+.$ 
Consequently we write
\begin{align*}
    \pi_i\Psi_i\coloneqq \pi_i\Psi_i^+=\pi_i\Psi_i^-
\end{align*} to emphasize that this composite does not depend on $\pm$. 
The map $\pi_i\Psi_i$ is  given by
$$\pi_i\Psi_i(\mc{F})_j=\begin{cases}\epsilon_i(\mc{F}_j)&j<i\\
\epsilon_i(\mc{F}_{j-1})\oplus \langle e_i\rangle&j>i,\end{cases}$$
and is a closed $T_m$-equivariant embedding $\fl{m-1}^{\gamma_i}\hookrightarrow \GL_m/P_i$ with image
$$
\pi_i\Psi_i(\fl{m-1})=\big\{\{\mc{F}_j\}_{j\in [m]\setminus i}: \mc{F}_{i-1}\subset \{x_i=0\}\text{ and }e_i\in \mc{F}_{i+1}\big\}\subset \GL_m/P_i.
$$
For $Z\subset \fl{m-1}$, we define 
\[
\PP_iZ\coloneqq \pi_i^{-1}\pi_i\Psi_i Z\subset \fl{m}.
\]
Like the pattern maps we can also view $\mathbb{P}_iZ$ in terms of matrices. 
For an $(m-1)\times (m-1)$ matrix $M$, let $\mathbb{G}_iM$ be the set obtained from $\Psi_i^+M$ by replacing the $0$ in entry $(i,i)$ with all values $+ \in \mathbb{C}^*$. 
Then $\mathbb{G}_i$ descends to a map
$$\text{subsets of }\fl{m-1}\to \text{subsets of }\fl{m},$$
and $\mathbb{P}_iZ=(\Psi_i^+Z)\sqcup (\mathbb{G}_iZ)\sqcup (\Psi_i^-Z)$. For example when $m=4$ we have
\[
\PP_{2}\begin{bmatrix}a&b&c\\d&e&f\\g&h&i\end{bmatrix}B_3=\left\{\begin{bmatrix}
a&b&0&c\\
0&0 &1&0\\
d&e&0&f\\
g&h&0&i
\end{bmatrix}B_4\right\}\sqcup \left\{\begin{bmatrix}
a&b&0&c\\
0&+ &1&0\\
d&e&0&f\\
g&h&0&i
\end{bmatrix}B_4:+\in \mathbb{C}^*\right\} \sqcup \left\{\begin{bmatrix}
a&0&b&c\\
0&1&0&0\\
d&0&e&f\\
g&0&h&i\end{bmatrix}B_4\right\}.
\]

\begin{thm}
\label{thm:ZBuilding}
    Let $Z\subset \fl{m-1}$ be a $T_{m-1}$-invariant subvariety. 
    If we consider $Z$ as a $T_m$-invariant subvariety of $\fl{m-1}^{\gamma_i}$ for some fixed $1\le i \le m-1$, then the following are true.
    \begin{enumerate}
        \item The map $$ (\pi_i\Psi_i)^{-1} \pi_i\colon \PP_iZ\to Z$$ realizes $\PP_i Z$ as a $T_m$-equivariant $\PP^1$-bundle over $Z$, which is $T_m$-equivariantly isomorphic to the $\mathbb{P}^1$-bundle $\proj((\mc{F}_i/\mc{F}_{i-1})\oplus \CC_{\chi_i})_Z\to Z$.
        \item The closed subsets $\Psi_i^-|_Z$ and $\Psi_i^+|_Z$ correspond to the $T_m$-equivariant sections $\proj(\{0\}\oplus \CC_{\chi_i})_Z$ and  $\proj((\mc{F}_i/\mc{F}_{i-1})\oplus \{0\})_Z$  respectively.
    \end{enumerate}
\end{thm}
\begin{proof}
Both results follow from the case $Z=\fl{m-1}$ by restriction, so we consider only $Z=\fl{m-1}$. 
\begin{enumerate}
    \item Since $\mathbb{P}_i\fl{m-1}$ can be defined by the pullback diagram
    \begin{center}
        \begin{tikzcd}
            \mathbb{P}_i\fl{m-1}\ar[r,hookrightarrow]\ar[d]&\fl{m}=\proj(\mc{F}_{i+1}/\mc{F}_{i-1})_{\GL_m/P_i}\ar[d,"\pi_i"]\\\fl{m-1}^{\gamma_i}\ar[r,hookrightarrow,"\pi_i\Psi_i"]&\GL_m/P_i,
        \end{tikzcd}
    \end{center}
    we have $\mathbb{P}_i\fl{m-1}$ is the projectivization of the pullback of $\mc{F}_{i+1}/\mc{F}_{i-1}$ under the closed embedding $\pi_i\Psi_i$. 
    This pullback is given by
    \begin{align}
        \label{eq:pipsipullback}
        (\pi_i\Psi_i)^*(\mc{F}_{i+1}/\mc{F}_{i-1})=(\mc{F}_i\oplus \mathbb{C}_{\chi_i}) /(\mc{F}_{i-1}\oplus \{0\})\cong (\mc{F}_{i}/\mc{F}_{i-1})\oplus \CC_{\chi_i}\end{align} so the result follows.
    \item We can realize $\Psi_i^-\fl{m-1}$ and $\Psi_i^+\fl{m-1}$ as $T_m$-equivariant sections of $\pi_i|_{\mathbb{P}_i\fl{m-1}}$ taking a partial flag $\{\mc{F}_j\}_{j\in [m]\setminus i}\in \pi_i\Psi_i(\fl{m-1})$ to respectively
        \begin{align*}
             &\{0\}\subset \mc{F}_1\subset \cdots \subset \mc{F}_{i-1}\subset \mc{F}_{i-1}\oplus \langle e_i\rangle \subset \mc{F}_{i+1}\subset \cdots \subset \mc{F}_{m-1}\subset \CC^{m},\text{ and}\\
             &\{0\}\subset \mc{F}_1\subset \cdots \subset \mc{F}_{i-1}\subset \mc{F}_{i+1} \cap \{x_i=0\}\subset \mc{F}_{i+1}\subset \cdots \subset \mc{F}_{m-1}\subset \CC^{m}.
        \end{align*}
        These two sections correspond to the choice of intermediate sub-bundles $\mc{F}_{i-1}\oplus \langle e_i\rangle$ and $\mc{F}_{i+1}\cap \{x_i=0\}$ between $\mc{F}_{i-1}|_{\pi_i\Psi_i(\fl{m-1})}$ and $\mc{F}_{i+1}|_{\pi_i\Psi_i(\fl{m-1})}$, which in the pullback bundle \eqref{eq:pipsipullback} correspond to $\{0\}\oplus \CC_{\chi_i}$ and $(\mc{F}_i/\mc{F}_{i-1})\oplus \{0\}$ .\qedhere
\end{enumerate}
\end{proof}

\section{The varieties $X(\wh{F})$ and relations on the building operations}
\label{sec:RelationsonBuilding}

In this section we introduce bicolored nested forests as a tool to study compositions of the building operations $\Psi_i^-$, $\Psi_i^+$, and $\mathbb{P}_i$.  
This culminates in Definition~\ref{def:XF}, in which we introduce the varieties  $X(\wh{F})$ used to define $\qfl_{n}$.

\subsection{Combinatorics of bicolored nested forests}
\label{sec:bnfor_combinatorics}

Let $\reseq$ be the set of words from the alphabet 
\[
\bigcup_{i=1}^{\infty} \{\rletter{i}^-,\rletter{i}^+,\eletter{i}\},
\]
and for $\rt \in \reseq$ define $|\rt|$ to be the number of $\eletter{i}$ letters in $\rt$. 
For example $\rletter{2}^-\rletter{3}^+\eletter{2}\in \reseq$ and $|\rletter{2}^-\rletter{3}^+\eletter{2}|=1$. We define a distinguished subset $\reseq_n\subset \reseq$ by
\[
\reseq_n\coloneqq \{\xletter{1}\cdots \xletter{n}\in \reseq\suchthat \xletter{i}\in \{\rletter{1}^-,\ldots,\rletter{i}^-,\rletter{1}^+,\ldots,\rletter{i-1}^+,\eletter{1},\ldots,\eletter{i-1}\}\text{ for all $i$}\}.
\]
Note that every word of $\reseq_n$ begins with $\rletter{1}^-$.

\begin{defn}
\label{defn:bnforrelations}
Let $\bnfor_n$ be the quotient of $\reseq_n$ by the local relations
        \begin{gather*}
        \eletter{i} \eletter{j}=\eletter{j} \eletter{i+1}\text{ for }i> j\\
\rletter{i}^- \rletter{j}^-=\rletter{j}^- \rletter{i+1}^-\text{ for }i\ge j, \qquad \qquad 
\rletter{i}^+ \rletter{j}^+=\rletter{j}^+ \rletter{i+1}^+\text{ for }i>j,\\
\rletter{i}^+ \rletter{j}^-=\rletter{j}^- \rletter{i+1}^+\text{ for }i\ge j, \qquad \qquad 
\rletter{i}^- \rletter{j}^+=\rletter{j}^+ \rletter{i+1}^-\text{ for $i>j$},\\
\eletter{i} \rletter{j}^-=\rletter{j}^- \eletter{i+1}\text{ for }i\ge j, \qquad \qquad 
\rletter{i}^- \eletter{j}=\eletter{j} \rletter{i+1}^-\text{ for }i >j\\
\eletter{i} \rletter{j}^+=\rletter{j}^+ \eletter{i+1}\text{ for }i>j, \qquad \qquad 
\rletter{i}^+ \eletter{j}=\eletter{j}\rletter{i+1}^+\text{ for }i>j.
\end{gather*}
This is well-defined as the relations all preserve $\reseq_n$.
\end{defn}

We recall from {\cite[Proposition 9.4]{BGNST1}} how the equivalence classes in $\bnfor_{n}$ can be represented  diagrammatically.  
A \emph{bicolored nested forest} is a nested forest in which each internal node has been colored with either black ($\bnode$) or white ($\wnode$). 
We now define a map $\wh{F}$ from $\reseq_{n}$ to bicolored nested forests with support in $[n]$.  

\begin{defn}
\label{def:forest_insertion}
For $\rt \in \reseq_{n}$, the associated bicolored nested forest $\wh{F}(\rt)$ with support in $[n]$ is defined recursively by $\wh{F}(\emptyset)=\emptyset$ and 
\begin{enumerate}
    \item $\wh{F}(\rt \cdot \rletter{i}^-)$ is obtained from $\wh{F}(\rt)$ by inserting a new tree with no internal nodes as a leaf between $i-1$ and $i$ and relabeling leaves appropriately,
    \item $\wh{F}(\rt \cdot \rletter{i}^+)$ is obtained from $\wh{F}(\rt)$ by replacing the $i$th leaf with a white node $\wnode$ whose children are leaves and relabeling leaves appropriately, and
    \item $\wh{F}(\rt \cdot \eletter{i})$ is obtained from $\wh{F}(\rt)$ by replacing the $i$th leaf with a black node $\bnode$ whose children are leaves and relabeling leaves appropriately.
\end{enumerate}
\end{defn}

\begin{eg}
\label{eg:forest_insertion}
We demonstrate the recursive process for $\Omega=\rletter{1}^-\rletter{1}^{+}\rletter{2}^{-}\eletter{1}\eletter{3}\rletter{2}^+$:
\[
\wh{F}(\rletter{1}^-) = 
\begin{tikzpicture}[scale = 0.5, baseline = 0.5*-0.1cm]
\draw[thin] (0.75, 0) -- (1.25, 0);
\foreach \x in {1}{
	\draw[thick] (\x - 0.1, 0 - 0.1) -- (\x + 0.1, 0 + 0.1);
	\draw[thick] (\x - 0.1, 0 + 0.1) -- (\x + 0.1, 0 - 0.1);
	\draw (\x, 0) node[inner sep = -2pt]  (\x) {};
	\draw (\x) node[below] {$\scriptstyle \x$};}
\end{tikzpicture},
\qquad
\wh{F}(\rletter{1}^-\rletter{1}^{+}) = 
\begin{tikzpicture}[scale = 0.5, baseline = 0.5*-0.1cm]
\draw[thin] (0.75, 0) -- (2.25, 0);
\foreach \x in {1, 2}{
	\draw[thick] (\x - 0.1, 0 - 0.1) -- (\x + 0.1, 0 + 0.1);
	\draw[thick] (\x - 0.1, 0 + 0.1) -- (\x + 0.1, 0 - 0.1);
	\draw (\x, 0) node[inner sep = -2pt]  (\x) {};
	\draw (\x) node[below] {$\scriptstyle \x$};}
\draw (1.5, 0.5) circle (2pt) node[inner sep = -2pt] (l1) {};
\foreach \a/\b in {1/l1, 2/l1} {\draw[thin] (\a) -- (\b);}
\end{tikzpicture},
\qquad
\wh{F}(\rletter{1}^-\rletter{1}^{+}\rletter{2}^{-}) = 
\begin{tikzpicture}[scale = 0.5, baseline = 0.5*-0.1cm]
\draw[thin] (0.75, 0) -- (3.25, 0);
\foreach \x in {1, 2, 3}{
	\draw[thick] (\x - 0.1, 0 - 0.1) -- (\x + 0.1, 0 + 0.1);
	\draw[thick] (\x - 0.1, 0 + 0.1) -- (\x + 0.1, 0 - 0.1);
	\draw (\x, 0) node[inner sep = -2pt]  (\x) {};
	\draw (\x) node[below] {$\scriptstyle \x$};}
\draw (2, 1) circle (2pt) node[inner sep = -2pt] (l1) {};
\foreach \a/\b in {1/l1, 3/l1} {\draw[thin] (\a) -- (\b);}
\end{tikzpicture},
\qquad
\wh{F}(\rletter{1}^-\rletter{1}^{+}\rletter{2}^{-}\eletter{1}) = 
\begin{tikzpicture}[scale = 0.5, baseline = 0.5*-0.1cm]
\draw[thin] (0.75, 0) -- (4.25, 0);
\foreach \x in {1, 2, 3, 4}{
	\draw[thick] (\x - 0.1, 0 - 0.1) -- (\x + 0.1, 0 + 0.1);
	\draw[thick] (\x - 0.1, 0 + 0.1) -- (\x + 0.1, 0 - 0.1);
	\draw (\x, 0) node[inner sep = -2pt]  (\x) {};
	\draw (\x) node[below] {$\scriptstyle \x$};}
\draw[fill] (1.5, 0.5) circle (2pt) node[inner sep = -2pt] (l1) {};
\draw (2.5, 1.5) circle (2pt) node[inner sep = -2pt] (l3) {};
\foreach \a/\b in {1/l1, 2/l1, l1/l3, l3/4} {\draw[thin] (\a) -- (\b);}
\end{tikzpicture},
\]
\[
\wh{F}(\rletter{1}^-\rletter{1}^{+}\rletter{2}^{-}\eletter{1}\eletter{3}) = 
\begin{tikzpicture}[scale = 0.5, baseline = 0.5*-0.1cm]
\draw[thin] (0.75, 0) -- (5.25, 0);
\foreach \x in {1, 2, 3, 4, 5}{
	\draw[thick] (\x - 0.1, 0 - 0.1) -- (\x + 0.1, 0 + 0.1);
	\draw[thick] (\x - 0.1, 0 + 0.1) -- (\x + 0.1, 0 - 0.1);
	\draw (\x, 0) node[inner sep = -2pt]  (\x) {};
	\draw (\x) node[below] {$\scriptstyle \x$};}
\draw[fill]  (1.5, 0.5) circle (2pt) node[inner sep = -2pt] (l1) {};
\draw (3, 2) circle (2pt) node[inner sep = -2pt] (l3) {};
\draw[fill] (3.5, 0.5) circle (2pt) node[inner sep = -2pt] (l4) {};
\foreach \a/\b in {1/l1, 2/l1, 3/l4, 4/l4, l1/l3, 5/l3} {\draw[thin] (\a) -- (\b);}
\end{tikzpicture},
\qquad\text{and}\qquad
\wh{F}(\rletter{1}^-\rletter{1}^{+}\rletter{2}^{-}\eletter{1}\eletter{3}\rletter{2}^+) = 
\begin{tikzpicture}[scale = 0.5, baseline = 0.5*-0.1cm]
\draw[thin] (0.75, 0) -- (6.25, 0);
\foreach \x in {1, 2, 3, 4, 5, 6}{
	\draw[thick] (\x - 0.1, 0 - 0.1) -- (\x + 0.1, 0 + 0.1);
	\draw[thick] (\x - 0.1, 0 + 0.1) -- (\x + 0.1, 0 - 0.1);
	\draw (\x, 0) node[inner sep = -2pt]  (\x) {};
	\draw (\x) node[below] {$\scriptstyle \x$};}
\draw[fill] (2, 1) circle (2pt) node[inner sep = -2pt] (l1) {};
\draw  (2.5, 0.5) circle (2pt) node[inner sep = -2pt] (l2) {};
\draw (3.5, 2.5) circle (2pt) node[inner sep = -2pt] (l3) {};
\draw[fill] (4.5, 0.5) circle (2pt) node[inner sep = -2pt] (l4) {};
\foreach \a/\b in {1/l1, 2/l2, 3/l2, 4/l4, 5/l4, l1/l3, l2/l1, 6/l3} {\draw[thin] (\a) -- (\b);}
\end{tikzpicture}.
\]
\end{eg}

\begin{thm}
{\cite[Proposition 9.4]{BGNST1}}
Two elements $\rt,\rt'\in \reseq_n$ are in the same equivalence class of $\bnfor_n$ if and only if $\wh{F}(\rt)=\wh{F}(\rt')$.
\end{thm}

Going forward we identify elements of $\bnfor_{n}$ with the associated bicolored nested forest, so that we can write $\wh{K} = \wh{H} \cdot \xletter{i}$ to mean  $\wh{H} = \wh{F}(\rt)$ and $\wh{K} = \wh{F}(\rt \cdot \xletter{i})$ for some $\rt\in \reseq_n$.  
The defining relations of $\bnfor_n$ preserve $|\rt|$, so it makes sense to discuss $|\wh{F}|$. 
Diagrammatically, $|\wh{F}|$ is the number of black nodes in $\wh{F}$. 
See \Cref{eg:forest_insertion} for a bicolored nested forest built from an $\rt\in \reseq_6$ with $|\rt|=2$. 

\begin{defn}
\label{defn:forest_subsets}
We conclude by identifying some distinguished subsets of $\bnfor_{n}$ which will appear in later sections.
\begin{enumerate}
\item Let $\indexedforests_n = \{ \wh{F} \in \bnfor_{n} \suchthat \text{each $\rt \in \wh{F}$ has the form $(\rletter{1}^{-})^{n-k}\eletter{i_1}\cdots \eletter{i_k}$}\}$; elements of this set map to forests without white nodes in which no nesting occurs, so that the support of each tree is a contiguous interval of $[n]$.  

\item Let $\operatorname{Tree}_{n} = \{ \wh{F} \in \bnfor_{n} \suchthat |\wh{F}| = n-1\}$; elements of this set have representatives of the form $\rletter{1}^- \eletter{i_1} \cdots \eletter{i_{n-1}}$ and map to singleton trees $(T_{\mathbf{c}})$ with entirely black nodes.  
\end{enumerate}
\end{defn}

We note that in the above definition, $\operatorname{Tree}_{n} \subseteq \indexedforests_n \subseteq \bnfor_{n}$.

\subsection{Combinatorics of building operations}
\label{sec:building_combinatorics}

\begin{lem}
\label{lem:babyrtsame}
We have the relations
        \begin{gather*}
        \PP_{j}\PP_{i}=\PP_{i+1}\PP_{j} \text{ for }i> j\\
 \Psi_{j}^-\Psi_{i}^-= \Psi_{i+1}^-\Psi_{j}^-\text{ for }i\ge j, \qquad \qquad 
 \Psi_{j}^+\Psi_{i}^+= \Psi_{i+1}^+\Psi_{j}^+\text{ for }i>j,\\
 \Psi_{j}^-\Psi_{i}^+= \Psi_{i+1}^+\Psi_{j}^-\text{ for }i\ge j, \qquad \qquad 
 \Psi_{j}^+\Psi_{i}^-= \Psi_{i+1}^-\Psi_{j}^+\text{ for $i>j$},\\
 \Psi_{j}^-\PP_{i}= \PP_{i+1}\Psi_{j}^-\text{ for }i\ge j, \qquad \qquad 
 \PP_{j}\Psi_{i}^-=\Psi_{i+1}^-\PP_{j} \text{ for }i >j\\
 \Psi_{j}^+\PP_{i}=\PP_{i+1}\Psi_{j}^+ \text{ for }i>j, \qquad \qquad 
 \PP_{j}\Psi_{i}^+=\Psi_{i+1}^+\PP_{j}\text{ for }i>j.
\end{gather*}
Moreover, these relations hold when all $\mathbb{P}$'s are replaced by $\mathbb{G}$'s.
\end{lem}
\begin{proof}
Since $\mathbb{P}_i=\Psi_i^-\sqcup \mathbb{G}_i\sqcup \Psi_i^+$, it suffices to check all of the above relations with $\mathbb{G}_i$ in place of $\mathbb{P}_i$. 
It suffices to show the stronger statement that these relations hold for the operations $\Psi_i^-$, $\Psi_i^+$, and $\mathbb{G}_i$ on subsets of matrices without considering equivalence classes mod $B$. These are straightforward so we omit their explicit verification.
\end{proof}

\begin{rem}
The relations in~\Cref{lem:babyrtsame} are not exhaustive: in \Cref{lem:2additionalbuilding} we show that two additional relations are needed to describe all interactions between the building operations.
\end{rem}

\begin{eg}
   The following computation witnesses the relation $\mathbb{G}_1\mathbb{G}_2=\mathbb{G}_3\mathbb{G}_1$:
\[
\begin{bmatrix}a&b&c\\d&e&f\\g&h&i\end{bmatrix}
\xrightarrow{\mathbb{G}_2}
\begin{bmatrix}a&b&0&c\\0&+&1&0\\d&e&0&f\\g&h&0&i\end{bmatrix}
\xrightarrow{\mathbb{G}_1}
\begin{bmatrix}+&1&0&0&0\\a&0&b&0&c\\0&0&+&1&0\\d&0&e&0&f\\g&0&h&0&i\end{bmatrix}
\xleftarrow{\mathbb{G}_3}
\begin{bmatrix}+&1&0&0\\a&0&b&c\\d&0&e&f\\g&0&h&i\end{bmatrix}
\xleftarrow{\mathbb{G}_1}
\begin{bmatrix}a&b&c\\d&e&f\\g&h&i\end{bmatrix}.
\]
\end{eg}

Since the relations of \Cref{lem:babyrtsame} are the opposites to the defining relations of $\bnfor_n$, the following is well-defined.

\begin{defn}\label{def:XF}
    For $\wh{F}\in \bnfor_n$, we define  $X(\wh{F})\subset\fl{n}$ recursively:  $X(\rletter{1}^{-})=\fl{1}$ is a single point, and 
\begin{enumerate}
    \item $X(\wh{F} \cdot \rletter{j}^\pm)=\Psi_{j}^{\pm}X(\wh{F} )$
    \item $X(\wh{F}\cdot  \eletter{j})=\PP_jX(\wh{F})$.
\end{enumerate}
Similarly, we define $X^{\circ}(\wh{F})\subset \fl{n}$ recursively: $X^{\circ}(\rletter{1}^{-})=\fl{1}$ is a single point, and
\begin{enumerate}
    \item $X^{\circ}(\wh{F}\cdot \rletter{j}^{\pm})=\Psi_j^{\pm}X^{\circ}(\wh{F})$
    \item $X^{\circ}(\wh{F}\cdot \eletter{j})=\GG_jX^{\circ}(\wh{F})$.
\end{enumerate} 
\end{defn}

\begin{rem}
\label{rem:quasischubert}
We conclude by revisiting the distinguished families from Definition~\ref{defn:forest_subsets}.
\begin{enumerate}
\item In Theorem~\ref{thm:samesetsrichardson}, we show that the $X(F)$ for $F \in \indexedforests_n$ are the \emph{quasisymmetric Schubert cycles} of~\cite[Definition 6.2]{NST_c}. 

\item For $T \in \operatorname{Tree}_{n}$, the $X(T)$ are the top-dimensional quasisymmetric Schubert cycles, which we use to construct $\qfl_{n}$.  

\end{enumerate}
\end{rem}

\section{The Bott manifolds $X(\wh{F})$}
\label{sec:BottManifold}

We now describe the toric structure of varieties $X(\wh{F})$ and $X^{\circ}(\wh{F})$ in Definition~\ref{def:XF}.

\begin{thm}
\label{thm:XF_torus_orbit}
For $\wh{F} \in \bnfor_{m}$, the variety $X^{\circ}(\wh{F})$ is a torus orbit in $\fl{m}$ of dimension $|\wh{F}|$. Furthermore, $X(\wh{F})$ is the closure of the torus orbit $X^{\circ}(\wh{F})$ in $\fl{m}$.
\end{thm}

We prove the theorem using the following.  
Recall the meaning of $\fl{m-1}^{\gamma_i}$ from Definition~\ref{defn:gamma_action}.

\begin{prop}
\label{prop:XBottManifold}
Fix $\wh{F}\in \bnfor_{m-1}$, and consider $X(\wh{F})\subset \fl{m-1}^{\gamma_i}$ for some fixed $1\le i \le m$.
\begin{enumerate}[itemsep = 1ex, label=(\arabic*)]
    \item \label{it1:5.1} The map $\Psi_i^{-}:X(\wh{F})\to X(\wh{F}\cdot \rletter{i}^{-})$ is a $T_m$-equivariant isomorphism.
    \item \label{it2:5.1} If $i < m$, the map $\Psi_i^{+}:X(\wh{F})\to X(\wh{F} \cdot \rletter{i}^{+})$ is a $T_m$-equivariant isomorphism.
    \item \label{it3:5.1} If $i < m$, there exists a $T_m$-equivariant isomorphism $X(\wh{F}\cdot \eletter{i})\cong \proj(\mc{F}_i/\mc{F}_{i-1}\oplus \mathbb{C}_{\chi_i})_{X(\wh{F})}$. Furthermore, $X(\wh{F}\cdot \rletter{i}^-)$ and $X(\wh{F}\cdot \rletter{i}^+)$ are the $T_m$-equivariant sections $\proj(\{0\}\oplus \mathbb{C}_{\chi_i})_{X(\wh{F})}$ and $\proj((\mc{F}_i/\mc{F}_{i-1})\oplus \{0\})_{X(\wh{F})}$  respectively.
\end{enumerate}
\end{prop}
\begin{proof}
These are immediate corollaries of \Cref{thm:ZBuilding} and \Cref{def:XF}.
\end{proof}

\begin{proof}[Proof of Theorem~\ref{thm:XF_torus_orbit}]
The dimension statement follows because each $\mathbb{G}_i$ defining $X^{\circ}(\wh{F})$ increases the dimension by $1$ and each $\Psi_i^{\pm}$ preserves the dimension.
    To show that $X^{\circ}(\wh{F})$ is a torus orbit, we induct on $m$. 
    By Proposition~\ref{prop:XBottManifold}\ref{it1:5.1} and~\ref{it2:5.1} if we know $X^{\circ}(\wh{F})$ is a torus orbit then so is $X^{\circ}(\wh{F}\cdot \rletter{i}^{\pm})$. 
    \Cref{prop:XBottManifold} further implies that for any $Y\subset X(\wh{F})$ we have that $\mathbb{G}_iY\to Y$ is a $\mathbb{C}^*$-bundle obtained from $\mathbb{P}_iY$ by removing the two distinguished sections. 
    By applying this for $Y=X^{\circ}(\wh{F})$ we conclude that $T_m$ acts transitively on $\mathbb{G}_iX^{\circ}(\wh{F})=X^{\circ}(\wh{F}\cdot \eletter{i})$ from the fact that $T_m\cong \gamma_i(T_m)\times \mathbb{C}^*_{\chi_i}$ with $\gamma_i(T_m)$ acting transitively on the base of the $\mathbb{C}^*$-bundle $X^{\circ}(\wh{F}\cdot \eletter{i})\to X^{\circ}(\wh{F})$ and $\mathbb{C}^*_{\chi_i}$ acting transitively on the fibers while fixing the base. Finally, for any $Z\subset \fl{m-1}$ we have $\Psi_i^{\pm}\overline{Z}=\overline{\Psi_i^{\pm}Z}$ and $\mathbb{P}_i\overline{Z}=\overline{\mathbb{P}_iZ}=\overline{\mathbb{G}Z}$, which shows by induction that $X(\wh{F})$ is the closure of $X^{\circ}(\wh{F})$.
\end{proof}

Recall that a \emph{combinatorial cube} is a polytope whose face lattice is identical to that of a cube of the same dimension. 
If we have a sequence of varieties $X_1,X_2,\ldots,X_m$ where $X_1=\{\point\}$ is a single point and $X_i=\proj(\mc{L}_{i-1}\oplus \mathbb{C})_{X_{i-1}}$ for $\mc{L}_{i-1}$ a toric line bundle on $X_{i-1}$, then $X_m$ is a smooth projective toric variety whose moment polytope is a combinatorial cube. 
The toric structure is defined recursively by saying if $T_{i-1}$ is the torus for $X_{i-1}$, then $X_i$ is a toric variety for $T_i\coloneqq T_{i-1}\times \mathbb{C}^*$ where $T_{i-1}$ acts trivially on the factor of $\mathbb{C}$ and $\mathbb{C}^*$ acts by scaling this $\mathbb{C}$ factor. 
The dense torus orbit can also be obtained by taking $\proj(\mc{L}_{i-1}\oplus \mathbb{C})_{T_i}$ and removing the two distinguished sections. 
Such an $X_m$ obtained in this way is called a \emph{Bott manifold} (see \cite{MP08}).  

\begin{defn}
For $\rt_1,\rt_2\in \reseq_n$, we say $\rt_1\le_{re} \rt_2$ if $\rt_1$ is obtained from $\rt_2$ by switching some letters $\eletter{i}$ to $\rletter{i}^{\pm}$.
\end{defn}

\begin{prop}
For $\Omega_{1} \eletter{i} \Omega_{2} \in \reseq_{n}$, 
\label{prop:bicoloredchange}
\leavevmode
\begin{enumerate}[label=(\arabic*)]
    \item \label{it1:etorplus} $\wh{F}(\rt_1\rletter{i}^+\rt_2)$ is obtained from $\wh{F}(\rt_1\eletter{i}\rt_2)$ by changing the black node associated to $\eletter{i}$ to a white node and 
    \item  \label{it2:etorminus}  $\wh{F}(\rt_1\rletter{i}^-\rt_2)$ is obtained from $\wh{F}(\rt_1\eletter{i}\rt_2)$ by deleting the left edge of the black node associated to $\eletter{i}$ and contracting the resulting node.
\end{enumerate}
In particular, the relation $\le_{re}$ descends to a partial order on $\bnfor_n$.
\end{prop}
\begin{proof}
    This follows immediately from the way that bicolored nested forests are created recursively from sequences in $\reseq_n$.
\end{proof}

We define the operation of \emph{left edge deletion} at a node in $\internal{\wh{F}}$ to be the operation described in Proposition~\ref{prop:bicoloredchange}\ref{it2:etorminus}.  
We sometimes emphasize that a forest is obtained by left edge deletion at $v$ by drawing the contracted edge through the former location of $v$, as shown below.

\begin{eg}
\label{eg:edge_deletion}
We demonstrate \Cref{prop:bicoloredchange} using the nested forest $\wh{F}(\rletter{1}^-\rletter{1}^{+}\rletter{2}^{-}\eletter{1}\eletter{3}\rletter{2}^+)$ from Example~\ref{eg:forest_insertion}.  
Three examples of $\le_{re}$-smaller nested forests are:
\[
\begin{tikzpicture}[scale = 0.5, baseline = 0.5*-0.1cm]
\draw[thin] (0.75, 0) -- (6.25, 0);
\foreach \x in {1, 2, 3, 4, 5, 6}{
	\draw[thick] (\x - 0.1, 0 - 0.1) -- (\x + 0.1, 0 + 0.1);
	\draw[thick] (\x - 0.1, 0 + 0.1) -- (\x + 0.1, 0 - 0.1);
	\draw (\x, 0) node[inner sep = -2pt]  (\x) {};
	\draw (\x) node[below] {$\scriptstyle \x$};}
\draw (2, 1) circle (2pt) node[inner sep = -2pt] (l1) {};
\draw  (2.5, 0.5) circle (2pt) node[inner sep = -2pt] (l2) {};
\draw (3.5, 2.5) circle (2pt) node[inner sep = -2pt] (l3) {};
\draw[fill] (4.5, 0.5) circle (2pt) node[inner sep = -2pt] (l4) {};
\foreach \a/\b in {1/l1, 2/l2, 3/l2, 4/l4, 5/l4, l1/l3, l2/l1, 6/l3} {\draw[thin] (\a) -- (\b);}
\draw (3.5, -1.5) node {$\wh{F}(\rletter{1}^-\rletter{1}^{+}\rletter{2}^{-}{\color{red}\rletter{1}^{+}}\eletter{3}\rletter{2}^+)$};
\end{tikzpicture},
\qquad
\begin{tikzpicture}[scale = 0.5, baseline = 0.5*-0.1cm]
\draw[thin] (0.75, 0) -- (6.25, 0);
\foreach \x in {1, 2, 3, 4, 5, 6}{
	\draw[thick] (\x - 0.1, 0 - 0.1) -- (\x + 0.1, 0 + 0.1);
	\draw[thick] (\x - 0.1, 0 + 0.1) -- (\x + 0.1, 0 - 0.1);
	\draw (\x, 0) node[inner sep = -2pt]  (\x) {};
	\draw (\x) node[below] {$\scriptstyle \x$};}
\draw[fill] (2, 1) circle (2pt) node[inner sep = -2pt] (l1) {};
\draw  (2.5, 0.5) circle (2pt) node[inner sep = -2pt] (l2) {};
\draw (3.5, 2.5) circle (2pt) node[inner sep = -2pt] (l3) {};
\foreach \a/\b in {1/l1, 2/l2, 3/l2,  l1/l3, l2/l1, 6/l3} {\draw[thin] (\a) -- (\b);}
\draw (3.5, -1.5) node {$\wh{F}(\rletter{1}^-\rletter{1}^{+}\rletter{2}^{-}\eletter{1}{\color{red}\rletter{3}^{-}}\rletter{2}^+)$};
\end{tikzpicture},
\qquad\text{and}\qquad
\begin{tikzpicture}[scale = 0.5, baseline = 0.5*-0.1cm]
\draw[thin] (0.75, 0) -- (6.25, 0);
\foreach \x in {1, 2, 3, 4, 5, 6}{
	\draw[thick] (\x - 0.1, 0 - 0.1) -- (\x + 0.1, 0 + 0.1);
	\draw[thick] (\x - 0.1, 0 + 0.1) -- (\x + 0.1, 0 - 0.1);
	\draw (\x, 0) node[inner sep = -2pt]  (\x) {};
	\draw (\x) node[below] {$\scriptstyle \x$};}
\draw (2.5, 0.5) circle (2pt) node[inner sep = -2pt] (l2) {};
\draw (3.5, 2.5) circle (2pt) node[inner sep = -2pt] (l3) {};
\draw[fill] (4.5, 0.5) circle (2pt) node[inner sep = -2pt] (l4) {};
\foreach \a/\b in {2/l2, 3/l2, 4/l4, 5/l4, 6/l3} {\draw[thin] (\a) -- (\b);}
\draw (l3) -- (2, 1) -- (l2);
\draw (7, 0) node {$=$};
\begin{scope}[xshift = 7cm]
\draw[thin] (0.75, 0) -- (6.25, 0);
\foreach \x in {1, 2, 3, 4, 5, 6}{
	\draw[thick] (\x - 0.1, 0 - 0.1) -- (\x + 0.1, 0 + 0.1);
	\draw[thick] (\x - 0.1, 0 + 0.1) -- (\x + 0.1, 0 - 0.1);
	\draw (\x, 0) node[inner sep = -2pt]  (\x) {};
	\draw (\x) node[below] {$\scriptstyle \x$};}
\draw (2.5, 0.5) circle (2pt) node[inner sep = -2pt] (l2) {};
\draw (4, 2) circle (2pt) node[inner sep = -2pt] (l3) {};
\draw[fill] (4.5, 0.5) circle (2pt) node[inner sep = -2pt] (l4) {};
\foreach \a/\b in {2/l2, 3/l2, l3/l2, 4/l4, 5/l4, 6/l3} {\draw[thin] (\a) -- (\b);}
\end{scope}
\draw (7, -1.5) node {$\wh{F}(\rletter{1}^-\rletter{1}^{+}\rletter{2}^{-}{\color{red}\rletter{1}^{-}}\eletter{3}\rletter{2}^+)$};
\end{tikzpicture}.
\]
\end{eg}

\begin{defn}
    For $\wh{F}\in \bnfor_n$, we define   \begin{align*}\Face(\wh{F})&=\{\wh{G}:\wh{G}\le_{re}\wh{F}\}\text{ and }\\\Vert(\wh{F})&=\{\wh{G}:\wh{G}\le_{re}\wh{F}\text{ and }|\wh{G}|=0\}.
    \end{align*}
\end{defn}

Restricting $\le_{re}$ to $\Face(\wh{F})$ gives the face poset of a $|\wh{F}|$-dimensional cube.
Indeed, the choice of whether to change each black node white or to perform left edge deletion is equivalent to choosing one from a pair of opposite faces.  
Figure~\ref{fig:led_poset_kreweras_eg} shows in the left panel two such cubes and the elements of $\Face(\wh{F})$ associated with each face.  
We now give this interpretation a geometric~meaning.  

\begin{thm}
\label{thm:smoothTorbit}
For $\wh{F}\in \bnfor_n$,
     $X(\wh{F})\subset \fl{n}$ is a Bott manifold of dimension $|\wh{F}|$ with dense torus orbit $X^{\circ}(\wh{F})$. 
     The distinct torus orbits of $X(\wh{F})$  are given by $\{X^{\circ}(\wh{G}):\wh{G}\in \Face(\wh{F})\}$, and the distinct torus orbit closures are given by $\{X(\wh{G}):\wh{G}\in \Face(\wh{F})\}$.
\end{thm}
\begin{proof}
    By \Cref{prop:XBottManifold}, any $\rt\in \reseq_n$ representing $\wh{F}$ induces a Bott manifold structure on $X(\wh{F})$ with dense torus orbit $X^{\circ}(\wh{F})$ as described above, and the dimension is $|\wh{F}|$ by \Cref{thm:XF_torus_orbit}. 
    
    As \Cref{thm:XF_torus_orbit} shows the closure of the torus orbit $X^{\circ}(\wh{G})$ is $X(\wh{G})$, it remains to describe the distinct torus-orbit closures. 
    We proceeed inductively. 
    Suppose the result holds for all $\wh{F}\in \bnfor_{n-1}$; we aim to prove it  for each $\wh{F}\cdot \xletter{i}\in \bnfor_n$. 
    
    First, note that $\Psi_i^{\pm}:X(\wh{F})\to X(\wh{F}\cdot \rletter{i}^{\pm})$ is an isomorphism and moreover induces a bijection between torus orbit closures via $$X(\wh{G})\mapsto \Psi_i^{\pm}X(\wh{G})=X(\wh{G}\cdot\rletter{i}^{\pm}),$$
    so we conclude the result for $\wh{F}\cdot \rletter{i}^{\pm}$ as $\Face(\wh{F}\cdot \rletter{i}^{\pm})=\{\wh{G}\cdot\rletter{i}^{\pm}:\wh{G}\in \Face(\wh{F})\}$.
    Consider now the $\mathbb{P}^1$-bundle $X(\wh{F}\cdot\eletter{i})\to X(\wh{F})$. For any projective toric variety $X$ with a toric line bundle $\mc{L}$, the torus orbit closures on the toric variety $\proj(\mc{L}\oplus \mathbb{C})$ are given by the $\mathbb{P}^1$-bundles over the torus orbit closures in $X$, together with the images of the torus orbit closures in $X$ in the two disjoint sections of the split projective bundle. 
    Consequently by \Cref{prop:XBottManifold} the torus orbit closures in $X(\wh{F})$ are given by
    \begin{align*}\bigsqcup_{\wh{G}\in \Face(\wh{F})}\{\Psi_i^-X(\wh{G}),\Psi_i^+X(\wh{G}),\mathbb{P}_iX(\wh{G})\}=\bigsqcup_{\wh{G}\in \Face(\wh{F})}\{X(\wh{G}\cdot\rletter{i}^-),X(\wh{G}\cdot\rletter{i}^+),X(\wh{G}\cdot\eletter{i})\}\end{align*}
    and we conclude as $\Face(\wh{F}\cdot \eletter{i})=\bigcup_{\wh{G}\in \Face(\wh{F})}\{\wh{G}\cdot \rletter{i}^-,\wh{G}\cdot \rletter{i}^+,\wh{G}\cdot \eletter{i}\}$.
\end{proof}

\section{Torus fixed points of $X(\wh{F})$}
\label{sec:torusfixedpoints}

In this section we describe the combinatorics of the fixed point sets
\[
I_{\wh{F}}\coloneqq X(\wh{F})^T\subset S_n.
\]
As the following theorem shows, elements of $I_{\wh{F}}$ always lie in $\NC_n$.

\begin{thm}
\label{thm:noncrossing_fixed_points}
For any $\wh{F} \in \bnfor_{n}$, we have \[I_{\wh{F}} = \{ \ncperm(\wh{G}) \suchthat \text{$\wh{G} \in \Vert(\wh{F})$}\}\subset \NC_n.\]
In particular, if $\wh{G} \in \bnfor_{n}$ has $|\wh{G}|=0$, then $X_{\wh{G}}=\{\ncperm(\wh{G})\}$.
\end{thm}

\begin{eg}
We apply the theorem to the fixed point set for the nested forest $\wh{F}(\rletter{1}^-\rletter{1}^{+}\rletter{2}^{-}\eletter{1}\eletter{3}\rletter{2}^+)$ from Examples~\ref{eg:forest_insertion} and~\ref{eg:edge_deletion}.   We have:
\begin{align*}
(6321)(54)
\leftrightarrow
\begin{tikzpicture}[scale = 0.5, baseline = 0.5*-0.1cm]
\draw[thin] (0.75, 0) -- (6.25, 0);
\foreach \x in {1, 2, 3, 4, 5, 6}{
	\draw[thick] (\x - 0.1, 0 - 0.1) -- (\x + 0.1, 0 + 0.1);
	\draw[thick] (\x - 0.1, 0 + 0.1) -- (\x + 0.1, 0 - 0.1);
	\draw (\x, 0) node[inner sep = -2pt]  (\x) {};
	\draw (\x) node[below] {$\scriptstyle \x$};}
\draw(2, 1) circle (2pt) node[inner sep = -2pt] (l1) {};
\draw  (2.5, 0.5) circle (2pt) node[inner sep = -2pt] (l2) {};
\draw (3.5, 2.5) circle (2pt) node[inner sep = -2pt] (l3) {};
\draw (4.5, 0.5) circle (2pt) node[inner sep = -2pt] (l4) {};
\foreach \a/\b in {1/l1, 2/l2, 3/l2, 4/l4, 5/l4, l1/l3, l2/l1, 6/l3} {\draw[thin] (\a) -- (\b);}
\draw (3.5, -1.5) node {$\wh{F}(\rletter{1}^-\rletter{1}^{+}\rletter{2}^{-}{\color{red}\rletter{1}^{+}}{\color{red}\rletter{3}^{+}}\rletter{2}^+)$};
\end{tikzpicture}
&&
(6321)
\leftrightarrow
\begin{tikzpicture}[scale = 0.5, baseline = 0.5*-0.1cm]
\draw[thin] (0.75, 0) -- (6.25, 0);
\foreach \x in {1, 2, 3, 4, 5, 6}{
	\draw[thick] (\x - 0.1, 0 - 0.1) -- (\x + 0.1, 0 + 0.1);
	\draw[thick] (\x - 0.1, 0 + 0.1) -- (\x + 0.1, 0 - 0.1);
	\draw (\x, 0) node[inner sep = -2pt]  (\x) {};
	\draw (\x) node[below] {$\scriptstyle \x$};}
\draw(2, 1) circle (2pt) node[inner sep = -2pt] (l1) {};
\draw  (2.5, 0.5) circle (2pt) node[inner sep = -2pt] (l2) {};
\draw (3.5, 2.5) circle (2pt) node[inner sep = -2pt] (l3) {};
\foreach \a/\b in {1/l1, 2/l2, 3/l2,  l1/l3, l2/l1, 6/l3} {\draw[thin] (\a) -- (\b);}
\draw (3.5, -1.5) node {$\wh{F}(\rletter{1}^-\rletter{1}^{+}\rletter{2}^{-}{\color{red}\rletter{1}^{+}}{\color{red}\rletter{3}^{-}}\rletter{2}^+)$};
\end{tikzpicture}
\\[2ex]
(632)(54)
\leftrightarrow
\begin{tikzpicture}[scale = 0.5, baseline = 0.5*-0.1cm]
\draw[thin] (0.75, 0) -- (6.25, 0);
\foreach \x in {1, 2, 3, 4, 5, 6}{
	\draw[thick] (\x - 0.1, 0 - 0.1) -- (\x + 0.1, 0 + 0.1);
	\draw[thick] (\x - 0.1, 0 + 0.1) -- (\x + 0.1, 0 - 0.1);
	\draw (\x, 0) node[inner sep = -2pt]  (\x) {};
	\draw (\x) node[below] {$\scriptstyle \x$};}
\draw  (2.5, 0.5) circle (2pt) node[inner sep = -2pt] (l2) {};
\draw (3.5, 2.5) circle (2pt) node[inner sep = -2pt] (l3) {};
\draw (4.5, 0.5) circle (2pt) node[inner sep = -2pt] (l4) {};
\foreach \a/\b in {2/l2, 3/l2, 4/l4, 5/l4, 6/l3} {\draw[thin] (\a) -- (\b);}
\draw (l3) -- (2, 1) -- (l2);
\draw (3.5, -1.5) node {$\wh{F}(\rletter{1}^-\rletter{1}^{+}\rletter{2}^{-}{\color{red}\rletter{1}^{-}}{\color{red}\rletter{3}^{+}}\rletter{2}^+)$};
\end{tikzpicture}
&&
(632)
\leftrightarrow
\begin{tikzpicture}[scale = 0.5, baseline = 0.5*-0.1cm]
\draw[thin] (0.75, 0) -- (6.25, 0);
\foreach \x in {1, 2, 3, 4, 5, 6}{
	\draw[thick] (\x - 0.1, 0 - 0.1) -- (\x + 0.1, 0 + 0.1);
	\draw[thick] (\x - 0.1, 0 + 0.1) -- (\x + 0.1, 0 - 0.1);
	\draw (\x, 0) node[inner sep = -2pt]  (\x) {};
	\draw (\x) node[below] {$\scriptstyle \x$};}
\draw  (2.5, 0.5) circle (2pt) node[inner sep = -2pt] (l2) {};
\draw (3.5, 2.5) circle (2pt) node[inner sep = -2pt] (l3) {};
\foreach \a/\b in {2/l2, 3/l2, 6/l3} {\draw[thin] (\a) -- (\b);}
\draw (l3) -- (2, 1) -- (l2);
\draw (3.5, -1.5) node {$\wh{F}(\rletter{1}^-\rletter{1}^{+}\rletter{2}^{-}{\color{red}\rletter{1}^{-}}{\color{red}\rletter{3}^{-}}\rletter{2}^+)$};
\end{tikzpicture}
\end{align*}
\end{eg}

Going forward, we will abuse notation and treat $\Psi^{-}_{i}$ and $\Psi^{+}_{i}$ as maps on permutations rather than just permutation matrices. These are given by the group homomorphism $\Psi_i^-:S_{n-1}\hookrightarrow S_n$, induced by the increasing injection $\{1,\ldots,n-1\}\hookrightarrow \{1,\ldots,i-1,i+1,\ldots,n\}$ onto the subgroup of $S_n$ with $u(i)=i$, and $\Psi_i^+w=(\Psi_i^-w)s_i$.

\begin{proof}[Proof of \Cref{thm:noncrossing_fixed_points}]
We know from \Cref{thm:smoothTorbit} that the fixed point set $I_{\wh{F}}$ is given by the points $X_{\wh{G}}$ such that $\wh{G}\in \Vert(\wh{F})$.  
Therefore we need only verify the the statement when $|\wh{G}| = 0$, meaning that $\wh{G}$ is represented by a sequence of $\rletter{i}^{\pm}$.  

For $n=1$ we have $\wh{G}=\rletter{1}^-$ and $X_{\wh{G}} = \{\idem_{S_1}\} = \{\ncperm(\wh{G})\}$. 
 For $n>1$, by the recursive construction of $X_{\wh{G}}$ in Definition~\ref{def:XF} it suffices to show that
\begin{align}
\label{eq:ncperm_recurse}
\ncperm(\wh{F} \cdot \rletter{i}^{-}) = \Psi^{-}_{i}\big(\ncperm(\wh{F})\big)
\quad\text{and}\quad
\ncperm(\wh{F} \cdot \rletter{i}^{+}) = \Psi^{+}_{i}\big(\ncperm(\wh{F})\big).
    \end{align}
We check these by comparing the definition of $\wh{F}\cdot \rletter{i}^{\pm}$ with that of $\Psi^{\pm}_{i}$: the former is straightforward, and the latter follows from that fact that for a backwards cycle $C$ on a set $A$ containing $i$ but not $i+1$, the product $C(i\;i+1)$ is the backwards cycle on $A\sqcup \{i+1\}$. 
\end{proof}

We now characterize $I_{\wh{F}}$ in a manner that connects nested forests to factorizations of noncrossing partitions, as mentioned in Remark~\ref{rem:nstfor_to_NC}.
For each internal node $v \in \internal{\wh{F}}$, let \emph{$\tau_{v}$}  denote the transposition $(i\, j)$ of the rightmost leaf descendant $i$ of $v_{L}$ and the rightmost leaf descendant $j$ of $v$, as shown in Figure~\ref{fig:bnestfor_with_tau_v}.

\begin{figure}[!ht]
    \centering
    \includegraphics[scale=0.8]{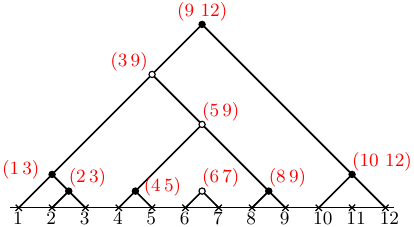}
    \caption{A bicolored nested forest with internal nodes labeled by transpositions $\tau_v$}
    \label{fig:bnestfor_with_tau_v}
\end{figure}

Say that a total order on $\internal{\wh{F}}$ is a \emph{linear extension} of $\wh{F}$ if each $v \in \internal{\wh{F}}$ is preceded by all of its ancestors.  
As we now explain, the $\tau_{v}$ have the property that any two product orders on
\begin{align}
\label{eq:tau_product}
\prod_{v \in \internal{\wh{F}}} \tau_{v}
\end{align}
which are linear extensions of $\wh{F}$ are related entirely by commuting factors past one another. Indeed, any two linear extensions are related by repeatedly swapping adjacent elements $v,v'$ which are not ancestor and descendant, and for such a pair the transpositions $\tau_v=(i,j)$ and $\tau_{v'}=(k,\ell)$ are disjoint and hence commute.

Thus \eqref{eq:tau_product} is unambiguously defined as the product of the vertices of $\wh{F}$ taken in any order dictated by a linear extension of $\wh{F}$. For the same reasons, one can define the product $\prod_{v \in S} \tau_{v}$ for any $S\subset\internal{\wh{F}}$.

\begin{thm}
\label{prop:cubical_factorization}
For each $\wh{F} \in \bnfor_{n}$, we have 
\[
I_{\wh{F}}  = \left\{ \prod_{v \in \internal{\wh{F}} \setminus S} \tau_{v} \suchthat S \subseteq \{\text{black vertices of $\wh{F}$}\}\right\}.
\]
  In particular, $\ncperm(\wh{F}) = \prod_{v \in \internal{\wh{F}}} \tau_{v}$.
\end{thm}
\begin{proof}
We first show the claim holds when $\wh{F}$ has only white nodes, so that $I_{\wh{F}} = \{\ncperm(\wh{F})\}$ by Theorem~\ref{thm:noncrossing_fixed_points} and the claim amounts to $\ncperm(\wh{F}) = \prod_{v\in \internal{\wh{F}}}\tau_v$. 
We proceed by induction on $|\internal{\wh{F}}|$.  
If $\wh{F}$ has no nodes, then the claim holds trivially.  
Otherwise let $v_{0}$ be the root of a tree in $\wh{F}$ that is not nested under any other tree, so that $v_{0}$ is the first element in some linear extension of $\wh{F}$.
Let $\wh{G}$ be the forest obtained by deleting $v_{0}$ and all incident edges.  
As $|\internal{\wh{G}}| < |\internal{\wh{F}}|$, our inductive hypothesis guarantees that $\ncperm(\wh{G}) = \prod_{v \in \internal{\wh{G}}} \tau_{v}$ is the unique element of $I_{\wh{G}}$.  
Moreover, $\ncperm(\wh{G}) = \prod_{v \in \internal{\wh{F}}-\{v_0\}} \tau_{v}$, since for $v \in \internal{\wh{G}}$ the value of $\tau_v$ does not depend on whether we consider $v$ as a node of $\wh{F}$ or $\wh{G}$.  
It therefore suffices to show that $\tau_{v_0}\ncperm(\wh{F})=\ncperm(\wh{G})$, which follows from that fact that if $c_A$ and $c_B$ are backwards cycles with $\max A<\min B$, then $(\max A, \max B)c_Ac_B=c_{A\sqcup B}$.

We now consider the general case of $\wh{F} \in \bnfor_{n}$.  
By Theorem~\ref{thm:noncrossing_fixed_points}, we have that $I_{\wh{F}} = \{ \ncperm(\wh{G}) \;|\; \wh{G} \in \Vert(\wh{F}) \}$.  
Further, each $\wh{G} \in \Vert(\wh{F})$ is obtained precisely by performing left edge deletion at some subset $S$ of black nodes from $\wh{F}$ and changing the remaining black nodes to white nodes.  
Thus applying the special case proved above to each $\wh{G} \in \Vert(\wh{F})$, we see that the theorem holds for $\wh{F}$.
\end{proof}

Recall that the moment polytope for each $X(\wh{F})$ is a combinatorial cube with vertices corresponding to the fixed point set $I_{\wh{F}}$.

\begin{cor}
\label{cor:sublattice}
For $\wh{F} \in \bnfor_{n}$, $I_{\wh{F}}$ is an induced Boolean sublattice in the Kreweras order, and $w \mapsto w \cdot \lambda$ maps $I_{\wh{F}}$ onto the vertices of the moment polytope of $X(\wh{F})$ for the dominant weight $\lambda$ in such a way that the Hasse diagram of $I_{\wh{F}}$ is identified with the $1$-skeleton of the polytope. 
\end{cor}
\begin{proof}
By~\Cref{thm:noncrossing_fixed_points}, $I_{\wh{F}} \subseteq \NC_{n}$. By~\cite[Lemma 2.11]{HS18}, this is an induced Boolean sublattice of the Kreweras order. 
Moreover, using the description of $I_{\wh{F}}$ given in Proposition~\ref{prop:cubical_factorization}, the edges of the moment polytope connect exactly those pairs of elements of $I_{\wh{F}}$ which differ by the inclusion of a single $\tau_{v}$. 
\end{proof}

\section{Translated Richardsons and polypositroids}
\label{sec:richardsons_and_polypositroids}

In this section we relate our $X(\wh{F})$ to certain Richardson varieties previously studied in~\cite{NST_a, NTremixed} and the quasisymmetric Schubert cycles of \cite{NST_c}; see \Cref{rem:qsymSchubert}. We then use this connection to describe the moment polytope of each $X(\wh{F})$ as a polypositroid ~\cite{LP20}.

Recall that a Richardson variety is the intersection $X^{v}_{w} = X^{v} \cap X_{w}$, which is nonempty if and only if $w \le v$.  
It is straightforward to see that $w \le w \cox$ if and only if $w\in S_{n-1}$, and in this case the Richardson variety $X^{w \cox}_{w}$ is known to be an $(n-1)$-dimensional toric variety \cite{NST_c}.
Consider now the image of each $X^{w \cox}_{w}$ under left multiplication by $w^{-1}$.   

\begin{thm}
\label{thm:samesetsrichardson}
We have
\[
\{X(T)\suchthat T\in \tree_n\}=\{w^{-1}X^{w\cox}_w\suchthat w\in S_{n-1}\}.
\]
In particular, there are $\cat{n-1}$ distinct translated Richardson varieties, one for each $T\in \tree_n$.
\end{thm}

For $T\in \operatorname{Tree}_n$ recall that we have
\[
X(T)=\mathbb{P}_{i_{n-1}}\mathbb{P}_{i_{n-2}}\cdots \mathbb{P}_{i_1}\{\point\}\subset \fl{n}
\]
for any sequence $\rletter{1}^-\eletter{i_{1}}\eletter{i_{2}}\cdots\eletter{i_{n-1}}$ associated to $T$. 
For any $m$, let $\varepsilon_i$ be the map from $S_{m-1}\to S_m$ which, in one-line notation, inserts a $1$ into the $i$th position and increases the remaining numbers by $1$. 
For example $\varepsilon_3 15684237=261795348$. 
Note that this coincides with the map $\Psi_{1,i}^{-}$ restricted to permutation matrices.
However, unlike $\Psi_{1,i}^-$, we shall reserve $\varepsilon_i$ for use on permutations only.

\begin{lem}
\label{lem:Yuv}
    Let $Y(u,v)\coloneqq u^{-1}X^v_u$. Then $\mathbb{P}_iY(u,v)=Y({\varepsilon_i u},{\varepsilon_{i+1}v}).$
\end{lem}  
\begin{proof}
    In \cite[\S 4]{NST_c} it was shown that $\pi_i^{-1}\pi_i\Psi_{1,i}X^v_u=X^{\varepsilon_{i+1}v}_{\varepsilon_i u}$.
    Since $\pi_i$ is equivariant with respect to left multiplication, we get the following sequence of equalities
    \begin{align*}
    Y(\varepsilon_i u,\varepsilon_{i+1}v)=&(\varepsilon_i u)^{-1}X_{\varepsilon_i u}^{\varepsilon_{i+1} v}=(\varepsilon_i u)^{-1}\pi_i^{-1}\pi_i \Psi_{1,i}X^v_u\\
    =&\pi_i^{-1}\pi_i (\varepsilon_i u)^{-1}\Psi_{1,i}X^v_u=\pi_i^{-1}\pi_i \Psi_{i}^- u^{-1}X^v_u=\mathbb{P}_i Y(u,v),
    \end{align*}
    where we use the fact that $(\varepsilon_i u)^{-1}\Psi_{1,i}=\Psi_{i}^- u^{-1}$.
\end{proof}

\begin{proof}[Proof of \Cref{thm:samesetsrichardson}]
Suppose first that $T\in \tree_n$ is associated to the sequence $\rletter{1}^-\eletter{i_{1}}\cdots \eletter{i_{n-1}}\in \reseq_n$.
    Let $v= \varepsilon_{i_{n-1}+1}\cdots \varepsilon_{i_1+1}\idem_{S_1}$ and $u=\varepsilon_{i_{n-1}}\cdots \varepsilon_{i_1}\idem_{S_1}$. 
    By \cite[Proposition B.4(2)]{NST_c} we have $v=u\cox$ and  $u(n)=n$.
    By repeated applications of \Cref{lem:Yuv} we have $X(T)=u^{-1}X_u^{v}=u^{-1}X_u^{u\cox}$.
    
    Conversely given $u\in S_{n-1}$, by induction one can show that $u=\varepsilon_{i_{n-1}}\cdots \varepsilon_{i_1}\idem_{S_1}$ for some sequence $i_{1},\ldots,i_{n-1}$ with $i_j\le j$, and the tree $T\in \tree_n$ associated to the sequence $\rletter{1}^-\eletter{i_{1}}\cdots \eletter{i_{n-1}}\in \reseq_n$ has $X(T)=u^{-1}X^{u\cox}_u$.

    Showing that there are $\cat{n-1}$-many distinct translated Richardson varieties amounts to showing that the $X(T)$ for $T\in \tree_n$ are distinct.  
    This follows either from the identification with the quasisymmetric Schubert cycles of \cite{NST_c} as described in \Cref{rem:qsymSchubert}.
    A second proof can be obtained using the results of  Section~\ref{sec:qsymvar}: we characterize when two forests in $\bnfor_n$ produce the same torus-orbit closure, and in particular show that this does not occur for any two trees.
\end{proof}

\begin{rem}
\label{rem:qsymSchubert}
    In \cite{NST_c} certain translates $u^{-1}X^v_u$ of Richardson varieties called \emph{quasisymmetric Schubert cycles} were defined for any forest $F\in \indexedforests_n$. The description of $X(T)$ for $T=\rletter{1}^-\eletter{i_1}\cdots \eletter{i_{n-1}}\in \tree_{n}$ as a translated Richardson variety is exactly the same as the description of the quasisymmetric Schubert cycle associated to $T$ in \cite{NST_c} (matching the notation, $T$ would have been described as associated to a sequence $\rletter{1}\tletter{i_1}\cdots \tletter{i_{n-1}}\in \rtseq_n$). More generally for $F\in \indexedforests_n$, applying \Cref{lem:Yuv} recursively to $X((\rletter{1}^-)^{n-k})=X^{\idem_{S_{n-k}}}_{\idem_{S_{n-k}}}\subset \fl{n-k}$ realizes each $X(F)$  
    as the quasisymmetric Schubert cycle associated to $F$.
\end{rem}

Recall that faces of Bruhat interval polytopes are themselves Bruhat interval polytopes.
Since the torus orbit closures in a fixed toric Richardson variety correspond to faces of the associated Bruhat interval polytope, we infer \cite[Proposition 7.12]{TW15} that every torus orbit closure in a toric Richardson variety is also a toric Richardson variety.
This yields the following corollary.

\begin{cor}
    Every $X(\wh{F})$ is the left-translate of a toric Richardson variety by an element of $S_n$.
\end{cor}

\begin{rem}
    \Cref{thm:samesetsrichardson} provides an alternate perspective on the presence of noncrossing partitions arising as torus fixed points of $X(\wh{F})$. 
    Indeed, consider a fixed point $u \in I_{T}$ for $T \in \tree_{n}$.  
    Using the fact that $X(T)$ is a translated Richardson variety we will show that $u \in \NC_{n}$.  
    To begin, choose a maximal chain from $w$ to $w\cox$ containing $wu$ in the Bruhat order.
    Left translation by $w^{-1}$ gives a factorization of $\cox$ as a product of transpositions $\tau_1\cdots \tau_{n-1}$ with $u=\tau_1\cdots \tau_{m}$ for $m=\ell(u)-\ell(w)$. 
    Hence $u\in \NC_n$ by the characterization of $\NC_{n}$ due to Biane \cite{Bi97}. 
   \end{rem}

The description of each $X(T)$ as a translated Richardson variety also leads to a description of the defining hyperplanes for the moment polytope of $X(T)$.  
Recall the canonical labelling of $T$ given in Section~\ref{sec:forest_prelims} and suppose that $i$ is the label of an internal node.  
Let $\operatorname{Right}(T,i)$ (resp. $\operatorname{Left}(T,i)$) denote the set containing $i$ and the labels of each internal node in the right (resp. left) subtree of $i$.  
These sets are necessarily intervals of $\NN$ containing $i$.

Recall that a polytope is a generalized permutahedron if its edges are parallel to vectors of the form $e_i-e_j$ for distinct $i$ and $j$, and that a polytope is an alcoved polytope if its facet normals are parallel to vectors of the form $e_i+e_{i+1}+\cdots+e_j$ for $i\leq j$.  
Following~\cite{LP20}, a \emph{polypositroid} is a polytope that is both a generalized permutahedron and an alcoved polytope.  
The following is essentially the content of \cite[Remark 6.11]{NTremixed}.
\begin{thm}    
The moment polytope of $X(T)$ in the hyperplane $z_1+\cdots+ z_n=\lambda_1+\cdots +\lambda_n$ is the polypositroid defined by the following inequalities: for each $i\in \{1,\dots,n-1\}$ we have
    \begin{align*}
        \sum_{j\in \operatorname{Right}(T,i)}z_j\geq \sum_{j\in \operatorname{Right}(T,i)}\lambda_{j+1},\quad  \text{ and } 
        \sum_{j\in \operatorname{Left}(T,i)}z_j\leq \sum_{j\in \operatorname{Left}(T,i)}\lambda_{j}.
    \end{align*}
\end{thm}
As faces of polypositroids are polypositroids, we have in fact showed that the moment polytope of every $X(\wh{F})$ is a polypositroid. 
By \Cref{mainthm:plucker} (proved in \Cref{sec:PluckerCharacterization}), the $X(\wh{F})$ are the only irreducible subvarieties of $\fl{n}$ whose torus fixed points are contained in $\NC_n$.  
Thus the moment polytopes of our $X(\wh{F})$ account for all flag matroid polytopes which have vertices in $\NC_{n}$ and moreover have ``geometric origin.''

\begin{conj}
    Every polypositroid---and more generally every flag matroid polytope---whose vertices are contained in $\NC_n$ is the moment polytope of some $X(\wh{F})$.
\end{conj}

\begin{eg}
Figure~\ref{fig:polypositroid_eg} (left) depicts $T\in\operatorname{Tree}_4$ with the canonical labeling of $\internal{T}$. 
On the right is the facet description inside the hyperplane $z_1+z_2+z_3+z_4=\lambda_1+\lambda_2+\lambda_3+\lambda_4$.

\begin{figure}[!ht]
    \centering
    \includegraphics[scale=1]{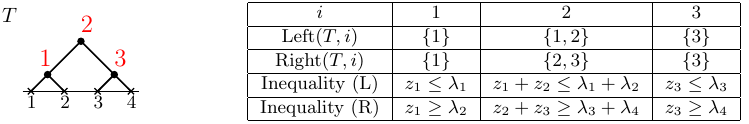}
    \caption{The facet inequalities for a particular moment polytope}
    \label{fig:polypositroid_eg}
\end{figure}

By \Cref{prop:cubical_factorization} and \Cref{cor:sublattice}, the set of vertices of this polytope is given by
\[
\{u\cdot \lambda \suchthat u\text{ a subword of the product }(2,4)(3,4)(1,2)\}.
\]
\end{eg}

\section{The quasisymmetric flag variety}
\label{sec:qsymvar}

 We define the \emph{quasisymmetric flag variety} as the toric complex 
 \[
 \qfl_n\coloneqq \bigcup_{T\in \operatorname{Tree}_n} X(T)\subset \fl{n}.
 \]
The union defining $\qfl_{n}$ is not disjoint as there is some overlap between  distinct $X(T)$, $X(T')$. In this section we characterize this overlap in terms of the torus fixed point sets $I_{\wh{F}}$ described in Section~\ref{sec:torusfixedpoints}.

\begin{thm}
\label{cor:containequiv}
For $\wh{F},\wh{G}\in \bnfor_n$ we have $X(\wh{G})\subset X(\wh{F})$ if and only if $I_{\wh{G}}\subset I_{\wh{F}}$.
\end{thm}

The criterion therein provides a combinatorial model for the toric structure of $\qfl_{n}$, see \Cref{fig:led_poset_kreweras_eg}.
First take the disjoint union of the moment polytopes for each $X(T)$, which we showed in Section~\ref{sec:BottManifold} are $(n-1)$-dimensional combinatorial cubes. 
Then create a polyhedral complex \emph{$\cqfl{n}$} by identifying faces from distinct polytopes that are equal in the sense that they share the same set of vertices. 
After identification, the faces of $\cqfl{n}$ are then in bijection with the distinct torus orbit closures in $\qfl_n$. 

\begin{figure}[!ht]
    \centering
    \includegraphics[width=\textwidth]{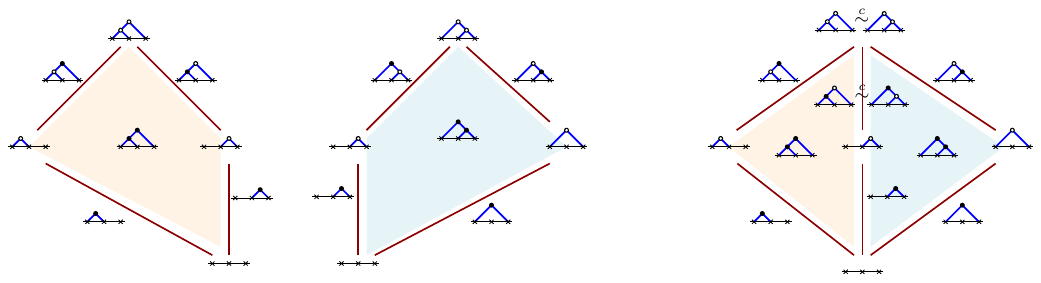}
    \caption{The combinatorial cubes corresponding to the toric orbit closures in each of the two components of $\qfl_3$ (left) and the global complex $\cqfl{3}$ encoding the inclusion order on all toric closures $X(\wh{F})$ in $\qfl_3$ (right).
    }
    \label{fig:led_poset_kreweras_eg}
\end{figure}

We prove Theorem~\ref{cor:containequiv} at the end of Section~\ref{sec:bruhat_max} after introducing an important equivalence relation on $\bnfor_n$ in Section~\ref{sec:coloredtamari}.
Section~\ref{sec:combinatorics_complex} contains further enumerative and structural results about $\cqfl{n}$.

\subsection{Colored Tamari equivalence and normal forms}
\label{sec:coloredtamari}

Every torus orbit closure in $\qfl_{n}$ is by definition contained in $X(T)$ for some $T \in \tree_{n}$. 
In Section~\ref{sec:BottManifold}, we showed every torus orbit closure in $X(T)$ is of the form $X(\wh{F})$ for a bicolored nested forest $\wh{F} \le_{re} T$. 
Thus the torus orbit closures in $\qfl_{n}$ can be parametrized by $\bnfor_n$.  However, there is some redundancy in this parametrization as is apparent from Figure~\ref{fig:led_poset_kreweras_eg}. 
This is explained by the following two additional relations that the building operations satisfy.

\begin{lem}
\label{lem:2additionalbuilding}
For all $1 \le i < n$ we have the relations
\[
\Psi_{i+1}^+\PP_{i}= \PP_{i}\Psi_{i}^+\quad \text{ and }\quad \Psi_{i+1}^+\Psi_{i}^+=\Psi_{i}^+\Psi_{i}^+.
\]
\end{lem}
\begin{proof}
Both relations can be verified with elementary matrix computations.  
\end{proof}

In the correspondence between words in $\rtseq_n$ and compositions of building operations, these relations correspond to $\eletter{i} \rletter{i+1}^+=\rletter{i}^+ \eletter{i}$ and $\rletter{i}^+ \rletter{i+1}^+=\rletter{i}^+ \rletter{i}^+$. 
Considering the relations at the level of binary forests leads to the following definition.

\begin{defn}
\label{fig:right_rotation}
We say that $\wh{F}, \wh{G}\in \bnfor_n$ are \emph{colored Tamari equivalent}, denoted by $\wh{F}\ctam \wh{G}$, if one can be transformed into the other by a sequence of \emph{colored Tamari rotations} shown below.
\begin{center}
\includegraphics[scale=1]{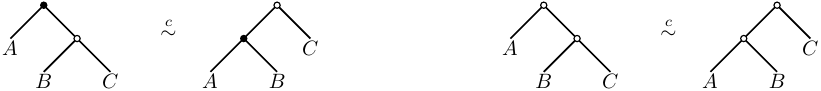}
\end{center}
\end{defn}

By the preceding discussion, we get the following result.

\begin{prop}
\label{prop:2additionalbuilding}
If $\wh{F},\wh{G}\in \bnfor$ satisfy $\wh{F}\ctam \wh{G}$, then $X(\wh{F})=X(\wh{G})$, and in particular  $I_{\wh{F}}=I_{\wh{G}}$.
\end{prop} 

\begin{defn}
    We say that $\wh{F} \in \bnfor_{n}$ is in \emph{normal form} if every right child in $\internal{\wh{F}}$ is a black node.  
Let \emph{$\bnfornf_n$} be the set of bicolored nested forests that are in normal form. 
\end{defn}

Every element of $\bnfor_n$ can be transformed to some element of $\bnfornf_n$ by applying colored Tamari rotations repeatedly. We will prove that this normal form is unique in the next section. Figure~\ref{fig:refornormalform} depicts a bicolored tree as well as its colored Tamari equivalent normal form.

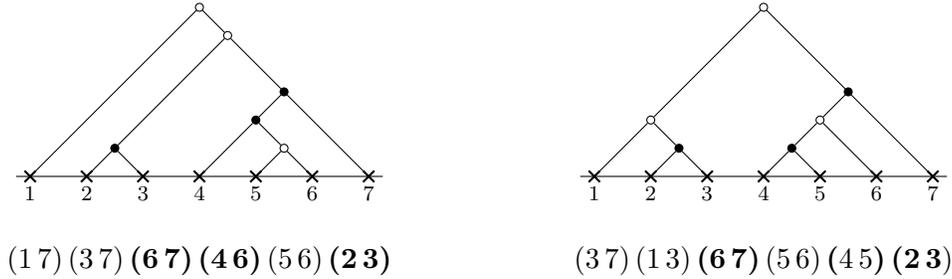
\begin{figure}[ht!]
    \centering
    \begin{tikzpicture}[scale = 0.75, baseline = 0.75*-0.1cm]
    \draw[thin] (0.75, 0) -- (7.25, 0);
    \foreach \x in {1, 2, 3, 4, 5, 6, 7}{
    	\draw[thick] (\x - 0.1, 0 - 0.1) -- (\x + 0.1, 0 + 0.1);
    	\draw[thick] (\x - 0.1, 0 + 0.1) -- (\x + 0.1, 0 - 0.1);
    	\draw (\x, 0) node[inner sep = -2pt]  (\x) {};
    	\draw (\x) node[below] {$\scriptstyle \x$};}
    \draw (4, 3)  circle (2pt) node[inner sep = 1pt] (l1) {};
    \draw[fill] (2.5, 0.5)  circle (2pt) node[inner sep = 1pt] (l2) {};
    \draw  (4.5, 2.5) circle (2pt) node[inner sep = 1pt] (l3) {};
    \draw[fill] (5, 1)  circle (2pt) node[inner sep = 1pt] (l4) {};
    \draw (5.5, 0.5)  circle (2pt) node[inner sep = 1pt] (l5) {};
    \draw[fill] (5.5, 1.5)  circle (2pt) node[inner sep = 1pt] (l6) {};
    \foreach \a/\b in {l1/1, l1/l3, l3/l2, l3/l6, l2/2, l2/3, l6/l4, l6/7, l4/4, l4/l5, l5/5, l5/6} {\draw[thin] (\a) -- (\b);}
    \draw (4, -1.5) node {$(1\,7)\,(3\,7)\,\mathbf{\boldsymbol(6\,7\boldsymbol)}\,\mathbf{\boldsymbol(4\,6\boldsymbol)}\,(5\,6)\,\mathbf{\boldsymbol(2\,3\boldsymbol)}$};
    \begin{scope}[xshift = 10cm]
    \draw[thin] (0.75, 0) -- (7.25, 0);
    \foreach \x in {1, 2, 3, 4, 5, 6, 7}{
    	\draw[thick] (\x - 0.1, 0 - 0.1) -- (\x + 0.1, 0 + 0.1);
    	\draw[thick] (\x - 0.1, 0 + 0.1) -- (\x + 0.1, 0 - 0.1);
    	\draw (\x, 0) node[inner sep = -2pt]  (\x) {};
    	\draw (\x) node[below] {$\scriptstyle \x$};}
    \draw (2, 1)  circle (2pt) node[inner sep = 1pt] (l1) {};
    \draw[fill] (2.5, 0.5)  circle (2pt) node[inner sep = 1pt] (l2) {};
    \draw (4, 3)  circle (2pt) node[inner sep = 1pt] (l3) {};
    \draw[fill] (4.5, 0.5)  circle (2pt) node[inner sep = 1pt] (l4) {};
    \draw (5, 1)  circle (2pt) node[inner sep = 1pt] (l5) {};
    \draw[fill] (5.5, 1.5)  circle (2pt) node[inner sep = 1pt] (l6) {};
    \foreach \a/\b in {l3/l1, l3/l6, l1/1, l1/l2, l2/2, l2/3, l6/l5, l6/7, l5/6, l5/l4, l4/4, l4/5} {\draw[thin] (\a) -- (\b);}
    \draw (4, -1.5) node {$(3\,7)\,(1\,3)\,\bm{(6\,7)}\,(5\,6)\,\boldsymbol(4\,5\boldsymbol)\,\mathbf{\boldsymbol(2\,3)\boldsymbol}$};
    \end{scope} 
    \end{tikzpicture} 
    \caption{A bicolored tree which is not in normal form (left), its Tamari-equivalent normal form (right), and the associated factorizations of $\cox$ for each tree (below).}
    \label{fig:refornormalform}
\end{figure}

\begin{rem}
Proposition~\ref{prop:cubical_factorization} relates bicolored nested forests to factorizations of noncrossing partitions, where the colored Tamari equivalence allows certain relations $(i\,k)(j\,k)=(j\,k)(i\,j)$.
\end{rem}

\subsection{The Bruhat maximal element of $X(\wh{F})$ and uniqueness of normal form}
\label{sec:bruhat_max}

We will need to understand the Bruhat order on $I_{\wh{F}}$. 
Since $X(\wh{F})$ is a torus orbit closure in $\fl{n}$, we have that $I_{\wh{F}}$ is a flag matroid~\cite{GGMS87}. 
Hence we have the following fact.

\begin{fact}[{\cite[\S~1.9]{BGW}}]
\label{fact:coxetermax}
  For any $\wh{F}\in\bnfor_n$, $I_{\wh{F}}$ has a unique Bruhat-maximum element. This element is characterized by the property that all adjacent vertices are lower in the Bruhat order.
\end{fact}

\newcommand{\rightchildren}{\operatorname{RC}}

In order to describe this distinguished element, we define a new map. For  $\wh{F}$ in $\bnfornf_n$, let $\rightchildren(\wh{F})$ denote the set of internal nodes that are right children. By definition of normal form, we note that all nodes in $\rightchildren(\wh{F})$ are black.
We define a map
\[
\begin{array}{rccc}
\ForToNC: &\bnfornf_n&\to& \NC_{n} \\
& \wh{F} & \mapsto & \ncperm\Big(\wh{F} \setminus \rightchildren(\wh{F})\Big)
\end{array}
\]
where $\wh{F} \setminus \rightchildren(\wh{F})$ denotes the bicolored nested forest obtained by left edge deletion for each node $v \in \rightchildren(\wh{F})$  and contracting it in the sense of~\Cref{prop:bicoloredchange}.  
This is a natural extension of the map on $\indexedforests_n$ which we defined in~\cite[Section 7.2]{BGNST1} using the same notation.

\begin{eg} We have
\[
\ForToNC\left( \begin{tikzpicture}[scale = 0.5, baseline = 0.5*1cm]
\draw[thin] (0.75, 0) -- (7.25, 0);
    \foreach \x in {1, 2, 3, 4, 5, 6, 7}{
    	\draw[thick] (\x - 0.1, 0 - 0.1) -- (\x + 0.1, 0 + 0.1);
    	\draw[thick] (\x - 0.1, 0 + 0.1) -- (\x + 0.1, 0 - 0.1);
    	\draw (\x, 0) node[inner sep = -2pt]  (\x) {};
    	\draw (\x) node[below] {$\scriptstyle \x$};}
    \draw (2, 1)  circle (2pt) node[inner sep = 0pt] (l1) {};
    \draw[fill] (2.5, 0.5)  circle (2pt) node[inner sep = 0pt] (l2) {};
    \draw (4, 3)  circle (2pt) node[inner sep = 0pt] (l3) {};
    \draw[fill] (4.5, 0.5)  circle (2pt) node[inner sep = 0pt] (l4) {};
    \draw (5, 1)  circle (2pt) node[inner sep = 0pt] (l5) {};
    \draw[fill] (5.5, 1.5)  circle (2pt) node[inner sep = 0pt] (l6) {};
    \foreach \a/\b in {l3/l1, l3/l6, l1/1, l1/l2, l2/2, l2/3, l6/l5, l6/7, l5/6, l5/l4, l4/4, l4/5} {\draw[thin] (\a) -- (\b);}
\end{tikzpicture} \right)
= \ncperm\left( \begin{tikzpicture}[scale = 0.5, baseline = 0.5*1cm]
    \draw[thin] (0.75, 0) -- (7.25, 0);
    \foreach \x in {1, 2, 3, 4, 5, 6, 7}{
    	\draw[thick] (\x - 0.1, 0 - 0.1) -- (\x + 0.1, 0 + 0.1);
    	\draw[thick] (\x - 0.1, 0 + 0.1) -- (\x + 0.1, 0 - 0.1);
    	\draw (\x, 0) node[inner sep = -2pt]  (\x) {};
    	\draw (\x) node[below] {$\scriptstyle \x$};}
    \draw (2, 1)  circle (2pt) node[inner sep = 0pt] (l1) {};
    \draw (4, 3)  circle (2pt) node[inner sep = 0pt] (l3) {};
    \draw[fill] (4.5, 0.5)  circle (2pt) node[inner sep = 0pt] (l4) {};
    \draw (5, 1)  circle (2pt) node[inner sep = 0pt] (l5) {};
    \foreach \a/\b in {l3/l1, l3/7, l1/1, l1/3, l5/6, l5/l4, l4/4, l4/5} {\draw[thin] (\a) -- (\b);}
\end{tikzpicture}  \right) = (731)(654)(2).
\]
\end{eg}

\begin{lem}
\label{lem:ForToNcPsi}
For $\wh{F} = \wh{G} \cdot \xletter{i} \in \bnfornf_{n}$ with $\wh{G}\in \bnfornf_{n-1}$, we have 
\[
\ForToNC(\wh{F} ) = \begin{cases}
\Psi_i^-\ForToNC(\wh{G})&\text{if $\xletter{i} = \rletter{i}^{-}$ or leaf $i$ is a right child in $\wh{G}$}\\
\Psi_i^+\ForToNC(\wh{G})&\text{otherwise.}\end{cases}
\]
\end{lem}
\begin{proof}
From the definitions we immediately verify 
\[
\wh{F} \setminus \rightchildren(\wh{F}) = \begin{cases}
(\wh{G} \setminus \rightchildren(\wh{G})) \cdot \rletter{i}^{-}
&\text{if $\xletter{i} = \rletter{i}^{-}$ or leaf $i$ is a right child in $\wh{G}$}\\
(\wh{G} \setminus \rightchildren(\wh{G})) \cdot \xletter{i} &\text{otherwise}\end{cases}
\]
after which the formulas~\eqref{eq:ncperm_recurse} complete the proof.   
\end{proof}

\begin{prop}
\label{prop:ForToNCBruhatMax}
    If $\wh{F} \in \bnfornf_n$, then $\ForToNC(\wh{F})$ is the Bruhat-maximum element of $I_{\wh{F}}$.
\end{prop} 
\begin{proof} 
We proceed by induction on $n$.
For $n = 1$, $\ForToNC(\wh{F})$ is the only element of $I_{\wh{F}}$.  
For $n > 1$, we have $\wh{F} = \wh{G} \cdot \xletter{i}$ for $\xletter{i} \in \{\rletter{i}^{\pm}, \eletter{i}\}$ and $\wh{G}\in \bnfornf_{n-1}$.  
Setting $w = \ForToNC(\wh{F})$ and $u = \ForToNC(\wh{G})$, Lemma~\ref{lem:ForToNcPsi} states that $w = \Psi^{\epsilon}_{i}(u)$ for $\epsilon \in \{+, -\}$.  

For $\xletter{i} = \rletter{i}^{\epsilon}$, we note that both $\Psi_{i}^{-}$ and $\Psi_{i}^{+}$ preserve the Bruhat order.  
This follows, for instance, from the tableau criterion \cite[Theorem 2.6.3]{BjBr05}.
Thus resorting to our inductive hypothesis on $u$, we have that $w = \Psi^{\epsilon}_{i}(u)$ is the Bruhat maximum of $I_{\wh{F}} = \Psi^{\epsilon}_{i}(I_{\wh{G}})$.

Now suppose that $\xletter{i} = \eletter{i}$. 
By \Cref{fact:coxetermax}, it suffices to show that $w$ is greater than all adjacent fixed points in $I_{\wh{F}}$. 
As $I_{\wh{F}}$ is a combinatorial cube and $\Psi^{\epsilon}_{i}$ is a face inclusion, all but one of these adjacent elements are contained in $\Psi^{\epsilon}_{i}(I_{\wh{G}})$ and therefore covered by the previous argument. 
The remaining adjacent fixed point is  $ws_{i}$, so what remains is to show that $w(i) > w(i+1)$.  
Let $v$ be the (black) node associated to $\xletter{i}$.  
If $v$ is a right child in $\wh{F}$, then $w(i) = i$ and $w(i+1) < i+1$.  
If $v$ is not a right child, then $i$ is the smallest element of its cycle, which must also contain $i+1$, so again $w(i) > w(i+1)$.  \end{proof}

We need an alternative construction related to the map $\ForToNC$.
Let $(F,S)$ be a pair consisting of an indexed forest $F\in\indexedforests_n$ and a subset $S \subseteq \internal{F}$. We define $$\wh{(F,S)}\in \Face(F)\subset \bnfornf_n$$ by performing left-edge deletion on all vertices in $S\cap \rightchildren(F)$, and coloring the remaining vertices of $S$ white.
An example is shown in Figure~\ref{fig:ForToNC_cell} (left).

\begin{figure}[t]
    \centering
    \includegraphics[width=\textwidth]{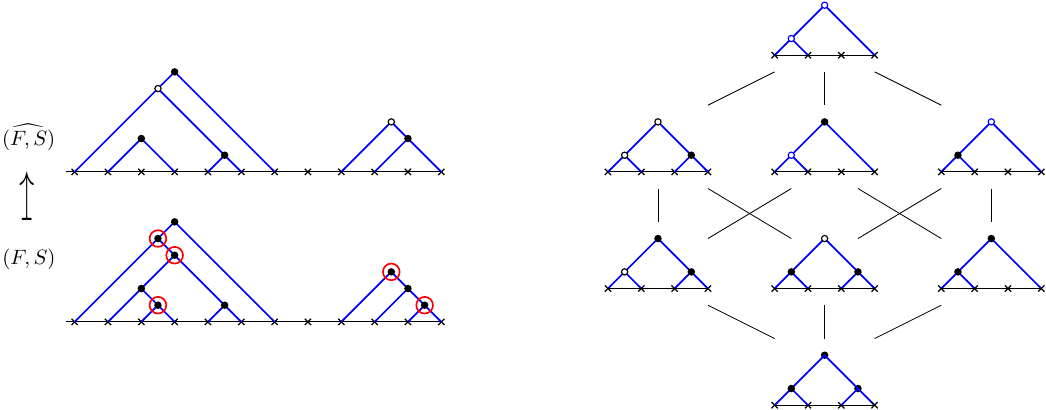}
    \caption{An example of the construction in Proposition~\ref{prop:bnfornf_construction} for $n=12$ (left) and  all elements  $\wh{(F,S)}\in \bnfornf_4$ for $F = \wh{F}(\rletter{1}^{-}\eletter{1}\eletter{1}\eletter{3}) \in \indexedforests_{4}$ (right).}
    \label{fig:ForToNC_cell}
\end{figure}

\begin{prop}
\label{prop:bnfornf_construction}
The map $(F,S)\mapsto \wh{(F,S)}$ is a bijection from $\{(F, S) \suchthat F \in \indexedforests_{n}, S \subseteq \internal{F}\}$ onto $\bnfornf_n$.  Furthermore $\ForToNC\wh{(F,S)}=\ForToNC(F)$ for all $(F, S)$.
\end{prop}

\begin{proof}
The fact that 
$\ForToNC\wh{(F,S)}=\ForToNC(F)$ follows from the definition of $\ForToNC$. For fixed $F$ the forests $\wh{(F,S)}$ are distinct as they correspond to distinct elements of $\Face(F)$: given a fixed sequence $(\rletter{1}^-)^{n-k}e_{i_1}\cdots e_{i_k}$ for $F$ the choice of $S$ determines the subset of $e_{i_1},\ldots,e_{i_k}$ to transform to $\rletter{i}^{\pm}$. This shows that the map $(F,S)\mapsto \wh{(F,S)}$ is an injection, so it remains to construct an inverse map.  
Given $\wh{G}$ in $\bnfornf_n$ we create $G\in \indexedforests_n$ by coloring all of its vertices black and then connecting the remaining nested trees using the procedure described below.  

For each $v\in \internal{\wh{G}}$, let $T_1,\ldots,T_k$ be the outermost trees which are nested in the subtree below $v$, listed from left to right. Denoting their roots $w_1,\ldots,w_k$ we create new black nodes $v_1',\ldots,v_k'$ in the interior of the edge from $v$ to $v_R$ in this order, and then for each $i$ we connect $w_i$ to $v_i'$.

The resulting forest $G$ does not depend on the order in which we apply the connection procedure. 
Setting $S \subseteq \internal{G}$ to be set consisting of the newly created nodes together with the black nodes that came from white nodes of $\wh{G}$, $\wh{G}\mapsto (G,S)$ is the inverse map. 
\end{proof}

\begin{prop}
\label{prop:equalIimpliesequalF} 
    Let $\wh{F},\wh{G}\in\bnfornf_{n}$.  If $I_{\wh{F}}=I_{\wh{G}}$  then $\wh{F}=\wh{G}$.
\end{prop}

\begin{proof}
We show that we can reconstruct $\wh{F}$ from $I=I_{\wh{F}}$. 
First, we recover
 $w=\ForToNC(\wh{F})$ as the maximal element in $I$ with respect to the Bruhat order by \Cref{prop:ForToNCBruhatMax}.
 By \Cref{prop:bnfornf_construction}, this means $\wh{F}=\wh{(F,S)}$ for the unique $F \in \indexedforests_{n}$ with $\ForToNC(F) = w$ and some unique $S\subset \internal{F}$ so it remains to show that the $I_{\wh{(F,S)}}$ are distinct for fixed $F$. Indeed, by \Cref{prop:bnfornf_construction} again, as we vary $S$ we obtain distinct $\wh{(F,S)}\in \Face(F)$, so we conclude the vertex sets $I_{\wh{(F,S)}}$ are distinct as a face is determined by its vertex set.
\end{proof}

\begin{rem}
    One can show that $I_{\wh{(F,S)}}$ is the smallest sublattice of $I_F$ which contains $\ForToNC(F)$ and $\ncperm(\wh{(F,S)}\setminus A)$ for $A\subset \internal{\wh{(F,S)}}$ the subset of black nodes that are not right children.
\end{rem}

\begin{prop}
\label{prop:varietydeterminedbytamari}
Every colored Tamari equivalence class in $\bnfor_n$ has a unique normal form representative in $\bnfornf_n$. Moreover, for $\wh{F},\wh{G}\in \bnfor_n$, the following are equivalent:
    \begin{enumerate}[label=(\arabic*)]
        \item \label{it:equivalent_1} $X(\wh{F})=X(\wh{G})$,
        \item \label{it:equivalent_2} $\wh{F}\ctam \wh{G}$, and
        \item \label{it:equivalent_3} $I_{\wh{F}}=I_{\wh{G}}$.
    \end{enumerate}
\end{prop}
\begin{proof}
 If $\wh{F}\ctam \wh{G}$ are both in normal form, then  by \Cref{prop:2additionalbuilding} we have $X(\wh{F})=X(\wh{G})$, and in particular $I_{\wh{F}}=I_{\wh{G}}$ and  so by \Cref{prop:equalIimpliesequalF} we conclude that $\wh{F}=\wh{G}$.

We know already that ~\ref{it:equivalent_2} implies~\ref{it:equivalent_1} by \Cref{prop:2additionalbuilding}, and obviously ~\ref{it:equivalent_1} implies~\ref{it:equivalent_3}. It remains only to show that~\ref{it:equivalent_3} implies~\ref{it:equivalent_2}. 
By \Cref{prop:2additionalbuilding} colored Tamari rotations preserve $I_{\wh{F}}$. Let $\wh{F}'$ and $\wh{G}'$ be the normal form colored Tamari equivalents to $\wh{F}$ and $\wh{G}$. Then $I_{\wh{F}'}=I_{\wh{G}'}$ and so by \Cref{prop:equalIimpliesequalF} we deduce that $\wh{F}'=\wh{G}'$, and so conclude that $\wh{F}\ctam \wh{G}$.
\end{proof}

We are now ready to prove \Cref{cor:containequiv}.
\begin{proof}[Proof of \Cref{cor:containequiv}]
Clearly $X(\wh{G}) \subset X(\wh{F})$ implies $I_{\wh{G}}\subset I_{\wh{F}}$. Suppose now $I_{\wh{G}}\subset I_{\wh{F}}$. By \Cref{cor:sublattice}, both $I_{\wh{G}}$ and $I_{\wh{F}}$ are boolean lattices inside the Kreweras lattice. 
Recall the absolute length of a noncrossing partition $w\in \NC_n$ is $n$ minus the number of cycles of $w$, which is the minimal number of transpositions needed to multiply to $w$. Let $u, v \in I_{\wh{G}}$ be the top and bottom elements in the Kreweras order, of absolute lengths $a$ and $b$ respectively, so that $|I_{\wh{G}}|=2^{a-b}$. On the other hand the set of permutations in $I_{\wh{F}}$ between $u$ and $v$ in the Kreweras order is a subinterval of $I_{\wh{F}}$ with $2^{a-b}$ elements. It follows that $I_{\wh{G}}$ is this subinterval of $I_{\wh{F}}$, so 
there exists $\wh{G}'\in \Face(\wh{F})$ with  $I_{\wh{G}}=I_{\wh{G}'}$.
By \Cref{prop:varietydeterminedbytamari} this implies $X(\wh{G}')=X(\wh{G}) \subset X(\wh{F})$.
\end{proof}

\subsection{Combinatorics of torus orbit closures}
\label{sec:combinatorics_complex}
In this section we describe several enumerative aspects of the complex $\cqfl{n}$.  
We moreover show in Proposition~\ref{prop:cqfl_genfun} that the number of such orbits for $n=1,2,3,\ldots$ is given by generalized Catalan numbers~\cite[\href{https://oeis.org/A064062}{A064062}]{oeis}, while counting them according to dimension gives a refinement known as \emph{Borel's triangle}~\cite[\href{https://oeis.org/A234950}{A234950}]{oeis}.

We now describe the cell structure of $\cqfl{n}$. 
In general the intersection of two torus orbit closures under inclusion is a union of one or more torus orbit closures. As shown on the right in Figure~\ref{fig:led_poset_kreweras_eg}, the intersection of the top-dimensional orbit closures in $\cqfl{3}$ is the (non-disjoint) union of two intervals.  
Theorem~\ref{cor:containequiv} shows that faces of $\cqfl{n}$ are in bijection with the fixed point sets $I_{\wh{F}} \subseteq \NC_{n}$ for $\wh{F}\in \bnfor_n$, with inclusion of faces corresponding to inclusion of sets. The faces are also in bijection with normal form forests in $\bnfornf_n$, but inclusion is harder to compute with these objects; in particular it is strictly stronger than the restriction of the order $\le_{re}$.

\begin{prop}
\label{prop:cqfl_genfun}
    Let $G(z,u)=\sum_{n,k\geq 0}f_{n,k}z^nu^k$ where $f_{n,k}$ is the number of torus orbits in $\qfl_n$ of dimension $k$. Then
    \begin{align}
    \label{eq:cqfl_genfun}
    G(z,u)=\frac{1+2u-\sqrt{1-4(u+1)z}}{2(z+u)}.
    \end{align}
\end{prop}

\begin{proof}
   We have that $f_{n,k}$ is the number of $\wh{F}\in\bnfornf_n$ with $k$ black nodes.  
   First let $G_{\cox}(z,u)$ be the generating function for trees $T \in \bnfornf_n$, so that $\ncperm(T)=\cox$.  
   Decomposition at the root gives a quadratic functional equation for $G_{\cox}(z,u)$ that has the solution
   \begin{align}
    \label{eq:cqfl_genfun_connected}
    G_{\cox}(z,u)=\frac{1+2u-z-\sqrt{1+z^2-2(2u+1)z}}{2u}.
    \end{align}  
  For the general case, note that each $\wh{F} \in \bnfornf_n$ is given by the choice of an element $w\in\NC_{n}$ and for each cycle $C = (c_{1}\,\cdots\,c_{k})$ of $w$, a tree in $\bnfornf_k$. 
  As in~\cite[\S 2.2]{ChapotonNadeau2017}, we can therefore apply the $R$-transform from free probability to obtain the equation $G(z,u)=G_{\cox}(zG(z,u),u)$. 
  This can be solved using the quadratic equation for $G_{\cox}(z,u)$, from which we obtain the desired result. 
\end{proof}

 The expression \eqref{eq:cqfl_genfun} is the generating function for Borel's triangle~\cite[\href{https://oeis.org/A234950}{A234950}]{oeis} whose entries have the explicit closed form
 \[f_{n,k}=\frac{1}{n}\binom{2n}{n-k-1}\binom{n+k-1}{k}.\]
 From the expression \eqref{eq:cqfl_genfun_connected}, it follows that the enumeration for trees is given by the classical \emph{large Schr\"{o}der numbers}~\cite[\href{https://oeis.org/A006318}{A006318}]{oeis} and the refined version according to $|\wh{F}|$ is \cite[\href{https://oeis.org/A088617}{A088617}]{oeis}.

\begin{rem}
    For comparison, every face of the complex attached to the complex $\operatorname{HHMP}_n$ discussed in the introduction is a cube, and intersections of faces are faces. 
    These are indexed bijectively by the words in $\reseq_n$ without any $\rletter{i}^{+}$, and a face $F_1$ contains a face $F_2$ if the word for $F_1$ can be transformed to that for $F_2$ by changing some letters $\eletter{i}$ to either $\rletter{i}^{-}$ or $\rletter{i+1}^{-}$. 
    The total number of faces is $1\cdot 3\cdot 5 \cdots (2n-1)$ \cite[\href{https://oeis.org/A001147}{A001147}]{oeis}, refined according to dimension by the generating polynomial $(1)(2+t)(3+2t)\cdots (n+(n-1)t)$. 
\end{rem}

\section{The affine paving of $\qfl_n$}
\label{sec:affinepaving}

We now describe a family of affine charts for $\qfl_{n}$ around each of its torus fixed points.  
We begin in Section~\ref{sec:paving1} by defining the charts in terms of the Bott manifold structure of the $X(\wh{F})$ and showing that they partition $\qfl_{n}$.  
Then in Section~\ref{sec:paving2} we explicitly construct each chart and show that our partition can be equivalently obtained by intersecting $\qfl_{n}$ with the Bruhat decomposition of $\fl{n}$.

\subsection{The paving}
\label{sec:paving1}

For each $F\in\indexedforests_n$, the Bott manifold structure of $X(F)$ from Section~\ref{sec:BottManifold} gives an affine chart $C(F)$ around the $T$-fixed point $\ForToNC(F)\in X(F)$. 
Explicitly, $C(F)$ is isomorphic to an affine space $\mathbb{A}^{|F|}$ of dimension $|F|$ and  decomposes into sub-torus-orbits of $X(F)$ as
\begin{align}
\label{eq:C_F}
C(F) = \bigsqcup_{\substack{\wh{H} \in \Face(F) \\ \ForToNC(F) \in I_{\wh{H}}}} X^{\circ}(\wh{H}).    
\end{align}

\begin{thm}
\label{thm:paving_Bott}
The affine charts form a partition of $\qfl_n$:
\begin{equation}
\label{eq:chart_decomposition}
\qfl_n = \bigsqcup_{F\in\indexedforests_n} C(F).
\end{equation}
Moreover, for any total ordering $F_1,F_2,\ldots,F_{|\NC_{n}|}$ of $\indexedforests_n$ that extends the pullback of the Bruhat order via $\ForToNC$, we have $\bigsqcup_{i=1}^{k} C(F_{i})=\bigcup_{i=1}^k X(F_i)$.
\end{thm}

We prove the theorem after the following remark and lemma.

\begin{rem}
\label{rem:complex_paving}
The decomposition in Theorem~\ref{thm:paving_Bott} can be interpreted as a partition of $\cqfl{n}$ by associating each $C(F)$ with the half-open subspace of the moment polytope for $X(F)$ around the vertex $\ForToNC(F)$; see Figure~\ref{fig:complex_paving}.
\end{rem}

\begin{figure}[!ht]
    \centering
    \includegraphics[height=6cm]{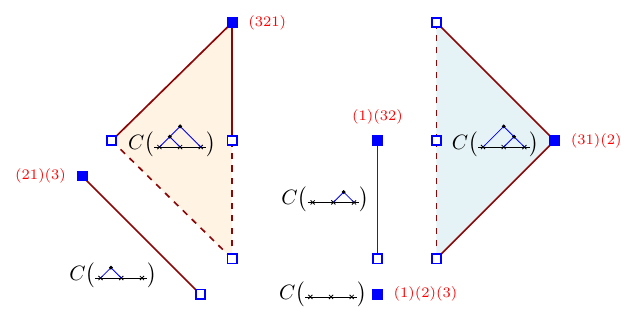}
    \caption{The decomposition of $\cqfl{3}$ induced by our affine paving of $\qfl_{3}$ as described in Remark~\ref{rem:complex_paving}; compare to Figure~\ref{fig:led_poset_kreweras_eg}.}
    \label{fig:complex_paving}
\end{figure}

\begin{lem}
\label{lem:charts_grouping}
Let $F\in\indexedforests_n$ and $\wh{H}\in\bnfornf_{n}$.  Then $\ForToNC(\wh{H})=\ForToNC(F)$ if and only if $\wh{H} \in \Face(F)$ and $\ForToNC(F) \in I_{\wh{H}}$ .
\end{lem}

\begin{proof}
If $\ForToNC(\wh{H})=\ForToNC(F)$, then $H$ corresponds to a pair of the form $(F,S)$ by the construction of \Cref{prop:bnfornf_construction}. It follows from the definition of $\le_{re}$ that $\wh{H} \in \Face(F)$ and likewise $\ForToNC(F)\in I_{\wh{H}}$. 
Conversely, if $\wh{H} \in\bnfornf_{n}$ satisfies these two conditions, we know that $I_{\wh{H}}\subset I_{F}$ by \Cref{cor:containequiv}. Since $\ForToNC(F)$ is the Bruhat-maximum element of $I_F$ by \Cref{prop:ForToNCBruhatMax}, it must be the Bruhat-maximum element of $I_{\wh{H}}$ as well, which implies $\ForToNC(\wh{H})=\ForToNC(F)$ by \Cref{prop:ForToNCBruhatMax} again.
\end{proof}

\begin{proof}[Proof of Theorem~\ref{thm:paving_Bott}]
By Proposition~\ref{prop:varietydeterminedbytamari}, we have
\begin{align}
    \label{eq:torus_decomposition}
    \qfl_n=\bigsqcup_{\wh{F}\in\bnfornf_n}X^{\circ}(\wh{F}).
\end{align}
By Lemma~\ref{lem:charts_grouping} and Equation~\eqref{eq:C_F}, we have
\begin{align}
\label{eq:C_F_bis}
C(F) = \bigsqcup_{\substack{\wh{H} \in \bnfornf \\ \ForToNC(\wh{H})=\ForToNC(F)}} X^{\circ}(\wh{H}).    
\end{align}
As $\ForToNC$ is surjective when restricted to $\indexedforests_{n}$, it follows immediately that Equation~\eqref{eq:C_F_bis} coarsens the partition in Equation~\eqref{eq:torus_decomposition} into the one in Equation~\eqref{eq:chart_decomposition}.

We now show that $\bigcup_{i=1}^k C(F_i)=\bigcup_{j=1}^k X(F_j)$ for any $k$. Since $C(F)\subset X(F)$ for any $F\in\indexedforests$, we only have to show that any $\mathcal{F}\in X(F_j)$ for some $j\leq k$ is included in $C(F_i)$ for some $i\leq k$. We have that $\mathcal{F}\in X^{\circ}(\wh{G})$ for some $\wh{G}\in\bnfornf_n$ such that $X^{\circ}(\wh{G})\subset X(F_j)$. By \Cref{cor:containequiv}, this implies that $I_{\wh{G}}\subset I_{F_j}$. In particular, using the characterization of \Cref{prop:ForToNCBruhatMax},  we have $\ForToNC(\wh{G})\le \ForToNC(F_j)$ in Bruhat order. By our choice of total order we then have $\ForToNC(\wh{G})=\ForToNC(F_i)$ for some $i\leq j$. Because of \eqref{eq:C_F_bis} we then have $X^{\circ}(\wh{G})\subset C(F_i)$, and thus $\mathcal{F}\in C(F_i)$, which concludes the proof.
\end{proof}

\begin{rem}[Points over $\mathbb{F}_q$]
The definition of $\qfl_n$ and all arguments used so far make sense over any field, not just $\mathbb{C}$, and so using \eqref{eq:chart_decomposition} and \eqref{eq:torus_decomposition} we may count the number of points of $\qfl_n$ over a finite field $\mathbb{F}_q$. 
    In this case $C(F)\simeq \mathbb{F}_q^{\internal{F}}$ has cardinality $q^{|F|}$. Summing over all $F\in\indexedforests_n$ we get
    \[\#\qfl_n(\mathbb{F}_q)=\sum_{F\in\indexedforests_n}q^{|F|}=\sum_{k=0}^{n-1}c_{n,k}q^k\]
    where $c_{n,k}=\frac{n-k}{n+k}\binom{n+k}{k}$ \cite{NT_forest}. Using \eqref{eq:torus_decomposition} we get the alternative expression
    \[\#\qfl_n(\mathbb{F}_q)=\sum_{F\in\bnfor_n}(q-1)^{|F|}=\sum_{k=0}^{n-1}f_{n,k}(q-1)^k\]
    where the numbers $f_{n,k}$ were introduced in \Cref{sec:combinatorics_complex}. These two expressions are polynomial in $q$, and thus one can extract $f_{n,k}=\sum_{m=k}^n\binom{m}{k}c_{n,m}$. This gives another proof that the numbers $f_{n,k}$ are given by Borel's triangle ~\cite[\href{https://oeis.org/A234950}{A234950}]{oeis} as seen in \Cref{sec:combinatorics_complex}.  
\end{rem}

\subsection{Paving with noncrossing Bruhat cells}
\label{sec:paving2}

We now give a combinatorial description of our affine paving for $\qfl_{n}$.  
Throughout, we use the convention that a matrix $M$ with entries in  $\CC \cup \{\ast, +\}$ represents the set of all matrices whose entries are, depending on the corresponding entry of $M$, either a particular complex number (if in $\CC$) or taken freely from either $\CC$ (if $\ast$) or $\CC^{\times}$ (if $+$).

We first recall the combinatorial construction of Bruhat cells in $\fl{n}$.  
The \emph{inversion set} of $w \in S_{n}$ is $\inv{w}=\{(i,j) \suchthat \text{$i<j$ and $w(i)>w(j)$}\}$.  
For $w \in S_{n}$, let \emph{$M(w)$} be the matrix with $1$'s in positions $(w(i), i)$, $\ast$'s in positions $(w(j),i)$ for $(i, j) \in \inv{w}$, and $0$'s elsewhere; see for example Figure~\ref{fig:f_to_mf}. 
Then $M(w)$ is isomorphic to an affine space where each $\ast$ represents a coordinate, and this gives a complete set of representatives for the Bruhat cell $BwB \subseteq \fl{n}$.  
In order to reproduce the standard action on $\fl{n}$, we have $T$ act on elements of $M(w)$ by scaling the $k, \ell$ entry by the character $\chi_{k} \chi_{w(\ell)}^{-1}$.  
Thus as a $T$-representation
\begin{equation}
\label{eq:bruhat_chars}
M(w)\cong \bigoplus_{(i,j)\in \inv{w}}\mathbb{C}_{\chi_{w(j)}\chi_{w(i)}^{-1}}.
\end{equation}

In order to state a similar result for $\qfl_{n}$, we introduce an important subset of the inversion set of a noncrossing partition. 

\begin{defn}
\label{defn:invnc}
The \emph{noncrossing inversion set} of $w \in \NC_{n}$ is
\[
\invnc{w}\coloneqq \{(i,j)\in \inv{w}:w(i,j)\in \NC_n\}.
\]
\end{defn}

\begin{defn}
The \emph{noncrossing Bruhat cell} for $w\in \NC_n$ is the set represented by the matrix $M_{\NC}(w)$ with entries $1$ in position $(w(i), i)$ for each $i \in [n]$, $\ast$ in position $(w(j),i)$ for each $(i,j)\in \invnc{w}$, and $0$ elsewhere.
For $F\in \indexedforests_n$, we define 
\[
M_{\NC}(F) \coloneqq M_{\NC}(\ForToNC(F)).
\]
\end{defn}

For $w = \ForToNC(F)$, we have a canonical identification
\[
M_{\NC}(w)\cong \bigoplus_{(i,j)\in \invnc{w}}\mathbb{C}_{\chi_{w(j)}\chi_{w(i)}^{-1}}\subset \bigoplus_{(i,j)\in \inv{w}}\mathbb{C}_{\chi_{w(j)}\chi_{w(i)}^{-1}}\cong M(w).
\]
See for example Figure~\ref{fig:f_to_mf}. We now state the main result of the section.

\begin{figure}[t]
\[
M(w)=\begin{bmatrix}
\ast&1&0&0&0&0\\
\ast&0&1&0&0&0\\
\ast&0&0&\ast&\ast&1\\
\ast&0&0&\ast&1&0\\
\ast&0&0&1&0&0\\
1&0&0&0&0&0
\end{bmatrix}
\hspace{6em}
M_{\NC}(w)=\begin{bmatrix}
\ast&1&0&0&0&0\\
\ast&0&1&0&0&0\\
\ast&0&0&\ast&0&1\\
0&0&0&\ast&1&0\\
0&0&0&1&0&0\\
1&0&0&0&0&0
\end{bmatrix}
\]
\caption{The Bruhat cell and noncrossing Bruhat cell for $w = 612543$.
}
\label{fig:f_to_mf}
\end{figure}

\begin{thm}
\label{thm:paving}
For $w\in \NC_n$ we have $\qfl_n\cap BwB=M_{\NC}(w)B$ and
\begin{equation}
\label{eq:hhmpequalsunion}
\qfl_n = \bigsqcup_{w\in \NC_n} M_{\NC}(w)B.
\end{equation}
Moreover, if for any total ordering $w_1,w_2,\ldots,w_{|\NC_{n}|}$ of $\NC_n$ that extends the Bruhat order we set
\[
X_k = \bigcup_{i=1}^{k} M_{\NC}(w_{i})B,
\]
then each $X_k$ is closed, $X_{1} \subseteq X_{2} \subseteq \cdots \subseteq X_{|\NC_{n}|} = \qfl_{n}$, and $X_{k+1}\setminus X_k=M_{\NC}(w_{k+1})B$.
\end{thm}

We will prove this result at the end of the section after some preparation. 

\begin{rem}
The second part of Theorem~\ref{thm:paving} shows that the $X_{i}$ form an affine paving of $\qfl_{n}$; see Section~\ref{sec:GKM_recall} for precise definitions.
\end{rem}

Our first step is to extend a combinatorial characterization of the noncrossing inversion set stated in~\cite[Remark 8.18]{BGNST1}.  
For $\wh{F} \in \bnfor_{n}$, the \emph{spread} of $v \in \internal{\wh{F}}$ is the pair $(i, j)$ consisting of the leftmost leaf descendant $i$ of $v$ and the rightmost leaf descendant $j$ of $v$.
Figure~\ref{fig:bnestfor_with_spread} depicts a bicolored nested forest in which each internal node is labeled by its spread.
Note that the spread of an internal node $v$ is not necessarily the transposition $\tau_v$ assigned to each $v \in \internal{\wh{F}}$ earlier in Section~\ref{sec:torusfixedpoints}.

\begin{figure}[!h]
    \centering
    \includegraphics[scale=0.8]{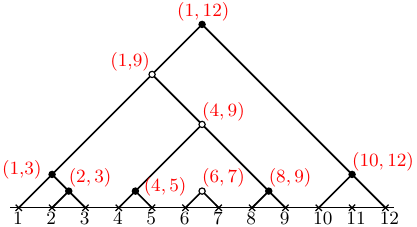}
    \caption{A bicolored nested forest with spreads recorded for each internal node}
    \label{fig:bnestfor_with_spread}
\end{figure}

\begin{prop}
\label{prop:spreadset}
If $\wh{F} \in \bnfornf_{n}$ and $w = \ForToNC(\wh{F})$, then 
\[
\{(i, j) \in \invnc{w} \suchthat w(i, j) \in I_{\wh{F}} \}
=
\{\text{spreads $(i, j)$ of black nodes $v \in \internal{\wh{F}}$}\}.
\]
In particular, if ${F} \in \indexedforests_{n}$ then $\invnc{w}$ is the set of all spreads for $v \in \internal{{F}}$.
\end{prop}
\begin{proof}
By~\Cref{prop:ForToNCBruhatMax}, it is sufficient to show that the set of elements of $I_{\wh{F}}$ that are adjacent to $w$ give the spreads of all black nodes in $\wh{F}$.  
We do so using induction on $n$.  
As $n = 1$ is vacuous, assume that $n > 1$ and write $\wh{F} = \wh{G} \cdot \xletter{k}$ and $u = \ForToNC(\wh{G})$ with $\wh{G}\in \bnfornf_{n-1}$. We split into cases according to \Cref{lem:ForToNcPsi}. First, suppose that either $\xletter{k}=\rletter{k}^-$ or that leaf $k$ is a right child of $\wh{G}$ (so in particular $\xletter{k}\ne \eletter{k}$). Then for any $v \in \internal{\wh{G}}$ with spread $(i, j)$ we have $\Psi_{k}^{-}\big(u(i\, j)\big) = w \Psi_{k}^{-}\big((i\, j)\big)=w(i + \delta_{i \ge k},  j + \delta_{j \ge k})$, and the node in $\internal{\wh{F}}$ corresponding to $v$ has spread $(i + \delta_{i \ge k},  j + \delta_{j \ge k})$ so the claim holds.

Now suppose $\xletter{k}\ne \rletter{k}^-$ and $k$ is not a right child in $\wh{G}$. If $\xletter{k}=\rletter{k}^+$, $\Psi_k^+(u(ij))=ws_k(i+\delta_{i\ge k},j+\delta_{j\ge k})s_k=w(i+\delta_{i\ge k+1},j+\delta_{j\ge k+1})$, and we conclude as before. If $\xletter{k} = \eletter{k}$ 
then the same reasoning applies except that $\wh{F}$ has one more black node than $\wh{G}$, with spread $(k, k+1)$, and $w$ has the additional adjacent element $w s_{k}$.
\end{proof}

We next give a recursive characterization of the charts in $\fl{n}$ determined by the $M_{\NC}(F)$ using the building operations from Section~\ref{sec:BuildingOperators}.

\begin{prop}
\label{prop:M_chart_recurse}
The representatives of the noncrossing Bruhat cells in $\fl{n}$ are characterized recursively by $M_{\NC}(\rletter{1}^{-})B = \fl{1}$ and for $F = G \cdot \eletter{i} \in \indexedforests_{n}$,
\[
M_{\NC}(F)B = 
\begin{cases}
\mathbb{G}_{i}M_{\NC}(G)B \sqcup \Psi_i^{-}M_{\NC}(G)B & \text{if leaf $i$ is a right child in $G$} \\
\mathbb{G}_{i}M_{\NC}(G)B \sqcup \Psi_i^{+}M_{\NC}(G)B & \text{otherwise.}
\end{cases}
\]
\end{prop}

We prove the proposition using a technical argument involving a modified version of the $\mathbb{G}_{i}$ operation.  Given a matrix $M$ with column $i$ equal to the $j$th basis vector for $j < i$, let $\mathbb{G}_i'M$ be obtained from $\Psi_i^-M$ by setting the $j,i$-entry to $+$.

\begin{eg}
We have
\[
\mathbb{G}_2\begin{bmatrix}a&1&b\\c&0&d\\e&0&f\end{bmatrix}=\begin{bmatrix}a&1&0&b\\0&+&1&0\\c&0&0&d\\e &0&0&f\end{bmatrix}
\qquad\text{and}\qquad
\mathbb{G}'_2\begin{bmatrix}a&1&b\\c&0&d\\ e &0&f\end{bmatrix}=\begin{bmatrix}a&+&1&b\\0&1&0&0\\c&0&0&d\\e &0&0&f\end{bmatrix}
\]
\end{eg}

\begin{lem}
\label{lem:Gshift}
If $M$ is an $(m-1)\times (m-1)$ matrix whose $i$th column is $e_j$ for some $j<i$, then we have an equality of sets $(\mathbb{G}_iM)B_m=(\mathbb{G}_i'M)B_m$.
\end{lem}
\begin{proof}    
For $x \in \CC^{\times}$, let $(\mathbb{G}_{i}M)(x)$ and $(\mathbb{G}_{i}'M)(x^{-1})$ denote the matrices obtained by setting the newly introduced $+$'s to $x$ and $x^{-1}$, respectively.  
Let $c_{i}$ and $c_{i+1}$ be the $i$th and $i+1$st columns of $(\mathbb{G}_{i}M)(x)$.
Then  $(\mathbb{G}_{i}'M)(x^{-1})$ is obtained from $(\mathbb{G}_{i}M)(x)$ by performing the column operations $c_i \mapsto x^{-1}c_i$ followed by $c_{i+1} \mapsto x c_i -x c_{i+1}$.  
Both forward column operations correspond to right multiplication by elements of $B_m$, so we have $(\mathbb{G}_{i}M)(x)B_m = (\mathbb{G}_{i}'M)(x^{-1})B_m$.  
We conclude as $x\mapsto x^{-1}$ is a bijection on $\mathbb{C}^*$.
\end{proof}

\begin{proof}[Proof of Proposition~\ref{prop:M_chart_recurse}]
Let $w = \ForToNC(F)$ and $u = \ForToNC(G)$. 
By Lemma~\ref{lem:ForToNcPsi}, we have $w = \Psi_{i}^{\pm}(u)$.  Further, by Proposition~\ref{prop:spreadset}, we have
\[
\invnc{w} = \{ (j + \delta_{j \ge i}, k + \delta_{k \ge i}) \suchthat (j, k) \in \invnc{u}\} \cup \{(i, i+1)\}.
\]
Thus, if we write $M_{\NC}(w;  x)$ for the matrix obtained from $M_{\NC}(w)$ by setting the $(w(i+1), i)$ entry equal to $x$, then we have
\[
M_{\NC}(w; 0) = \begin{cases}
\Psi^{-}_{i}M_{\NC}(u) & \text{if leaf $i$ is a right child in $G$} \\ 
\Psi^{+}_{i}M_{\NC}(u) & \text{otherwise.}
\end{cases}
\]
Moreover, $M_{\NC}(w) = M_{\NC}(w; 0) \sqcup M_{\NC}(w; +)$, so we complete the proof by showing that 
\[
    \mathbb{G}_{i}M_{\NC}(u)B = M_{\NC}(w; +)B.
\]

If leaf $i$ is not a right child in $G$, then $w(i) > w(i+1)$ and we have a direct equality $M_{\NC}(w; +) = \mathbb{G}_{i}M_{\NC}(u)$.  
If leaf $i$ is a right child in $G$, then $w(i) < w(i+1)$ and $M_{\NC}(w;  +) = \mathbb{G}_{i}'M_{\NC}(u)$, so we must use Lemma~\ref{lem:Gshift} to relate $\mathbb{G}_{i}'$ and $\mathbb{G}_{i}$.  
Indeed, as a right child $i$ can never be the left endpoint of an internal node, Proposition~\ref{prop:spreadset} implies that  $u$ has no noncrossing inversion ending in $i$.  
Thus $M_{\NC}(G)$ has no $\ast$'s in column $i$ and the hypotheses of Lemma~\ref{lem:Gshift} are satisfied, giving $M_{\NC}(w; +)B = \mathbb{G}_{i}M_{\NC}(u)B$.
\end{proof}

\begin{prop}
\label{prop:charts_are_charts}
Let $w \in \NC_{n}$, and let $F$ be the unique forest in $\indexedforests_n$ such that  $w=\ForToNC(F)$.  
We have $M_{\NC}(F)B = C(F)$, and as a consequence there is a  $T$-equivariant isomorphism
\[
C(F) \cong  M_{\NC}(w) \cong  \bigoplus_{(i,j)\in \invnc{w}}\mathbb{C}_{\chi_{w(j)}\chi_{w(i)}^{-1}}.
\]
\end{prop}
\begin{proof}
We show that $C(F)$ agrees with the recursive characterization of $M_{\NC}(F)B$ given in Proposition~\ref{prop:M_chart_recurse}.  
As $C((\rletter{1}^{-})^{n-k}) = \idem_{S_{n-k}}\in \fl{n-k}$, this amounts to showing that for $F = G \cdot \eletter{i} \in \indexedforests_{n}$,
\[
C(F)=\begin{cases}
\mathbb{G}_iC(G) \sqcup \Psi_i^-C(G)&\text{if leaf $i$ is a right child in $G$}\\
\mathbb{G}_iC(G) \sqcup \Psi_i^+C(G)&\text{otherwise.}\end{cases}
\]

By the definition of $\Face(\wh{F})$ and the $\le_{re}$ order, we have 
\[
\Face(\wh{F})
=
\{\wh{H}'\cdot \xletter{i} 
\suchthat 
\text{$\xletter{i} \in \{\rletter{i}^{-}, \rletter{i}^{+}, \eletter{i}\}$ and $\wh{H}' \in \Face(G)$}\}.
\]
If we furthermore write $w = \ForToNC(F)$ and $u = \ForToNC(G)$,
then by Lemma~\ref{lem:ForToNcPsi} we have $w = \Psi^{\epsilon}_{i}(u)$ for $\epsilon \in \{+, -\}$.  
Thus
\[
\{\wh{H} \in \bnfor_{n} \suchthat w \in I_{\wh{H}} \}
=
\{\wh{H}'\cdot \xletter{i} 
\suchthat 
\text{$\xletter{i} \in \{\rletter{i}^{\epsilon}, \eletter{i}\}$, $\wh{H}' \in \bnfor_{n-1}$ and $u \in I_{\wh{H}'}$}\}.
\]
By intersecting the sets above, we arrive at the following description of $C(F)$:
\begin{multline*}
C(F) = \bigsqcup_{\substack{ \wh{H}' \le_{re} G \\ u \in I_{\wh{H}}}} \Big( X^{\circ}(\wh{H} \cdot \eletter{i}) \sqcup X^{\circ}(\wh{H} \cdot \rletter{i}^{\epsilon})\Big)\\[-4ex]
= \bigsqcup_{\substack{ \wh{H}' \le_{re} G \\ u \in I_{\wh{H}}}} \Big( \mathbb{G}_{i} X^{\circ}(\wh{H}) \sqcup \Psi_{i}^{\epsilon} X^{\circ}(\wh{H})\Big)
= \mathbb{G}_{i}C(G) \sqcup \Psi_{i}^{\epsilon}C(G)
\end{multline*}
This completes the proof as by Lemma~\ref{lem:ForToNcPsi} $\epsilon = -$ if leaf $i$ is a left child and $+$ otherwise.
\end{proof}

We can now complete the proof of Theorem~\ref{thm:paving}.

\begin{proof}[Proof of \Cref{thm:paving}]  
By Theorem~\ref{thm:paving_Bott}, $\qfl_{n}$ is the disjoint union of the charts $C(F)$ for $F \in \indexedforests_{n}$.  
By Proposition~\ref{prop:charts_are_charts}, these charts are exactly the $M_{\NC}(w)B$ for $w\in\NC_n$, giving Equation~\eqref{eq:hhmpequalsunion}. 
As $M_{\NC}(w)B \subseteq M(w)B=BwB$, and the $B w B$ form a partition of $\fl{n}$ by \eqref{eq:bruhat_decomposition}, it follows that $M_{\NC}(w)B = \qfl_n\cap M(w)B$.

Now for any $w\in S_n$, the Schubert variety $X^{w} = \overline{BwB}$ consists of all $BuB$ with $u \le w$ in Bruhat order. It follows that $X_{k} = \qfl_n\cap \big( \bigcup_{i=1}^{k} X^{w_{i}} \big)$ is closed, which concludes the proof.
\end{proof}

\section{Intrinsic characterizations of $\qfl_n$}
\label{sec:Intrinsic}
We now give two intrinsic characterizations of $\qfl_n$.  
These characterizations are independent of one another and are presented in Sections~\ref{sec:PluckerCharacterization} and~\ref{sec:EquivCharacterization}.

\subsection{Characterization with Pl\"{u}cker functions}
\label{sec:PluckerCharacterization}

This section proves Theorem~\ref{mainthm:plucker}, which states that
\[
\qfl_n=\hspace{-0.75em}\bigcap_{w\in S_n\setminus \NC_n}\hspace{-0.75em}\{\Pl_{w} =0\}.
\]
The proof is given at the end of the section.  We begin with a classical observation, which can be found for instance in \cite[Proposition 2.6]{leemasudapark2024torusorbitclosuresflag}.

\begin{prop}
\label{prop:plucker_points}
For $\mc{F} \in \fl{n}$, the set of torus fixed points in $\overline{T \cdot \mc{F}}$ is $\{wB \suchthat \Pl_{w}(\mc{F}) \neq 0\}$.
\end{prop}

Every element of $BwB$ has a unique representative $h$ in the set $M(w)$ defined in Section~\ref{sec:affinepaving}.  Thus when restricted to $BwB$, we can view the Pl\"{u}cker functions as polynomials in the matrix entries which are not uniformly $0$ or $1$ across all of $M(w)$, namely the $h_{w(j), i}$ for $(i, j) \in \inv{w}$.  
Going forward, we define a $\ZZ^{n}$-grading on such polynomials by setting
\[
\operatorname{degree}(h_{w(j), i}) = e_{w(j)} - e_{w(i)},
\]
where $e_{k}$ denotes the $k$th standard basis vector.  
This grading is the weight of the character of the $T$-action on each entry of $M(w)$ under the action of $T$ as described in Section~\ref{sec:affinepaving}. 
\begin{obs}
\label{obs:annoying}
Let $w \in S_{n}$.  
Expressed in the entries of $h \in M(w)$, $\Pl_{u}(h)/\Pl_{w}(h)$ is a homogeneous polynomial of degree $\sum (u^{-1}(i)-   w^{-1}(i))e_{n+1-i}$ for each $u \in S_{n}$.
\end{obs}
\begin{proof}
We compute directly that $\Pl_{w}(h) = 1$, so $\Pl_{u}(h)/\Pl_{w}(h) = \Pl_{u}(h)$ has denominator $1$. 
The claim now follows from the fact that the Pl\"{u}cker functions are $T$-equivariant and the weight of the character by which $T$ scales $\Pl_{u}(h)/\Pl_{w}(h)$ is $\sum u^{-1}(i)e_{n+1-i} - \sum  w^{-1}(i)e_{n+1-i}$.
\end{proof}

In order to perform a more granular analysis on the degree of each Pl\"{u}cker function, we establish some notation using the root system of type $A_{n-1}$.  
The positive roots in this system are the vectors $e_{i} - e_{j}$ for $1 \le i < j \le n$ and the negative roots are $e_{j} - e_{i}$ for $1 \le i < j \le n$.  
Given $F\in \indexedforests_{n}$, we define the polyhedral cone $\cone{F}$ by 
\[
    \cone{F}=\mathbb{R}_{\geq 0}\{e_{j} - e_{i}\suchthat (i,j) \text{ a spread in } F\},
\]
and for $w\in \NC_n$ we define $\cone{w}= w \cdot \cone{F}$ where $F\in\indexedforests_n$ is the unique forest such that $w=\ForToNC(F)$.
In view of Proposition~\ref{prop:spreadset}
\[
\cone{w} = \mathbb{R}_{\geq 0}\{e_{w(j)} - e_{w(i)}\suchthat (i,j)\in \invnc{w}\}.
\]

Spreads are characterized by that fact that if $(i, j)$ is a spread in $F \in \indexedforests_{n}$, then no spread has the form $(j, k)$.  
Such sets (and their associated cones) were first studied in~\cite{GGP97} and are commonly known as  \emph{noncrossing alternating forests}; see for example~\cite{AlNa09}. 
The following result can be found  in~\cite[\S 6]{GGP97}.

\begin{prop}
\label{prop:cones_are_all_good}
For each $w \in \NC_{n}$, $\cone{w}$ is simplicial and the only roots it contains are the generators $e_{w(j)} - e_{w(i)}$ for $(i,j)\in \invnc{w}$.
\end{prop}

The sets $(w(b), w(a))$ for $(a, b)$ a spread in $F$ also appear in the literature as a generalization of noncrossing alternating forests.  
Specifically,~\cite[\S 6]{JosuatNadeau2023Koszulity} shows that these are canonically in bijection with the set of $\cox$-clusters.
\begin{proof}[Proof of Theorem~\ref{mainthm:plucker}]
First, we take $\mc{F}\in \qfl_n$. 
By~\Cref{thm:noncrossing_fixed_points}, the torus fixed points in $\overline{T \cdot \mc{F}}$ are all noncrossing partitions.  
By~\Cref{prop:plucker_points}, this means that the nonvanishing Pl\"{u}cker functions of $\mathcal{F}$ are also indexed by elements of $\NC_{n}$.  

Conversely, suppose that $\Pl_{u}\mc{F} = 0$ for all $u\in S_n\setminus \NC_n$.  
Let $w\in S_n$ be such that $\mc{F}\in BwB$ and let $h \in \GL_n$ be the representative of $\mc{F}$ in $M(w)$. 
As $\Pl_w\mc{F}\ne 0$ on $BwB$ we conclude that $w\in \NC_n$. We now claim that for each $(a, b)\in \inv{w}\setminus \invnc{w}$ we have $h_{w(b),a}=0$.  

We proceed by induction on $w(a) - w(b)$.  Let $u = w(a\,b)$ and let $\alpha = (b-a)(e_{w(b)} - e_{w(a)})$.
We consider the set $S$ consisting of all multisubsets $M$ of $\inv{w}$ with the property that $\sum_{(i, j) \in M} e_{w(j)} - e_{w(i)} = \alpha$ so that by \Cref{obs:annoying} we have
\[
\Pl_{u}(h)/\Pl_{w}(h) = \sum_{M \in S} c_{M} \prod_{(i, j) \in M} h_{w(j), i}
\qquad\text{for some $c_{M} \in \ZZ$}.
\]
We now describe some properties of the elements  $M \in S$.  
First, by considering the first and last nonzero coordinate in the sum $\sum_{(i, j) \in M} e_{w(j)} - e_{w(i)} = \alpha$, every $e_{w(j)} - e_{w(i)} \in M$ has either $w(j) - w(i) < w(b) - w(a)$ or $(i, j) = (a, b)$.  
Second, by Proposition~\ref{prop:cones_are_all_good}, $M$ must contain at least one element of the form $e_{w(j)} - e_{w(i)}$ for $(i, j) \in \inv{w}\setminus \invnc{w}$.
Finally, a direct computation of $\Pl_{u}(h)/\Pl_{w}(h)$ shows that the coefficient of $(h_{w(b), a})^{b-a}$ is nonzero: for $i < a$ or $i \ge b$, $\det(h_{u(1), \ldots, u(i)}) = \det(h_{w(1), \ldots, w(i)}) = \pm 1$, while for $a \le i < b$, $\det(h_{u(1), \ldots, u(i)})$ contains $h_{w(b), a}$ with a coefficient of $\pm 1$.  
Thus by our assumption on $\Pl_{u}$ and our inductive hypothesis, we have $0 = (h_{w(b), a})^{b-a}$. This proves the claim. 
This shows $h\in M_{\NC}(w)$ and finally $\mc{F}\in\qfl_n$ by~\Cref{prop:charts_are_charts}.
\end{proof}

\subsection{Characterization via equivalence of flags}
\label{sec:EquivCharacterization}
We now give our second characterization of $\qfl_{n}$ as flags that can be obtained from $\id_{\fl{n}}$, the standard coordinate flag $\{0\}\subset \{e_1\}\subset \{e_1,e_2\}\subset \cdots \subset \{e_1,\dots,e_n\}$,  via certain elementary operations.

\begin{defn}\label{def:relation_for_characterization}
    Define $\sim$ to be the equivalence relation on complete flags generated by the relations $\sim_i$ for $1\le i \le n-1$ given by $\mc{F}\sim_i \mc{G}$ if
    \begin{enumerate}[label=(\arabic*)]
     \item \label{it1:intrinsic} 
     $\mc{F}_j=\mc{G}_{j}$ for all $j\ne i$,  and
        \item \label{it2:intrinsic} 
    $e_i\in \mc{F}_{i+1}$ and $\mc{F}_{i-1}\subset \{x_i=0\}$.
    \end{enumerate}
\end{defn}

We have that $\mc{F}\sim_i \mc{G}$ for $1\leq i\leq n-1$ if and only if  there exists $\mc{H}\in\fl{n-1}$ such that $\mc{F},\mc{G}\in \mathbb{P}_{i}\mc{H}$, where $\mathbb{P}_{i}$ is defined in Section~\ref{sec:pushpull}. In particular note that $\Psi_i^{-}\mc{H}\sim_i \Psi_i^{+}\mc{H}$.

\begin{thm}
\label{cor:IntrinsicQFL}
   The quasisymmetric flag variety $\qfl_n\subset \fl{n}$ is the equivalence class of $\sim$ containing the standard coordinate flag.
\end{thm}

We prove this at the end of the subsection after a preparatory lemma.

\begin{lem}
\label{lem:annoying}
Let $\wh{F}\in \bnfor_n$. Suppose there exists $i$ such that every element  $w\in I_{\wh{F}}$ satisfies $i\in \{w(i),w(i+1)\}$. Then we have $X(\wh{F})\subset\mathbb{P}_{i}X(\wh{G})$ for some $\wh{G}\in \bnfor_{n-1}$. 
\end{lem} 
\begin{proof}
We will prove the following statements, from which the conclusion follows immediately.
 \begin{enumerate}[label=(\roman*)]
        \item \label{it:case1} If $w(i)=i$ for all $w\in I_{\wh{F}}$, then $\wh{F}=\wh{G} \cdot \rletter{i}^-$ for some $\wh{G} \in \bnfor_{n-1}$.
        \item \label{it:case2} If $w(i+1)=i$ for all  $w\in I_{\wh{F}}$, then $\wh{F}\ctam \wh{G}\cdot \rletter{i}^+$ for some $\wh{G} \in \bnfor_{n-1}$.
        \item \label{it:case3} If $i\in \{w(i),w(i+1)\}$ for all $w\in I_{\wh{F}}$, but we are not in a scenario covered by Cases~\ref{it:case1} and~\ref{it:case2}, then $\wh{F}\ctam \wh{G}\cdot \eletter{i}$ for some $\wh{G} \in \bnfor_{n-1}$.
\end{enumerate}

    If we are in case~\ref{it:case1},  then $\ncperm(\wh{F})\in I_{\wh{F}}$ has $i$ as a fixed point, implying that the leaf labeled $i$ is a singleton tree in $\wh{F}$. 
    This immediately yields $\wh{F}=\wh{G}\cdot \rletter{i}^-$. 

    Now suppose we are in a situation described in~\ref{it:case2}.
    Let $v\in \internal{\wh{F}}$ have canonical label $i$.
    Let $P_1$ (respectively $P_2$) be the path beginning from the leaf labeled $i$ (respectively $i+1$) and terminating in $v$.
    Observe that all but the final edge in $P_1$ connect a right child to its parent node and similarly all edges but the final edge in $P_{2}$ connect a left child to its parent. 
    We claim that $v$ must be white. 
    Indeed, if $v$ were black, then left edge deletion at $v$ would result in an element in $I_{\wh{F}}$ that has $i$ and $i+1$ in different cycles.
    For  similar reasons we infer that all nodes in $P_2$ are necessarily white. 
    Thus we are in a situation depicted on the left in Figure~\ref{fig:tamari_annoying_1} where the ``half-filled'' nodes could be black or white.  
    By performing colored Tamari rotation as in Definition~\ref{fig:right_rotation}, first along $P_1$ and then along $P_2$ as in Figure~\ref{fig:tamari_annoying_1}, one can obtain a bicolored nested forest wherein $v$ has left and right children being leaves with labels $i$ and $i+1$.  
    Then as described in Definition~\ref{def:forest_insertion}, there exists a $\wh{G} \in \bnfor_{n-1}$ satisfying $\wh{F}=\wh{G}\cdot\rletter{i}^+$.

    \begin{figure}[!h]
        \centering
        \includegraphics[scale=1.2]{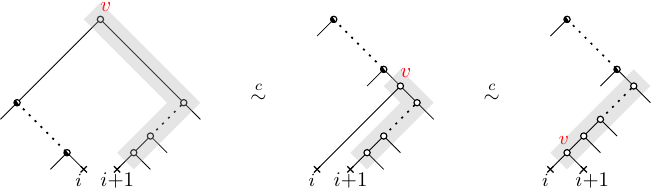}
        \caption{Case~\ref{it:case2} (left) and a bicolored nested forest that is colored Tamari equivalent. The half-filled nodes could be either black or white.}
        \label{fig:tamari_annoying_1}
    \end{figure}

Finally we consider~\ref{it:case3}.
Let $v\in \internal{\wh{F}}$ have canonical label $i$, like before.
For the condition in~\ref{it:case3} to hold, $v$ must be black. Indeed if $v$ were white then no element of $I_{\wh{F}}$ has $i$ as a fixed point.
For this same reason the left child of $v$ is necessarily the leaf labeled $i$.
As above, the path from the leaf labeled $i+1$ to $v$ can only contain white nodes, as this ensures that $i$ and $i+1$ are in the same cycle if $i$ is not a fixed point.
We use colored Tamari rotations exactly as in Case~\ref{it:case2} to obtain a bicolored nested forest where $v$ has left and right children given by $i$ and $i+1$.
 Left edge deletion at $v$ now gives $\wh{G}\in \bnfor_{n-1}$ satisfying $\wh{F}=\wh{G}\cdot\eletter{i}$.
Figure~\ref{fig:tamari_annoying_2} outlines this case.
\begin{figure}[!h]
    \centering
    \includegraphics[scale=1.2]{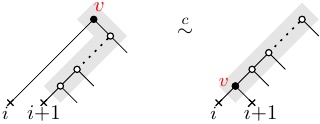}
    \caption{Case~\ref{it:case3} (left) and a bicolored nested forest that is colored Tamari equivalent}
    \label{fig:tamari_annoying_2}
\end{figure}
\end{proof}

\begin{proof}[Proof of Theorem~\ref{cor:IntrinsicQFL}]
    First we show that $\qfl_n$ is closed under these relations.  Suppose first that $\mc{F}\in \qfl_{n}$ satisfies condition \ref{it2:intrinsic} in Definition~\ref{def:relation_for_characterization} for a fixed $i$, which means that $\mc{F}\in \mathbb{P}_i\fl{n-1}$. 
    
    We claim that in that case $\mc{F}\in \mathbb{P}_i\qfl_{n-1}$. Condition \ref{it2:intrinsic} is closed and invariant under the action of $T$, so it is satisfied by all elements of the torus-orbit closure $\overline{T \cdot \mc{F}}$.  By Theorem~\ref{thm:XF_torus_orbit}, $\overline{T \cdot \mc{F}}=X(\wh{F})$ for some $\wh{F} \in \bnfor_{n}$, so the conditions \ref{it2:intrinsic} also apply to the set of torus fixed points $I_{\wh{F}}\subset \NC_n$. 
    This means that  $i$ is a fixed point or $i,i+1$ are in the same cycle for each element of $I_{\wh{F}}$, so Lemma~\ref{lem:annoying} guarantees the existence of $\wh{G} \in \bnfor_{n-1}$ such that  $X(\wh{F})\subset\mathbb{P}_{i}X(\wh{G})$. This proves the claim since $\mc{F} \in X(\wh{F})$ and $X(\wh{G})\subset\qfl_{n-1}$. 

    Now if $\mc{F}\sim_i\mc{H}$, then by condition \ref{it1:intrinsic} we get $\mc{H}\in \mathbb{P}_iX(\wh{G})\subset  \mathbb{P}_i\qfl_{n-1}$ as well, and thus $\mc{H}\in \qfl_{n}$.

    Conversely, we now show that we can reduce every element of $\qfl_n$ to $\idem_{\fl{n}}$ using these relations. We do this by induction on $n$, the case $n=2$ being trivial.
    
    Let $\mc{F}\in \qfl_n$. Since $\qfl_n=\mathbb{P}_1\qfl_{n-1}\cup \cdots \cup \mathbb{P}_{n-1}\qfl_{n-1}$  and every element of $\mathbb{P}_i\mc{H}$ is $\sim_i$-equivalent to $\Psi_i^-\mc{H}$, there exist $j\in\{1,\ldots,n-1\}$ and $\mc{G}\in \qfl_{n-1}$ such that $\mathcal{F}\sim_j\Psi_j^-\mc{G}$. By the inductive hypothesis we know that $\mc{G}\sim \idem_{\fl{n-1}}$.
    
       We claim that for any $i$, $\mc{H}\sim_i \mc{H}'$ in $\qfl_{n-1}$ implies  $\Psi_j^-\mc{H}\sim\Psi_j^-\mc{H}'$ in $\qfl_n$. To prove this, we let $\mc{K}\in \qfl_{n-2}$ be such that $\mc{H},\mc{H}'\in \mathbb{P}_i\mc{K}$ and consider different cases.
    \begin{enumerate}[label=(\roman*)]
        \item If $i\ge j$ then by Lemma~\ref{lem:babyrtsame} we have $\Psi_j^-\mathbb{P}_i\mc{K}=\mathbb{P}_{i+1}\Psi_{j}^-\mc{K}$ and so $\Psi_j^-\mc{H}\sim_{i+1}\Psi_j^-\mc{H}'$.
        \item If $j\ge i+2$ then by Lemma~\ref{lem:babyrtsame} we have $\Psi_j^-\mathbb{P}_i\mc{K}=\mathbb{P}_i\Psi_{j-1}^-\mc{K}$ and so $\Psi_j^-\mc{H}\sim_i\Psi_j^-\mc{H}'$.
        \item Finally suppose $j=i+1$. 
        By Lemma~\ref{lem:2additionalbuilding} we have
        $  \Psi_{i+1}^+\mathbb{P}_i\mc{K}=\mathbb{P}_i\Psi_{i}^+\mc{K}$. This in turn implies 
        \[
            \Psi_{i+1}^-\mc{H}\sim_{i+1}\Psi_{i+1}^+\mc{H}\sim_i\Psi_{i+1}^+\mc{H}'\sim_{i+1} \Psi_{i+1}^-\mc{H}'.
        \]
    \end{enumerate}
Thus $\Psi_j^{-}\mc{H}\sim \Psi_{j}^{-}\mc{H}'$ in all cases and the claim is proved. By induction we get that if $\mc{H}\sim \mc{H}'$ in $\qfl_{n-1}$ then $\Psi_j^-\mc{H}\sim\Psi_j^-\mc{H}'$. We apply this to $\mc{G}$ and $\idem_{\fl{n-1}}$ and get $\Psi_j^-\mc{G}\sim \Psi_j^-\idem_{\fl{n-1}}=\idem_{\fl{n}}$, which concludes the proof since $\mathcal{F}\sim_j\Psi_j^-\mc{G}$.
\end{proof}

\section{The GKM presentation of $H^{\bullet}_{T_{n}}(\qfl_{n})$}
\label{sec:CoinvariantsGKM}

In this section we give a presentation of $H^\bullet_{T_n}(\qfl_n)$ and $H^\bullet(\qfl_n)$ in terms of a certain combinatorially defined ``graph cohomology ring,'' and describe a free $\ZZ[\tl_n]$-basis for this ring. 
These results will be used in the next section to give a Borel-type presentation.

We appeal to GKM theory, which is a technique for computing the equivariant cohomology ring of a variety $X$ under the action of an algebraic torus  under suitable hypothesis. 
While originally developed for rational cohomology by Goresky, Kottwitz, and MacPherson~\cite{GKM98} with inspiration from Chang and Skjelbred \cite{ChSk74}, we present a variant with stricter hypotheses that computes integral cohomology.   

Throughout we use the case $X = \fl{n}$ as a motivating example.

\subsection{The GKM ring}
\label{sec:GKM_recall}

Without loss of generality we take $T = T_{n}$. 
As is standard in algebraic combinatorics, we denote by $t_{i}$ the \textit{negative} first Chern class $-c_{1}^{T}(\CC_{\chi_{i}}) \in H^{2}_{T_{n}}(\point)$.  
We then have a homomorphism of abelian groups
\[
\begin{array}{rcl}
-c_{1}^{T}(\CC_{(-)}) : \{\text{Characters of $T$}\} &\to& H^{2}_{T_{n}}(\point) \\
\chi_{1}^{a_{1}} \chi_{2}^{a_{2}} \cdots \chi_{n}^{a_{n}} & \mapsto & -(a_{1}t_{1} + a_{2}t_{2} + \cdots + a_{n}t_{n}).
\end{array}
\]
The equivariant cohomology ring of a point is freely generated by the $t_{i}$ and we identify $H^{\bullet}_{T_{n}}(\point) = \ZZ[\tl_{n}]$, so that all $T_n$-equivariant cohomology rings are $\ZZ[\tl_{n}]$-algebras.

For an edge labeled graph $G$ with vertices $V$, edges $E$, and edge labels given by a function
\[
\chi\colon  E\to \text{linear nonzero polynomials in $\ZZ[\tl_n]$}/\pm,
\]
we define the \emph{graph cohomology ring} for $G$ to be the $H^\bullet_{T_n}(\point)$-algebra
\[
H_{T_n}^\bullet(G)\coloneqq \{(f_v)_{v\in V}\suchthat \chi(uv)\text{ divides }f_v-f_u\text{ for all }uv\in E\}\subset \ZZ[\tl_n]^{\oplus V}
\]
with multiplication defined pointwise.

We now describe sufficient conditions for $H_{T_n}^\bullet(G)$ to be the cohomology ring for a variety $X$.  
Say that $X$ has a \emph{good affine paving} if there is a filtration $\emptyset=X_0\subset X_1\subset X_2\subset \cdots \subset X_{\ell}=X$ by closed subvarieties $X_{i}$ such that for each $i\ge 1$ the following hold.
\begin{enumerate}
\item The set $X_i\setminus X_{i-1}$ contains a unique $T$-fixed point $p_i$, and there is a $T$-equivariant isomorphism of algebraic varieties $X_i\setminus X_{i-1}\cong V_i$ for some linear $T$-representation $V_i$.

\item The representation $V_i$ decomposes into a direct sum of one-dimensional $T$-representations
\[
V_i=\bigoplus_{j\in A_i}V_{i,j}
\qquad\text{where}\qquad A_i\subset \{1,\ldots,i-1\}
\]
such that $\overline{V_{i,j}}=V_{i,j}\cup \{p_j\}$ and topologically $\overline{V_{i,j}}\cong \mathbb{P}^1$.

\item For each $j \in A_{i}$, $f_{i,j}=-c_1^T(V_{i, j}) \in H^\bullet_{T_n}(G)$ satisfies:
\begin{enumerate}
    \item $f_{i, j} \neq \pm f_{i, k}$ for $j \neq k$, and 
    \item $f_{i, j}$ is reduced, meaning that if $f_{i, j} = a_{1} t_{1} + a_{2} t_{2} + \cdots + a_{n} t_{n}$, then $\gcd(a_{1}, a_{2}, \ldots, a_{n}) = 1$.  
\end{enumerate}
\end{enumerate}

A good affine paving on $X$ defines a \emph{GKM graph}, which is an undirected, edge-labeled graph $G_{X}$ with vertex set given by the fixed points $X^{T} = \{p_{1}, \ldots, p_{\ell}\}$.  
For each one-dimensional summand $V_{i, j}$ in (2), $G_{X}$ has an edge $p_{i}p_{j}$, and this edge is labeled by $-c_{1}^T(V_{i, j})$.

\begin{eg}
\label{eg:fl_GAP}
Let $X = \fl{n}$.  
For any total order $w_{1}, \ldots, w_{n!}$ of $S_{n}$ that extends the Bruhat order, $X_{k} = \bigcup_{i = 1}^{k} Bw_{i}B$ defines a good affine paving.  
Then $V_{i} \cong M(w)$ with $p_{k} = w_{k}B$, and following Equation~\eqref{eq:bruhat_chars} we have $j \in A_{i}$ if and only if $w_{j} = w_{i}(a\,b)$ for $(a, b) \in \inv{w}$ and $f_{i, j} = t_{w(b)} - t_{w(a)}$. It follows that the GKM graph is obtained from $\cay{S_{n}}$ by labeling edges of the form $w$ to $(i, j)w$ by $t_{j} - t_{i}$.
\end{eg}

We now consider $\qfl_{n}$.  
Let $\cay{\NC_{n}}$ denote the Hasse diagram of the Kreweras order on $\NC_{n}$ as defined in Section~\ref{sec:ncp}, which is an induced subgraph of $\cay{S_{n}}$. 

\begin{thm}
\label{thm:good_affine_paving}
Theorem~\ref{thm:paving} gives a good affine paving of $\qfl_{n}$, and its GKM graph is obtained from $\cay{\NC_{n}}$ by labeling edges of the form $w$ to $(i, j)w$ by $t_{j} - t_{i}$.  
\end{thm}
\begin{proof}
Theorem~\ref{thm:paving} verifies condition (1) directly and shows that $\qfl_{n}^{T_{n}} = \{wB \suchthat w \in \NC_{n}\}$.  
Conditions (2) and (3) then follow from the fact that our filtration is obtained by intersecting $\qfl_{n}$ with a good affine paving for $\fl{n}$.  
The same reasoning computes the edges and edge labels for the GKM graph.
\end{proof}

\begin{thm}
\label{thm:GKM_1}
If $X$ has a good affine paving, then:
\begin{enumerate}
\item $X$ has a $T$-invariant homology basis given by the classes $[\overline{X_i\setminus X_{i-1}}]\in H_\bullet(X)$,

\item $H^\bullet_T(X)\cong H^\bullet_T(G_X)$, the graph cohomology ring, and

\item if $H^\bullet_T(X)$ is a free $\ZZ[\tl_n]$-module then  $H^\bullet_T(X)/\langle t_1,\ldots,t_n\rangle\cong H^\bullet(X)$.
\end{enumerate}
\end{thm}

While variants of Theorem~\ref{thm:GKM_1} appear as~\cite[Theorem 2.3]{MR2166181} and \cite[Theorem 1.2.2]{GKM98}, we did not find the exact statement required for $\qfl_{n}$.  Therefore we include a proof for completeness.

\begin{proof}
The first part follows from \cite[see Example 1.9.1 and 19.1.11]{Fulton_Intersection} (in fact the existence of the filtration where each $X_i\setminus X_{i-1}$ is isomorphic to an affine space suffices). 
The second part follows from \cite[Theorem 3.1]{MR2166181}. 
The last part follows from the implication $(iii)\implies (i)$ of \cite[Theorem 1.1]{MR2308029} after noting that $(S^1)^n$-equivariant cohomology is identical to $T_n$-equivariant cohomology because $\mathbb{C}^\ast \cong \mathbb{R}\times S^1$ and $\mathbb{R}$ is contractible.
\end{proof}

\begin{eg}
\label{eg:fl_GKM_ring}
By Theorem~\ref{thm:GKM_1}, $H^{\bullet}_{T_{n}}(\fl{n})$ is isomorphic to the graph cohomology ring 
\[
H^{\bullet}_{T_{n}}\big(\cay{S_{n}}\big) = \{(f_{w})_{w \in S_{n}} \;|\; \text{$t_{i} - t_{j}$ divides $f_{w} - f_{(i\,j)w}$ for all $(i \neq j)$}\} \subseteq \ZZ[\tl_{n}]^{\oplus S_{n}}.
\]
Moreover, the $[X^{w}] = [\overline{BwB}]$ give a homology basis for $H_{\bullet}(\fl{n})$.
\end{eg}

From~\Cref{thm:good_affine_paving}, we obtain the following corollary about $\qfl_{n}$.

\begin{cor}
\label{cor:qfl_GKM}
We have 
\[
H^{\bullet}_{T_{n}}(\qfl_{n}) \cong 
\{\big(f_{w}\big)_{w \in \NC_{n}} \suchthat \text{$t_{b} - t_{a}$ divides $f_{w} - f_{(a\,b)w}$ whenever $w, (a\,b)w \in \NC_{n}$}\} \subseteq \ZZ[\tl_{n}]^{\oplus \NC_{n}}.
\]
\end{cor}

Further recall that the affine charts from Proposition~\ref{prop:charts_are_charts} have closures $X(F)$ for $F \in \indexedforests_{n}$.

\begin{cor}
\label{cor:qfl_homology_basis}
The homology group $H_{\bullet}(\qfl_{n})$ has a homology basis $[X(F)]$ for $F \in \indexedforests_{n}$.
\end{cor}

\begin{rem}
Simple generalizations of Theorem~\ref{thm:GKM_1} exist to compute generalized cohomology theories such as equivariant $K$-theory.  However, determining a good basis for the resulting rings is a combinatorially specific task which does not transfer easily between theories.  
\end{rem}

\subsection{Flowup bases and double forest polynomials}
\label{sec:flowup}

The following definition characterizes a distinguished subset of $H_{T_n}^\bullet(G)$; the reader should compare this to the definition of \textit{generating family}
by Guillemin--Zara \cite[Definition 2.3]{GZ03} or that of \textit{canonical classes} by Tymoczko \cite[\S 2.2]{Tym08}.
\begin{defn}
\label{def:flowup}
Let $G=(V,E,\chi)$ be a GKM graph. Given a partial ordering $\le$ on $V$, a \emph{flowup basis} for  $H_{T_n}^\bullet(G)$ is a collection of elements $\{f_v\suchthat v\in V\}\subset H^\bullet_{T_n}(G)$ such that
\begin{enumerate}[label=(\arabic*)]
    \item \label{it1:flowup} $(f_v)_w=0$ if $v\not\le w$, and 
    \item \label{it2:flowup} $(f_v)_v=\pm \prod_{uv\in E\text{ and }u\le v} \chi(uv)$.
\end{enumerate}
\end{defn}

The following fact is classical and a key tool for producing $\ZZ[\tl_n]$-bases for $H^\bullet_{T_n}(G)$.

\begin{prop}
    Any flowup basis is a free $\ZZ[\tl_n]$-basis for $H^\bullet_{T_n}(G)$. 
\end{prop}
\begin{proof}
An outline of the classical proof can be found in \cite[\S 2.2]{Tym08}.
\end{proof}
 
In this section we describe a flowup basis for $H^{\bullet}_{T_{n}}(\qfl_{n})$ using the double forest polynomials defined in \cite{BGNST1}.
For $w \in S_{n}$, let
\begin{equation}
\label{eq:evaluation_map_definition}
\begin{array}{rcl}
\ev_{w}: \ZZ[\tl_{n}][\xl_{n}] & \to & \ZZ[\tl_{n}] \\
f(\xl_{n}; \tl_{n}) & \mapsto & f(t_{w(1)}, t_{w(2)}, \ldots, t_{w(n)}; \tl_{n}).
\end{array}
\end{equation}

\begin{eg}
\label{eg:classical_schubert_flowup}
Consider the graph cohomology ring for $\fl{n}$ from Example~\ref{eg:fl_GKM_ring}.  
Taking $\le$ to be the Bruhat order on $S_{n}$, we have a flowup basis with
\[
(f_{v})_{w} = \ev_{w}\Big(\schub{v}(\xl_{n}; \tl_{n})\Big),
\qquad\text{for $v, w \in S_{n}$}
\]
where $\schub{v}(\xl_{n}; \tl_{n})$ is the double Schubert polynomial; see Section~\ref{sec:flag_recall}.  The fact that conditions~\ref{it1:flowup} and~\ref{it2:flowup} in Definition~\ref{def:flowup} are met is nontrivial but follows from the AJS--Billey theorem \cite{AJS, Bil99}.
\end{eg}

The analogous statement for $\qfl_{n}$ makes use of the double forest polynomials defined in~\cite[\S 4]{BGNST1}, which we denote by $\forestpoly{F}(\xl_{n}, \tl_{n}) \in \ZZ[\xl_{n}][\tl_{n}]$ for each $F \in \indexedforests_{n}$.  
As with Schubert polynomials, we postpone a precise definition of double forest polynomials to Section~\ref{sec:qfl_brl}.

\begin{thm}
\label{thm:forest_flowup}
Taking $\le$ to be the Bruhat order restricted to $\NC_{n}$, double forest polynomials define a flowup basis for the graph cohomology ring $H^{\bullet}_{T_{n}}(\cay{\NC_{n}})$.
Specifically, for $v,w\in \NC_n$ we have
\[
    (f_{v})_{w} = \ev_{w}\left(\forestpoly{F}\right)
\]
where $F\in \indexedforests_n$ is the unique forest such that $v=\ForToNC(F)$.
\end{thm}
\begin{proof}
The claim follows from the analogue of the AJS--Billey theorem for double forest polynomials proved in~\cite[\S 8]{BGNST1}. 
In particular,~\cite[Theorem 8.14]{BGNST1} shows that $(f_{v})_{w} = 0$ whenever $v \not\le w$, and ~\cite[Theorem 8.17]{BGNST1} shows that $(f_{v})_{v} = \prod_{(i, j) \in \invnc{v}}( t_{v(j)} - t_{v(i)})$.
\end{proof}

Figure~\ref{fig:flowup_forest} shows one element of the flowup basis for $H^{\bullet}_{T_{n}}(\cay{\NC_{n}})$.

\begin{figure}[!htb]
    \centering
    \includegraphics[scale=0.75]{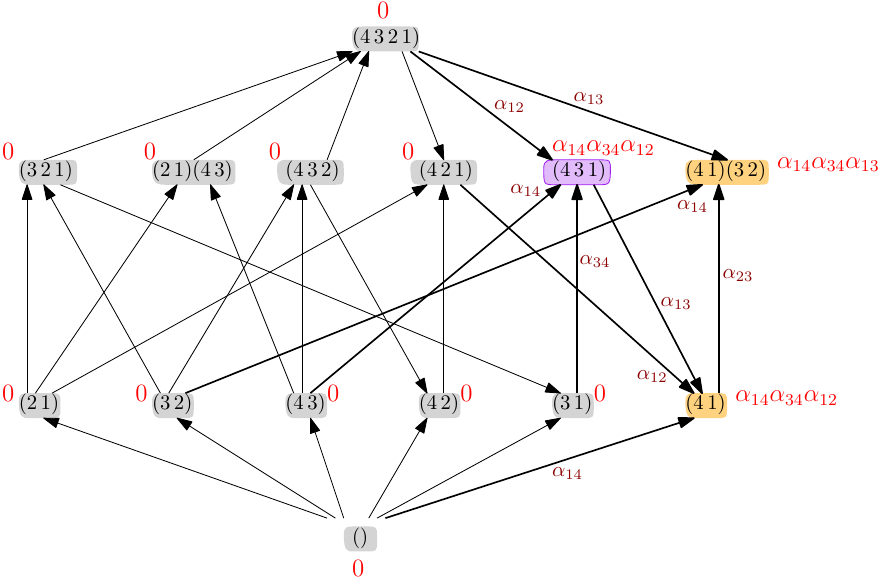}
    \caption{The flowup basis element of $H^{\bullet}_{T_{4}}(\cay{\NC_{4}})$  indexed by $(4\, 3\, 1)$.  
    This is an evaluation of the double forest polynomial for $F = \wh{F}(\rletter{1}^{-}\eletter{1}\eletter{1}\eletter{2})$.  For clarity we have set $\alpha_{ij}\coloneqq t_j-t_i$.}
    \label{fig:flowup_forest}
\end{figure}

\section{The Borel presentation of $H^{\bullet}_{T_{n}}(\qfl_{n})$}
\label{sec:equivbuildingopscohom}

\subsection{Recollections on the equivariant cohomology of $\fl{m}$}
\label{sec:flag_recall}

We begin by briefly reviewing key aspects of the equivariant cohomology of the complete flag variety that are relevant to us. 
The reader is referred to~\cite{AF24} for a more complete exposition.

The following is due to Borel~\cite{Bor53}.  The equivariant cohomology ring $H^{\bullet}_{T_{m}}(\fl{m})$ is generated by the character lattice of $T_{m}$ and the Chern classes $c_1^{T_{m}}(\mc{F}_i/\mc{F}_{i-1})\in H^{2}_{T_{m}}(\fl{m})$.  We therefore have a map
\[
\begin{array}{rcl}
\ZZ[\tl_{m}][\xl_{m}] & \to & H^{\bullet}_{T_{m}}(\fl{m}) \\[0.5ex]
x_{i} & \mapsto & -c_1^{T_m}(\mc{F}_i/\mc{F}_{i-1}) \\
t_{i} & \mapsto & -c_1^{T_m}(\CC_{\chi_{i}}).
\end{array}
\]
The kernel of this map is the ideal $\esymide{m} = \langle f(x_1,\ldots,x_m)-f(t_1,\ldots,t_m):f\in \sym{m}\rangle$, where $\sym{m}$ denotes the ring of symmetric polynomials in $\xl_{m}$.  Thus going forward we identify 
\begin{equation}
\label{eq:borel}
H^{\bullet}_{T_{m}}(\fl{m}) \cong \ZZ[\tl_{m}][\xl_{m}] / \esymide{m}.
\end{equation}

In light of the GKM presentation of $H^{\bullet}_{T_{m}}(\fl{m})$ given in Example~\ref{eg:fl_GKM_ring}, Borel's presentation has the following meaning. 
The inclusion of each torus fixed point $wB \in \fl{m}$ gives a pullback map
\[
\ev_{w}: H^{\bullet}_{T_{m}}(\fl{m}) \to H_{T_{m}}^\bullet(wB) \cong \ZZ[\tl_{m}].
\]
In Borel's presentation, we have $\ev_{w}(x_{i}) = t_{w(i)}$ and $\ev_{w}(t_{i}) = t_{i}$, so we can identify $\ev_{w}$ with the map of the same name defined on polynomials in~\eqref{eq:evaluation_map_definition}. 
In other words, the map
\[
\prod_{w \in S_{m}} \ev_{w} : \ZZ[\tl_{m}][\xl_{m}] \to \ZZ[\tl_{m}]^{\oplus S_{m}}
\]
surjects onto the graph cohomology ring $H^{\bullet}_{T_{m}}(\cay{S_{m}}) \subseteq \ZZ[\tl_{m}]^{\oplus S_{m}}$.  

Now let $X$ be an algebraic variety with an action of $T_{m}$ and $Z \subseteq X$ a $T_{m}$-invariant subvariety.  
For a cohomology class $f \in H^{\bullet}_{T_{m}}(X)$, we denote the  $T_{m}$-equivariant degree of $f$ on $Z$ by
\[
\degint_{Z} f\coloneqq \kappa^Z_*(\iota^*f)=\kappa^X_{\ast} (\mathbbm{1}_{Z} f)  \in \ZZ[\tl_{m}].
\]
where $\kappa^Z$ denotes the unique map $Z \to \point$,  $\mathbbm{1}_{Z} \in H^{\bullet}_{T_{m}}(X)$ is the pushforward of $1\in {H^{\bullet}_{T_{m}}(Z)}$ along the inclusion $\iota: Z \to X$, and the equality $\kappa^Z_*(\iota^*f)=\kappa^X_{\ast} (\mathbbm{1}_{Z} f)$ is the push--pull formula.

Borel's presentation provides a simple way to compute the degree on a Schubert variety using the \emph{divided difference operations} $\partial_{i}: \ZZ[\tl_{m}][\xl_{m}] \to \ZZ[\tl_{m}][\xl_{m}]$ defined by
\[
\partial_{i} f(\xl_{m}; \tl_{m}) = \frac{f(\xl_{m}, \tl_{m}) - f(x_{1}, \ldots, x_{i-1}, x_{i+1}, x_{i}, x_{i+2}, \ldots, x_{m}; \tl_{m})}{x_{i} - x_{i+1}}.
\]
Recall the forgetful map
$\pi_{i} \colon \fl{n}\to \GL_{m}/P_i$ from Section~\ref{sec:pushpull}.  
The following is due to Bernstein--Gelfand--Gelfand~\cite{BGG73} and Demazure~\cite{Dem74}; see \cite[Chapter 10, Lemma 6.5]{AF24} for textbook treatment.

\begin{prop}
\label{prop:BGG}
The map $(\pi_i)^*(\pi_i)_*:H^\bullet_{T_{m}}(\fl{m})\to H^\bullet_{T_{m}}(\fl{m})$ is given by $f\mapsto \partial_i f$.  Moreover, for $w \in S_{m}$ with $w(i) < w(i+1)$, we have
\[
\degint_{X^{ws_{i}}} f = \degint_{X^{w}} \partial_{i} f.
\]
\end{prop}

The \emph{double Schubert polynomials} are the unique family of polynomials $\{\schub{w}(\xl_{m}; \tl_{m}) \suchthat w \in S_{m}\}$ such that each $\schub{w}(\xl_{m}; \tl_{m})$ does not depend on $x_{m}$ and moreover satisfies
\[
\schub{w}(\tl_{m}; \tl_{m}) = \delta_{w, \operatorname{id}_{S_{m}}}
\qquad\text{and}\qquad
\partial_{i} \schub{w}(\xl_{m}; \tl_{m}) = \begin{cases}
\schub{ws_{i}}(\xl_{m}; \tl_{m}) & \text{if $w(i) > w(i+1)$,} \\
0 & \text{otherwise.}
\end{cases}
\]
We therefore have a $T_{m}$-equivariant Kronecker duality between double Schubert polynomials and the homology basis of Schubert cycles $[X^{w}]$ described in Example~\ref{eg:fl_GKM_ring}, as we have
\[
\degint_{X^{w}} \schub{v}(\xl_{m}; \tl_{m}) 
= \delta_{w,w'}.
\]

In the following subsections, we tell a parallel story for $\qfl_{m}$ using equivariantly quasisymmetric polynomials and double forest polynomials.  

\begin{rem}
Borel's presentation for ordinary cohomology is $H^{\bullet}(\fl{m}) \cong \ZZ[\xl_{m}]/\symide{m}$, which can be recovered from the equivariant version by setting $t_{i} \mapsto 0$, as is formalized in Theorem~\ref{thm:GKM_1}.  
\end{rem}

\subsection{Borel's theorem for equivariant quasisymmetry}
\label{sec:qfl_brl}

In order to state our analogue of Borel's theorem, we review several aspects of equivariant quasisymmetry from~\cite{BGNST1}.  
In Section~\ref{sec:geom_degree_1} we will give geometric interpretations which motivate these definitions.  
We define the \emph{equivariant Bergeron--Sottile maps}
\begin{align*}
\rope{i}^-f&=f(x_1,\ldots,x_{i-1},t_i,x_i,x_{i+1},\ldots;\tl)\\
\rope{i}^+f&=f(x_1,\ldots,x_{i-1},x_i,t_i,x_{i+1},\ldots;\tl),
\end{align*}
and the \emph{equivariant quasisymmetric divided difference}
\begin{align*}
\eope{i}f&=\frac{\rope{i}^+f-\rope{i}^-f}{x_i-t_i}=\rope{i}^-\partial_if=\rope{i}^+\partial_if.
\end{align*}
The \emph{double forest polynomials} are then the unique family of polynomials $\{\forestpoly{F}(\xl_n;\tl_n)\suchthat F\in \indexedforests_{n} \}$ such  that each $\forestpoly{F}(\xl_{n};\tl_{n})$ does not depend on $x_{n}$ and moreover satisfies
\[
\forestpoly{F}(\tl_n;\tl_n)=\delta_{F,\emptyset}
\qquad\text{and}\qquad
\eope{i}\forestpoly{F}=\begin{cases}
\forestpoly{G}(\xl_{n};\tl_{[n]/\{i\}}) & \text{if $F = G \cdot \eletter{i}$,}\\
0 & \text{otherwise}\end{cases}
\]
where $\tl_{[n]/\{i\}} = (t_{1}, \ldots, t_{i-1}, t_{i+1}, \ldots, t_{n})$.

\begin{rem}
\label{rem:foreststability}
While it is not immediately obvious from the definition,~\cite[Corollary 4.8]{BGNST1} shows that double forest polynomials have the following stability property: for each $F \in \indexedforests_{n-1}$, we have $\forestpoly{F}(\xl_{n-1}; \tl_{n-1}) = \forestpoly{F \cdot \rletter{n}^{-}}(\xl_{n}; \tl_{n})$ as polynomials.
\end{rem}

The ring of \emph{equivariantly quasisymmetry polynomials} $\eqsym{n}$ is
\[
\eqsym{n}=\{f(\xl_n;\tl_n)\in \ZZ[\tl_n][\xl_n]\suchthat \text{$\rope{i}^-f=\rope{i}^+f$ for $1\le i \le n-1$}\} \subseteq \ZZ[\tl_{n}][\xl_{n}],
\]
which is a subring of $\ZZ[\tl_{n}][\xl_{n}]$.  We also denote $\eqsymide{n} = \langle f(\xl_{n}; \tl_{n}) - f(\tl_{n}; \tl_{n}) \suchthat f \in \eqsym{n} \rangle$.

\begin{thm}
We have 
\[
H^{\bullet}_{T_{n}}(\qfl_{n}) \cong \ZZ[\tl_{n}][\xl_{n}] / \eqsymide{n},
\]
and moreover the double forest polynomials $\forestpoly{F}(\xl_{n}; \tl_{n})$ where $F\in\indexedforests_n$ give a free $\ZZ[\tl_{n}]$-basis for this ring.
\end{thm}

Our proof primarily relies on results from Section~\ref{sec:CoinvariantsGKM}, but we make use of one result which is deferred to Appendix~\ref{appendix} for ease of exposition.

\begin{proof}
By Theorem~\ref{thm:GKM_1}, we know that $H^{\bullet}_{T_{n}}(\qfl_{n})$ is isomorphic to the graph cohomology ring $H^{\bullet}_{T_{n}}(\cay{\NC_{n}}) \subseteq \ZZ[\tl_{n}]^{\NC_{n}}$.  
We prove the theorem by showing that the map
\[
\begin{array}{rcl}
\mathbf{ev}_{\NC} : \ZZ[\tl_{n}][\xl_{n}]  & \to & \ZZ[\tl_{n}]^{\NC_{n}} \\
f(\xl_{n}; \tl_{n}) & \mapsto & \big( \ev_{w}(f) \big)_{w \in \NC_{n}}
\end{array}
\]
induces the desired isomorphism onto the graph cohomology ring.  

The image of $\mathbf{ev}_{\NC}$ is contained in the graph cohomology ring, as for any permutation $w$, $t_{b} - t_{a}$ divides $\ev_{w}(f) - \ev_{(a\,b)w}(f)$.
Moreover, by Theorem~\ref{thm:forest_flowup}, the double forest polynomials map to a free (flowup) basis of $H^{\bullet}_{T_{n}}(\cay{\NC_{n}})$, so $\mathbf{ev}_{\NC}$ is surjective.  

What remains is to show that $\ker(\mathbf{ev}_{\NC}) = \eqsymide{n}$, which is Theorem~\ref{thm:we_forgot_to_prove_this} in Appendix~\ref{appendix}.
\end{proof}

We finally consider ordinary cohomology, proving Theorem~\ref{main:A}.  
As in the introduction write $\qsymide{n}$ for the ideal generated by quasisymmetric polynomials with no constant term.  

\begin{cor}
We have 
\[
H^{\bullet}(\qfl_{n}) \cong \ZZ[\xl_{n}] / \qsymide{n}.
\]
Moreover, the forest polynomials $\forestpoly{F}(\xl_{n}; 0)$ where $F\in\indexedforests_n$ give a free $\ZZ$-basis for this ring.
\end{cor}
\begin{proof}
By Theorem~\ref{thm:GKM_1}, we can obtain $H^{\bullet}(\qfl_{n})$ from  $H^{\bullet}_{T_{n}}(\qfl_{n})$ by performing a change of scalars along the homomorphism $\ZZ[\tl_{n}] \to \ZZ$ given by $t_{i} \mapsto 0$.  
Applying the base change to our Borel presentation, we have a canonical identification between the images of  $\ZZ[\tl_{n}][\xl_{n}]$, $\eqsym{n}$, $\eqsymide{n}$, and $\forestpoly{F}(\xl_{n}; \tl_{n})$ with $\ZZ[\xl_{n}]$, $\qsym{n}$, $\qsymide{n}$, and $\forestpoly{F}(\xl_{n}; 0)$.
\end{proof}

\subsection{Geometric realizations of $\rope{i}^{\pm}$ and $\eope{i}$}
\label{sec:geom_degree_1}

We now show that the equivariant Bergeron--Sottile maps correspond to equivariant cohomology operations that are adjoint to the building maps $\Psi_{i,j}$ and $\mathbb{P}_i$ defined in Section~\ref{sec:BuildingOperators}.

\begin{fact}
\label{fact:biggertorus}
Let $\gamma:T\to T'$ be a coordinate projection between two algebraic tori $T$ and $T'$. 
For $X$ a variety with a $T$-action, we have 
\[
    H^\bullet_{T}(X)=H^\bullet_{T}(\point)\otimes_{H^\bullet_{T'}(\point)} H^\bullet_{T'}(X).
\]
Furthermore, if $\Phi:X\to Y$ is a $T'$-equivariant map of varieties with a $T'$-action, then $\Phi^*:H^\bullet_{T}(Y)\to H^\bullet_{T}(X)$ and $\Phi_*:H^\bullet_{T}(X)\to H^\bullet_{T}(Y)$ are given by extending the corresponding maps on $T'$-equivariant cohomology by the identity map on $H^\bullet_{T}(\point)$.
\end{fact}

As in Definition~\ref{defn:gamma_action} we write $\gamma_{i}: T_{m} \to T_{m-1}$ for the $i$th coordinate projection and $\fl{m-1}^{\gamma_i}$ for $\fl{m-1}$ with the action of $T_{m}$ induced by $\gamma_{i}$.  
With Fact~\ref{fact:biggertorus}, the GKM presentation of the $T_{m-1}$-equivariant cohomology of $\fl{m-1}$ implies that
\[
H^\bullet_{T_m}(\fl{m-1}^{\gamma_i})\cong \{(g_w)_{w\in S_{m-1}}|t_{j+\delta_{j\ge i}}-t_{k+\delta_{k\ge i}}\text{ divides }g_w-g_{(j,k)w},\forall (j\ne k)\}\subset \mathbb{Z}[\tl_m]^{\oplus S_{m-1}}
\]
and the Borel presentation similarly implies
\[
H^\bullet_{T_m}(\fl{m-1}^{\gamma_i})\cong \frac{\ZZ[\tl_m][\xl_{m-1}]}{\langle f(\xl_{m-1})-f(t_1,\ldots,t_{i-1},t_{i+1},\ldots,t_m) \suchthat f\in \sym{m-1}\rangle}.
\]
The isomorphism between these two presentations is given by
$$
f\mapsto \big( f(t_{w(1)+\delta_{w(1)\ge i}},\ldots,t_{w(m-1)+\delta_{w(m-1)\ge i}};\tl_m) \big)_{w\in S_{m-1}}.
$$

\begin{prop}
\label{prop:pullbacks}
The pullback map
$\Psi_{i,j}^*\colon H^\bullet_{T_m}(\fl{m})\to H^\bullet_{T_m}(\fl{m-1}^{\gamma_i})$
sends $t_k\mapsto t_k$ for all $k$, and 
$$x_k\mapsto \begin{cases}x_{k-\delta_{k>j}}&k\ne j\\t_{i}&k=j.\end{cases}$$
\end{prop}
\begin{proof}
We have that $\Psi_{i,j}^*t_k=t_k$ holds since $\Psi_{i,j}$ is a $T_m$-equivariant map, so it suffices to show the result for $x_k$.
To avoid notational overlap we will let $x_1^{\gamma_i},\ldots,x_{m-1}^{\gamma_i}$ denote the $x_i$ generators in $\fl{m-1}^{\gamma_i}$. It suffices to show that for all $w \in S_{m-1}$
\[
\ev_{w}
(\Psi_{i,j}^*x_k) =\begin{cases}\ev_{w} (x^{\gamma_i}_{k-\delta_{k>j}}) &k\ne j\\t_i&k=j.\end{cases}
\]
Note that
\[
\Psi_{i,j}(w)(k)=\begin{cases}
w(k)+\delta_{w(k)\ge i} &k < j\\
t_i&k=j\\
w(k-1)+\delta_{w(k-1)\ge i} &k>j.\end{cases}
\]

For $w\in S_{m-1}$ we have $\ev_{w}(x_\ell^{\gamma_i}) =\gamma_i^* t_{w(\ell)}=t_{w(\ell)+\delta_{w(\ell)\ge i}}$, and therefore for $k\ne j$ we have $$\ev_{w}(x^{\gamma_i}_{k-\delta_{k>j}}) =t_{\Psi_{i,j}(w)(k)}.$$
On the other hand, because $\ev_w$ and $\ev_{\Psi_{i,j}(w)}$ are pullbacks under the inclusions $\{w\}\hookrightarrow \fl{m-1}^{\gamma_i}$ and $\{\Psi_{i,j}(w)\}\hookrightarrow \fl{m}$, we have
\begin{equation*}
    \ev_{w}(\Psi_{i,j}^* x_k) = \ev_{\Psi_{i,j}(w)}(x_k) =t_{\Psi_{i,j}(w)(k)}=\begin{cases}\ev_{w} (x^{\gamma_i}_{k-\delta_{k>j}}) &k\ne j\\t_i&k=j.\end{cases}\qedhere
\end{equation*}
\end{proof}

\begin{prop}
\label{prop:Howbuilderscohomology_notdegree}
    The maps $(\Psi_i^{\pm})^*:H^\bullet_{T_m}(\fl{m})\to H^\bullet_{T_m}(\fl{m-1}^{\gamma_i})$ are given by 
    \begin{align}\label{eq:composition_giving_ri}
        (\Psi_i^{\pm})^*f=\rope{i}^{\pm}f.
    \end{align}
    The map $(\Psi_i^{\pm})^*(\pi_i)^*(\pi_i)_*:H^\bullet_{T_m}(\fl{m})\to H^\bullet_{T_m}(\fl{m-1}^{\gamma_i})$ is given by
    \begin{align}\label{eq:composition_giving_ei}
        (\Psi_i^{\pm})^*(\pi_i)^*(\pi_i)_*f=\eope{i}f.
    \end{align}
\end{prop}
\begin{proof}
    Specializing the computation of $\Psi_{i,j}^*$ in Proposition~\ref{prop:pullbacks} to $j=i,i+1$ gives~\eqref{eq:composition_giving_ri}. 
    Since $(\pi_i)^*(\pi_i)_*$ computes $\partial_i$ by \Cref{prop:BGG}, we obtain~\eqref{eq:composition_giving_ei} from the identity $\eope{i}f=\rope{i}^{\pm}\partial_if$. 
\end{proof}

\begin{thm}
\label{thm:Howbuilderscohomology}
Let $Z\subset \fl{m-1}^{\gamma_i}$ be a $T_{m}$-invariant subvariety. 
Then for $f\in H^\bullet_{T_m}(\fl{m})$, we have equalities of $T_{m}$-equivariant degrees
\begin{equation}\label{eq:degrees_ei_ri}
\degint_{\Psi_i^{\pm}Z}f=\degint_Z \rope{i}^{\pm}f
\qquad\text{and}\qquad
\degint_{\mathbb{P}_i Z}f=\degint_Z\eope{i}f.
\end{equation}
\end{thm}
\begin{proof}
As the $\Psi_i^{\pm}$ are closed embeddings we have $\mathbbm{1}_{\Psi_i^{\pm}Z} = (\Psi_{i}^{\pm})_* \mathbbm{1}_{Z}\in H^\bullet_{T_m}(\fl{m})$. 
The first degree equation now comes from 
$$
\degint_{\Psi_i^{\pm}Z}f
=\degint_{\fl{m}} \left((\Psi_{i}^{\pm})_{\ast} \mathbbm{1}_{Z}\right)f
=\degint_{\fl{m-1}^{\gamma_i}}\mathbbm{1}_{Z}(\Psi_i^{\pm})^{\ast} f 
= \degint_Z\rope{i}^{\pm}f,
$$
where we used the push-pull formula in all three equalities.  
For the second degree equation, because the $\pi_i\Psi_i$ are closed embeddings, we have
\[\mathbbm{1}_{\mathbb{P}_iZ}=
\mathbbm{1}_{\pi_i^{-1}\pi_i\Psi_i Z} =\pi_i^*(\pi_i \Psi_i)_* \mathbbm{1}_{Z} =\pi_i^*(\pi_i)_*(\Psi_i^{-})_* \mathbbm{1}_{Z}\in H^\bullet_{T_m}(\fl{m}).
\]
Therefore using a similar argument as above,
\[
\degint_{\PP_iZ} f=\degint_{\fl{m}}\left(\pi_i^*(\pi_i)_*(\Psi_i^{-})_*\mathbbm{1}_{Z}\right)f=\degint_{\fl{m-1}^{\gamma_i}}\mathbbm{1}_{Z}(\Psi_i^{-})^*(\pi_i)^*(\pi_i)_*f=\degint_Z \eope{i}f,
\]
where the last equality uses~\eqref{eq:composition_giving_ei}. This  establishes the second part of~\eqref{eq:degrees_ei_ri}.
\end{proof}

\subsection{The degree on $X(\wh{F})$ using the $\star$-monoid}

We now describe a combinatorial procedure for computing the degree on the $X(\wh{F})$ varieties using the \emph{$\star$-monoid $\mc{S}$} from~\cite[\S 9.2]{BGNST1}.
For each $A \subseteq [n]$, writing $A=(i_1<\cdots < i_k)$ we define a map
\[
\begin{array}{rccl}
\gamma_A\coloneqq \gamma_{i_1}\cdots \gamma_{i_k}: & T_n & \to& T_{n-|A|},
\end{array}
\]
the coordinate projection away from $i_1,\ldots,i_k$.
To compose these maps, we define an operation on subsets $A,B\subset [n]$ with $|A| + |B| \le n$:
\[
A\star B =\{([n] \setminus B)_i\suchthat i\in A\}\cup B,
\]
where $([n] \setminus B)_i$ denotes the $i$th element of $[n] \setminus B$ in increasing order.  
The following proposition is straightforward and we omit the proof.
\begin{prop}
\label{eq:composition_of_gammas}
For $A, B \subseteq [n]$  with $|A| + |B| \le n$:, $\gamma_{A} \circ \gamma_{B} = \gamma_{A\star B}$.  In particular, $\gamma_i \circ \gamma_A =\gamma_{i \star A}$. 
\end{prop}

We let $\fl{[n]/A}$ denote $\fl{n-|A|}^{\gamma_A}$, and let $X(\wh{G})_{[n]/A}$ denote $X(\wh{G})\subset \fl{n-|A|}^{\gamma_A}$ equipped with the torus action of $T_{n}$ induced by $\gamma_A$. 
Taken in conjunction with~\Cref{fact:biggertorus}, this allows us to transfer our results about the $T_m$-varieties $\fl{m}$ and $\fl{m-1}^{\gamma_i}$ to results about the $T_n$-varieties $\fl{[n]/A}$ and $\fl{[n]/(i\star A)}$.

For $A \subseteq [n]$, we denote $t_{i, A} = t_{([n] \setminus A)_{i}}$ and $\tl_{[n]/A} = (t_{1, A}, t_{2, A}, \ldots, t_{n-|A|, A})$. 
We have isomorphisms
\[
H^\bullet_{T_n}(\fl{[n]/A})
\cong
\frac{\ZZ[\tl_n][\xl_{n-|A|}]}{\langle f(\xl_{n-|A|})-f(\tl_{[n]/A}):f\in \sym{n-|A|}\rangle}
\]
and a GKM presentation
\[
H^\bullet_{T_n}(\fl{[n]/A})\cong \left\{(g_w)_{w\in S_{n-|A|}}
\suchthat \text{$t_{j, A}-t_{k, A}$ divides $g_w-g_{(j,k)w}$ for all $(j\ne k)$}\right\}\subset \mathbb{Z}[\tl_n]^{\oplus S_{n-|A|}}
\]
with the isomorphism between the Borel presentation and the GKM presentation given by 
\[
f\mapsto (f(t_{w(1),A},\ldots,t_{w(n-|A|),A};t_1,\ldots,t_n))_{w\in S_{n - |A|}}.
\]

For $f\in \ZZ[\tl_n][\xl_n]$ define
\begin{align*}
\rope{i,A}^-f(\xl_n;\tl_n)&=f(x_1,\ldots,x_{i-1},t_{i,A},x_i,x_{i+1},\ldots,x_{n-1};\tl_n)\\
\rope{i,A}^+f(\xl_n;\tl_n)&=f(x_1,\ldots,x_{i-1},x_{i},t_{i,A},x_{i+1},\ldots,x_{n-1};\tl_n)\\
\eope{i, A}f(\xl_n; \tl_n)&= \frac{\rope{i,A}^+f(\xl_n; \tl_n) -\rope{i,A}^-f(\xl_n; \tl_n) }{x_i-t_{i,A}}.
\end{align*}

\begin{defn} 
For $A \subseteq [n]$ and $\rt \in \reseq_{n - |A|}$, let
\[
[\Phi_{\rt}]_{A} = \begin{cases}
\operatorname{id} & \text{if $\rt=\emptyset$} \\
[\Phi_{\rt'}]_{i\star A} \circ \rope{i, A}^{\pm} & \text{if $\rt=\rt' \cdot \rletter{i}^{\pm}$} \\ 
[\Phi_{\rt'}]_{i \star A} \circ \eope{i, A} & \text{if $\rt=\rt' \cdot \eletter{i}$.} \\ 
\end{cases}
\]
For $\rt \in \reseq_{n}$ we write $[\Phi_{\rt}] = [\Phi_{\rt}]_{\emptyset}$.
\end{defn}

As was shown in~\cite[Theorem 10.5]{BGNST1}, the operation $[\Phi_{\rt}]$ only depends on the colored Tamari equivalence class of the bicolored nested forest $\wh{F}(\rt)$ associated to $\rt$, we we can write $[\Phi_{\wh{F}}]$ for $\wh{F} \in \bnfor_{n}$ without ambiguity.

\begin{thm}
\label{thm:degonXfhat}
For $\wh{F}\in \bnfor_n$ and $f \in H^{\bullet}_{T_{n}}(\fl{n})$, we have
\[
\degint_{X(\wh{F})}f=[\Phi_{\wh{F}}]f.
\]
In particular, the double forest polynomials $\forestpoly{F}(\xl_{n}; \tl_{n})$ are Kronecker dual to the homology basis $[X(F)]$ of $H^{\bullet}_{T_{n}}(\qfl_{n})$ given in~\Cref{cor:qfl_homology_basis}.  
\end{thm}
\begin{proof}
By~\Cref{thm:Howbuilderscohomology}, we have that for $A\subset [n]$ and $\wh{F} = \wh{G} \cdot \xletter{i} \in \bnfor_{n-|A|}$, we have 
\[
\degint_{X(\wh{F})_{[n]/A}} f
= \begin{cases}
\degint_{X(\wh{G})_{[n]/(i\star A)}}\rope{i,A}^{\pm}f & \text{if $\xletter{i} = \rletter{i}^{\pm}$,} \\
\degint_{X(\wh{G})_{[n]/(i\star A)}}\eope{i,A}f & \text{if $\xletter{i} = \eletter{i}$}
\end{cases}
\]
after which the theorem follows recursively from the definition of $[\Phi_{\wh{F}}]$; see~\cite[\S 10]{BGNST1} for more details on applying these operators to double forest polynomials.   
\end{proof}

\begin{eg}
Let $\Omega=\rletter{1}^{-}\rletter{1}^{+}\eletter{2}\eletter{1}\rletter{2}^+\eletter{3}$.
Then
\[
    [\Phi_{\Omega}]=\rope{1,\{1,2,3,4,5\}}^-\,\rope{1,\{1,2,3,5\}}^+\,\eope{2,\{1,2,3\}}\,\eope{1,\{2,3\}}\,\rope{2,\{3\}}^+\,\eope{3}.
\]
\end{eg}

A polynomial $a(\tl_{n})\in \ZZ[\tl_{n}]$ is called \emph{Graham-positive} if it lies in $\ZZ_{\geq 0}[t_2-t_1,\dots, t_{n} - t_{n-1}]$.
As shown by Graham \cite{Gra01}, for any $T$-invariant subvariety $X\subset \fl{n}$, the decomposition 
\[
[X]=\sum a_w(\tl_{n})\,[X^w]
\]
into Schubert cycles has Graham-positive coefficients $a_w(\tl_{n})=\int_X\schub{w}(\xl_n;\tl_n)$. The Graham-positivity of $[\Phi_{\wh{F}}]\schub{w}(\xl_n;\tl_n)$ was shown through purely combinatorial means in ~\cite[Theorem 11.4]{BGNST1}, which we can now interpret geometrically.

\begin{cor}[{~\cite[Theorem 11.4]{BGNST1}}]
\label{cor:schub_to_for}
For $F \in \indexedforests_{n}$, the coefficient
\[
a_w(\tl_n)=\degint_{X(\wh{F})}\schub{w}(\xl_n;\tl_n)=[\Phi_{\wh{F}}]\schub{w}(\xl_n;\tl_n)
\]
is Graham positive.
\end{cor}

\begin{rem}
In~\cite[Theorem 11.6]{BGNST1} we also show that the product of two double forest polynomials has a Graham positive expansion as a sum of double forest polynomials for combinatorial reasons.  
While the analogous result for double Schubert polynomials has a geometric explanation due to Graham~\cite{Gra01}, we do not have a geometric explanation for this positivity. We also show in~\cite[Theorem 11.4]{BGNST1} that $[\Phi_{\wh{F}}]\forestpoly{G}(\xl_n;\tl_n)$ is Graham positive, and we also do not have a geometric explanation for this positivity.
\end{rem}

\appendix
\section{Double forest polynomials}
\label{appendix}

The purpose of this appendix is to give an alternative description for the ideal $\eqsymide{n}$ as the kernel of the map
\[
\mathbf{ev}_{\NC} = \prod_{w \in \NC_{n}} \ev_{w} : \ZZ[\tl_{n}][\xl_{n}] \to \ZZ[\tl_{n}]^{\oplus \NC_{n}}.
\]

\begin{thm}
\label{thm:we_forgot_to_prove_this}
We have $\eqsymide{n} = \ker(\mathbf{ev}_{\NC})$ and as a consequence
\begin{align}
\label{eqn:freebasis}
\ZZ[\tl_n][\xl_n]= \eqsymide{n} \oplus \bigoplus_{F \in \indexedforests_{n}}\ZZ[\tl_n]\forestpoly{F}(\xl_n;\tl_n).
\end{align}
\end{thm}

The proof appears at the end of the appendix.  

\begin{rem}
Theorem~\ref{thm:we_forgot_to_prove_this} belongs to a family of results known as ``Orbit Harmonics.''  
In~\cite{BeGa23}, the first two authors show that $\QQ \otimes_{\ZZ} \qscoinv{n}$ is the associated graded of the coordinate ring for the set of noncrossing partitions, considered as points using one-line notation.  
Specializing $t_{i} \mapsto i$ in Theorem~\ref{thm:we_forgot_to_prove_this}, we recover this result and find a new cohomological interpretation for it.
\end{rem}

We begin by extending the definition of forest polynomials to a basis of the full polynomial ring $\ZZ[\tl_{n}][\xl_{n}]$.  
First define
\[
\indexedforests = \bigsqcup_{m \ge 0} \indexedforests_{m} \Big/ \{\text{$F \sim G$ if $F \in \bnfor_{m}$ and $G = F \cdot (\rletter{m+1}^{-})^{k}$}\},
\]
so that we identify two forests if they differ only by some number of trailing isolated leaves.  
There is an obvious bijection between the internal nodes of any two forests identified in this manner, so we can speak of the internal nodes of $F \in \indexedforests$ without ambiguity.  
Say that an internal node $v$ of  $F \in \indexedforests$ is \emph{terminal} if both of its children are leaves, and let
\[
\LT(F) = \{i \suchthat \text{$F$ has a terminal node with children $i$ and $i+1$} \},
\]
so that $i \in \LT(F)$ if and only if $F = G \cdot \eletter{i}$ for some $G \in \indexedforests$.  
We then define
\[
\LTFor_{n} = \{F \in \indexedforests \suchthat  \LT(F) \subseteq [n]\}.
\]

We now consider double forest polynomials indexed by $\LTFor_{n}$.  Recall that as described in Remark~\ref{rem:foreststability}, we have $\forestpoly{F}(\xl_{m}; \tl_{m}) = \forestpoly{F \cdot (\rletter{m+1}^{-})^{k}}(\xl_{m+k}; \tl_{m+k})$ for all $m, k \ge 0$ and $F \in \indexedforests_{m}$.  
Thus each class $F \in \LTFor_{n}$ defines a unique forest polynomial in any set of $x$ and $t$ variables $(\xl_{m}; \tl_{m})$ such that $m \ge \max \supp{F}$.  

\begin{defn}
For $F \in \LTFor_{n}$, the \emph{$n$-truncated double forest polynomial} is defined to be
\[
\forestpoly{F}(\xl_{n}; \tl_{n}) \coloneq \forestpoly{F}(\xl_{n}, 0, \ldots, 0; \tl_{n}, 0, \ldots, 0),
\]
where the right-hand side is the specialization of the unique forest polynomial defined by $F$.
\end{defn}

The truncated forest polynomials have the property that 
\[
\forestpoly{F}(\tl_{n};\tl_{n})=\delta_{F,\emptyset}
\qquad\text{and}\qquad
\eope{i}\forestpoly{F}=\begin{cases}
\forestpoly{G}(\xl_{n};\tl_{[n]/\{i\}}, 0) & \text{if $F = G \cdot \eletter{i}$,}\\
0 & \text{otherwise,}
\end{cases}
\]
for $1 \le i \le n$, where $\eope{1}, \ldots, \eope{n-1}$ are as defined in Section~\ref{sec:qfl_brl} and 
\[
\eope{n}f(\xl_{n}; \tl_{n}) = \frac{f(\xl_{n}; \tl_{n}) - f(x_{1}, \ldots, x_{n-1}, t_{n}; \tl_{n})}{x_{n} - t_{n}}.
\]

In~\cite[Corollary 4.7 (2)]{BGNST1} we show that 
\begin{equation}
\label{eq:forestbasis}
\ZZ[\tl_{n}][\xl_{n}] = \bigoplus_{ F\in \LTFor_{n}} \ZZ[\tl_{n}] \forestpoly{F}(\xl_n;\tl_n).
\end{equation}
Thus as a consequence, we obtain a $\ZZ$-basis of $\ZZ[\xl_{n}]$ consisting of (single) forest polynomials
\[
\forestpoly{F}(\xl_n) \coloneq \forestpoly{F}(\xl_n; 0).
\]
We prove in~\cite[Corollary 4.6]{BGNST1} that these are the same forest polynomials studied  in~\cite{NST_a,NT_forest}.

\begin{thm}[{\cite[Theorem 9.7]{NST_a}}, {\cite[Theorem 3.7]{NT_forest}}]
\label{thm:singleforestideal}
The forest polynomials $\forestpoly{F}(\xl_n;0)$ with $F \in \LTFor_{n}$ and $n \in \LT(F)$ are a $\ZZ$-basis for $\qsymide{n}$.
\end{thm}

We also define the set of \emph{zigzag forests} to be
\[
\zigzag{n} = \{F \in \LTFor_{n} \suchthat \LT(F) = \{n\}\}.
\]
The $\forestpoly{F}(\xl_{n}; \tl_{n})$ for $F \in \zigzag{n}$ are called \emph{double fundamental quasisymmetric polynomials}, and in~\cite[\S 4 and \S 7]{BGNST1} we show that they form a basis for $\eqsym{n}$.  
Via~\cite{NST_a},~\cite{BGNST1} also show that the $\forestpoly{F}(\xl_{n}; 0)$ are the classical fundamental quasisymmetric basis for $\qsym{n}$.  

\begin{proof}[Proof of Theorem~\ref{thm:we_forgot_to_prove_this}]

Theorem~\ref{thm:forest_flowup} shows that 
\[
\ZZ[\tl_n][\xl_n]= \ker(\mathbf{ev}_{\NC}) \oplus \bigoplus_{F \in \indexedforests_{n}}\ZZ[\tl_n]\forestpoly{F}(\xl_n;\tl_n),
\]
so we only need to show that $\eqsymide{n} = \ker(\mathbf{ev}_{\NC})$.  By \eqref{eq:forestbasis}, it suffices to show the inclusions
\begin{equation}
\label{eq:appendix_inclusions}
\bigoplus_{\substack{\emptyset \neq F \in \LTFor_{n}\setminus \indexedforests_n}} \ZZ[\tl_{n}]\forestpoly{F}(\xl_{n}; \tl_{n})\subseteq \eqsymide{n} \subseteq \ker(\mathbf{ev}_{\NC}).
\end{equation}

For the second inclusion in Equation~\eqref{eq:appendix_inclusions} we use the fact, proved in~\cite[Theorem 7.1]{BGNST1}, that if $f(\xl_n;\tl_n)\in \eqsym{n}$ then $\ev_{w}f=\ev_{\idem}f$ for all $w \in \NC_n$, so clearly 
\[
\eqsymide{n} = \{f(\xl_n;\tl_n) - f(\tl_n;\tl_n) \suchthat f \in \eqsym{n}\} \subseteq \ker(\mathbf{ev}_{\NC}).
\] 

We now establish the first inclusion by showing that for all $F \in \LTFor_{n}\setminus \indexedforests_n$ we have $\forestpoly{F}(\xl_n;\tl_n) \in \eqsymide{n}$. 
We proceed by induction on $|F|$.  
Our base case is $|F|=1$, wherein the assumption $n \in \LT(T)$ implies that $F = (\rletter{1}^{-})^{n-1}\cdot \eletter{n}$, so that $\forestpoly{F}(\xl_n;\tl_n) \in \eqsym{n}$ as $F\in \zigzag{n}$ (alternatively as $\forestpoly{F}(\xl_n;\tl_n)=x_1+\cdots+x_n-t_1-\cdots-t_n)$.

Now assume that $|F| > 1$.  By \cite[Theorem 9.7]{NST_a}, the (single) forest polynomial $\forestpoly{F}(\xl_n)$ lies in $\qsymide{n}$, which is generated by the fundamental quasisymmetric polynomials $\forestpoly{G}(\xl_n)$ for $\emptyset\ne G \in \zigzag{n}$.  
One may then write 
\[
\forestpoly{F}(\xl_n)=\sum_{\emptyset\ne G\in \zigzag{n}}  f_G(\xl_n)\,\forestpoly{G}(\xl_n).
\]
As the double fundamental quasisymmetric polynomial $\forestpoly{G}(\xl_n;\tl_n)$ lies in $\eqsymide{n}$ for $\emptyset\neq G\in \zigzag{n}$, the difference
\begin{align}
\label{eq:the_difference}
\forestpoly{F}(\xl_n;\tl_n)-\sum_{\emptyset\neq G\in \zigzag{n}} f_G(\xl_n) \forestpoly{G}(\xl_n;\tl_n)
\end{align}
can be written as a $\ZZ[\tl_{n}]$-linear combination of double forest polynomials $\forestpoly{H}(\xl_n; \tl_{n})$ with $H\in \LTFor_{n}\setminus \indexedforests_n$.  
Furthermore, the difference~\eqref{eq:the_difference} contains no monomials consisting entirely of $x$-variables, so each $H$ must have $|H| < |F|$.  
By induction, we have now expressed $\forestpoly{F}(\xl_n;\tl_n)$ as an element of $\eqsymide{n}$, completing the proof.  
\end{proof}
\bibliographystyle{hplain}
\bibliography{qfl.bib}

\begin{thebibliography}{10}

\bibitem{AlNa09}
M.~Albenque and P.~Nadeau.
\newblock Growth function for a class of monoids.
\newblock In {\em 21st {I}nternational {C}onference on {F}ormal {P}ower
  {S}eries and {A}lgebraic {C}ombinatorics ({FPSAC} 2009)}, Discrete Math.
  Theor. Comput. Sci. Proc., AK, pages 25--38. Assoc. Discrete Math. Theor.
  Comput. Sci., Nancy, 2009.

\bibitem{AJS}
H.~H. Andersen, J.~C. Jantzen, and W.~Soergel.
\newblock Representations of quantum groups at a {$p$}th root of unity and of
  semisimple groups in characteristic {$p$}: independence of {$p$}.
\newblock {\em Ast\'{e}risque}, (220):321, 1994.

\bibitem{AF24}
D.~Anderson and W.~Fulton.
\newblock {\em Equivariant cohomology in algebraic geometry}, volume 210 of
  {\em Cambridge Studies in Advanced Mathematics}.
\newblock Cambridge University Press, Cambridge, 2024.

\bibitem{ABB04}
J.-C. Aval, F.~Bergeron, and N.~Bergeron.
\newblock Ideals of quasi-symmetric functions and super-covariant polynomials
  for {$S_n$}.
\newblock {\em Adv. Math.}, 181(2):353--367, 2004.

\bibitem{BeGa23}
N.~Bergeron and L.~Gagnon.
\newblock The excedance quotient of the {B}ruhat order, quasisymmetric
  varieties, and {T}emperley--{L}ieb algebras.
\newblock {\em J. Lond. Math. Soc. (2)}, 110(4):Paper No. e13007, 2024.

\bibitem{BGNST1}
N.~Bergeron, L.~Gagnon, P.~Nadeau, H.~Spink, and V.~Tewari.
\newblock Equivariant quasisymmetry and noncrossing partitions, 2025,
  2504.15234.

\bibitem{BS98}
N.~Bergeron and F.~Sottile.
\newblock Schubert polynomials, the {B}ruhat order, and the geometry of flag
  manifolds.
\newblock {\em Duke Math. J.}, 95(2):373--423, 1998.

\bibitem{BGG73}
I.~N. Bern\v{s}te\u{\i}n, I.~M. Gelfand, and S.~I. Gelfand.
\newblock Schubert cells, and the cohomology of the spaces {$G/P$}.
\newblock {\em Uspehi Mat. Nauk}, 28(3(171)):3--26, 1973.

\bibitem{Bi97}
P.~Biane.
\newblock Some properties of crossings and partitions.
\newblock {\em Discrete Math.}, 175(1-3):41--53, 1997.

\bibitem{BJV}
P.~Biane and M.~Josuat-Verg\`es.
\newblock Noncrossing partitions, {B}ruhat order and the cluster complex.
\newblock {\em Ann. Inst. Fourier (Grenoble)}, 69(5):2241--2289, 2019.

\bibitem{BB05}
S.~Billey and A.~Postnikov.
\newblock Smoothness of {S}chubert varieties via patterns in root subsystems.
\newblock {\em Adv. in Appl. Math.}, 34(3):447--466, 2005.

\bibitem{Bil99}
S.~C. Billey.
\newblock Kostant polynomials and the cohomology ring for {$G/B$}.
\newblock {\em Duke Math. J.}, 96(1):205--224, 1999.

\bibitem{BB03}
S.~C. Billey and T.~Braden.
\newblock Lower bounds for {K}azhdan-{L}usztig polynomials from patterns.
\newblock {\em Transform. Groups}, 8(4):321--332, 2003.

\bibitem{BjBr05}
A.~Bj\"{o}rner and F.~Brenti.
\newblock {\em Combinatorics of {C}oxeter groups}, volume 231 of {\em Graduate
  Texts in Mathematics}.
\newblock Springer, New York, 2005.

\bibitem{Bor53}
A.~Borel.
\newblock Sur la cohomologie des espaces fibr\'{e}s principaux et des espaces
  homog\`enes de groupes de {L}ie compacts.
\newblock {\em Ann. of Math. (2)}, 57:115--207, 1953.

\bibitem{BGW}
A.~V. Borovik, I.~M. Gelfand, and N.~White.
\newblock {\em Coxeter matroids}, volume 216 of {\em Progress in Mathematics}.
\newblock Birkh\"{a}user Boston, Inc., Boston, MA, 2003.

\bibitem{ChSk74}
T.~Chang and T.~Skjelbred.
\newblock The topological {S}chur lemma and related results.
\newblock {\em Ann. of Math. (2)}, 100:307--321, 1974.

\bibitem{ChapotonNadeau2017}
F.~Chapoton and P.~Nadeau.
\newblock Combinatorics of the categories of noncrossing partitions.
\newblock {\em S\'em. Lothar. Combin.}, 78B:Art. 37, 12, 2017.

\bibitem{Dem74}
M.~Demazure.
\newblock D\'{e}singularisation des vari\'{e}t\'{e}s de {S}chubert
  g\'{e}n\'{e}ralis\'{e}es.
\newblock {\em Ann. Sci. \'{E}cole Norm. Sup. (4)}, 7:53--88, 1974.

\bibitem{MR2308029}
M.~Franz and V.~Puppe.
\newblock Exact cohomology sequences with integral coefficients for torus
  actions.
\newblock {\em Transform. Groups}, 12(1):65--76, 2007.

\bibitem{Fulton_Intersection}
W.~Fulton.
\newblock {\em Intersection theory}, volume~2 of {\em Ergebnisse der Mathematik
  und ihrer Grenzgebiete. 3. Folge. A Series of Modern Surveys in Mathematics
  [Results in Mathematics and Related Areas. 3rd Series. A Series of Modern
  Surveys in Mathematics]}.
\newblock Springer-Verlag, Berlin, second edition, 1998.

\bibitem{GGMS87}
I.~M. Gelfand, R.~M. Goresky, R.~D. MacPherson, and V.~V. Serganova.
\newblock Combinatorial geometries, convex polyhedra, and {S}chubert cells.
\newblock {\em Adv. in Math.}, 63(3):301--316, 1987.

\bibitem{GGP97}
I.M. Gelfand, M.I. Graev, and A.~Postnikov.
\newblock Combinatorics of hypergeometric functions associated with positive
  roots.
\newblock In {\em The {A}rnold-{G}elfand mathematical seminars}, pages
  205--221. Birkh\"{a}user Boston, Boston, MA, 1997.

\bibitem{Ges84}
I.~M. Gessel.
\newblock Multipartite {$P$}-partitions and inner products of skew {S}chur
  functions.
\newblock In {\em Combinatorics and algebra ({B}oulder, {C}olo., 1983)},
  volume~34 of {\em Contemp. Math.}, pages 289--317. Amer. Math. Soc.,
  Providence, RI, 1984.

\bibitem{GKM98}
M.~Goresky, R.~Kottwitz, and R.~MacPherson.
\newblock Equivariant cohomology, {K}oszul duality, and the localization
  theorem.
\newblock {\em Invent. Math.}, 131(1):25--83, 1998.

\bibitem{Gra01}
W.~Graham.
\newblock Positivity in equivariant {S}chubert calculus.
\newblock {\em Duke Math. J.}, 109(3):599--614, 2001.

\bibitem{GZ03}
V.~Guillemin and C.~Zara.
\newblock The existence of generating families for the cohomology ring of a
  graph.
\newblock {\em Adv. Math.}, 174(1):115--153, 2003.

\bibitem{MR2166181}
M.~Harada, A.~Henriques, and T.~S. Holm.
\newblock Computation of generalized equivariant cohomologies of {K}ac-{M}oody
  flag varieties.
\newblock {\em Adv. Math.}, 197(1):198--221, 2005.

\bibitem{HHMP}
M.~Harada, T.~Horiguchi, M.~Masuda, and S.~Park.
\newblock The volume polynomial of regular semisimple {H}essenberg varieties
  and the {G}elfand-{Z}etlin polytope.
\newblock {\em Proc. Steklov Inst. Math.}, 305:318--344, 2019.

\bibitem{HS18}
J.~Heller and P.~Schwer.
\newblock Generalized non-crossing partitions and buildings.
\newblock {\em Electron. J. Combin.}, 25(1):Paper No. 1.24, 29 pp, 2018.

\bibitem{JosuatNadeau2023Koszulity}
M.~Josuat-Verg\`es and P.~Nadeau.
\newblock Koszulity of dual braid monoid algebras via cluster complexes.
\newblock {\em Ann. Math. Blaise Pascal}, 30(2):141--188, 2023.

\bibitem{Kre72}
G.~Kreweras.
\newblock Sur les partitions non crois\'{e}es d'un cycle.
\newblock {\em Discrete Math.}, 1(4):333--350, 1972.

\bibitem{LP20}
T.~Lam and A.~Postnikov.
\newblock Polypositroids.
\newblock {\em Forum Math. Sigma}, 12:Paper No. e42, 67 pp, 2024.

\bibitem{LS82}
A.~Lascoux and M.-P. Sch\"{u}tzenberger.
\newblock Polyn\^{o}mes de {S}chubert.
\newblock {\em C. R. Acad. Sci. Paris S\'{e}r. I Math.}, 294(13):447--450,
  1982.

\bibitem{leemasudapark2024torusorbitclosuresflag}
E.~Lee, M.~Masuda, and S.~Park.
\newblock Torus orbit closures in the flag variety, 2024, 2203.16750.

\bibitem{lian2023hhmp}
C.~Lian.
\newblock The {HHMP} decomposition of the permutohedron and degenerations of
  torus orbits in flag varieties.
\newblock {\em Int. Math. Res. Not. IMRN}, (20):13380--13399, 2024.

\bibitem{MP08}
M.~Masuda and T.~E. Panov.
\newblock Semi-free circle actions, {B}ott towers, and quasitoric manifolds.
\newblock {\em Mat. Sb.}, 199(8):95--122, 2008.

\bibitem{NST_c}
P.~Nadeau, H.~Spink, and V.~Tewari.
\newblock The geometry of quasisymmetric coinvariants, 2024, 2410.12643.

\bibitem{NST_a}
P.~Nadeau, H.~Spink, and V.~Tewari.
\newblock Quasisymmetric divided differences, 2024, 2406.01510.

\bibitem{NTremixed}
P.~Nadeau and V.~Tewari.
\newblock Remixed {E}ulerian numbers.
\newblock {\em Forum Math. Sigma}, 11:Paper No. e65, 26 pp, 2023.

\bibitem{NT_forest}
P.~Nadeau and V.~Tewari.
\newblock Forest polynomials and the class of the permutahedral variety.
\newblock {\em Adv. Math.}, 453:Paper No. 109834, 2024.

\bibitem{oeis}
N.~J.~A. Sloane.
\newblock The {E}ncyclopedia of {I}nteger {S}equences.

\bibitem{StThesis}
R.~P. Stanley.
\newblock {\em Ordered structures and partitions}.
\newblock Memoirs of the American Mathematical Society, No. 119. American
  Mathematical Society, Providence, R.I., 1972.

\bibitem{TW15}
E.~Tsukerman and L.~Williams.
\newblock Bruhat interval polytopes.
\newblock {\em Adv. Math.}, 285:766--810, 2015.

\bibitem{Tym08}
J.~S. Tymoczko.
\newblock Permutation actions on equivariant cohomology of flag varieties.
\newblock In {\em Toric topology}, volume 460 of {\em Contemp. Math.}, pages
  365--384. Amer. Math. Soc., Providence, RI, 2008.

\end{thebibliography}
\end{document}